\documentclass{amsart}
\usepackage{amssymb,amsmath}
\usepackage{graphicx}
\usepackage{enumerate}
\usepackage{ifthen}

\DeclareMathOperator{\tr}{tr}
\DeclareMathOperator{\sgn}{sgn}
\DeclareMathOperator{\diag}{diag}

\def\limr{{\lim_{\rho\rightarrow 0}\,}}

\def\R{{\mathbb{R}}}
\def\N{{\mathbb{N}}}

\def\Z{{\mathbb{Z}}}

\def\theta{{\vartheta}}
\def\phi{{\varphi}}
\def\epsilon{{\varepsilon}}
\long\def\umbruch{{\displaybreak[1]}}

\def\eg{e.\,g.\ }
\def\ie{i.\,e.\ }

\def\fracp#1#2{{\frac{\partial #1}{\partial #2}}}
\newcommand{\A}[1]{\ifthenelse{#1 = 2}{\lvert A\rvert^{#1}}{\tr A^{#1}}}

\mathchardef\ordinarycolon\mathcode`\:
  \mathcode`\:=\string"8000
  \begingroup \catcode`\:=\active
    \gdef:{\mathrel{\mathop\ordinarycolon}}
  \endgroup

\newtheorem{theorem}{Theorem}[section]
\newtheorem{lemma}[theorem]{Lemma}

\newtheorem{corollary}[theorem]{Corollary}

\theoremstyle{definition}
\newtheorem{definition}[theorem]{Definition}
\newtheorem{example}[theorem]{Example}

\newtheorem{remark}[theorem]{Remark}

\numberwithin{equation}{section}
\begin{document}
\title[Maximum-Principle Functions]{When maximum-principle functions cease to exist}

\author{Martin Franzen}
\address{Martin Franzen, Universit\"at Konstanz,
   Universit\"atsstrasse 10, 78464 Konstanz, Germany}
\curraddr{}
\def\ukaddress{@uni-konstanz.de}
\email{Martin.Franzen\ukaddress}
\thanks{We would like to thank F. Kuhl, B. Lambert, M. Langford, M. Makowski, M. Raum, J. Scheuer,
\thanks{O. Schn\"urer, M. Schweighofer, B. Stekeler, and S. Wenzel for discussions and support.}
\thanks{The author is a member of the DFG priority program SPP 1489.}}


\date{January 28, 2014.}



\begin{abstract}
  We consider geometric flow equations for 
  contracting and expanding normal velocities,
  including 
  powers of the Gauss curvature, $K$, of the mean curvature, $H$, and 
  of the norm of the second fundamental form, $|A|$, and 
  ask whether - after appropriate rescaling - closed strictly convex surfaces 
  converge to spheres.
  To prove this, many authors use certain functions
  of the principal curvatures,
  which we call maximum-principle functions.
  We show when such functions cease to exist and exist,
  while presenting newly discovered maximum-principle functions.
\end{abstract}

\maketitle


\setcounter{tocdepth}{2}
\tableofcontents 

\section{Overview}

We consider the geometric flow equations
\begin{align}\label{eq:flow}
  \frac{d}{dt} X = -F\nu.
\end{align}
and ask whether closed strictly convex $2$-dimensional surfaces $M_t$ in $\R^3$
converge to round points or to spheres at $\infty$.

The answer is affirmative for many normal velocities $F$, including certain powers of 
the Gauss curvature, $K$, the mean curvature, $H$,
and the norm of the second fundamental form, $|A|$.

Here, authors like B. Andrews \cite{ba:gauss,ac:surfaces}, O. Schn\"urer \cite{os:surfacesA2,os:surfaces}, and F. Schulze \cite{fs:convexity},
use functions of the principal curvatures $w$
to show convergence to a round point or to spheres at $\infty$.

In \cite{os:surfaces}, \cite{os:surfacesA2}, 
O. Schn\"urer proposes a characterization of
these functions.
And in this paper, we extend it to non-rational functions.
Our definition of \textit{maximum-principle functions} (MPF)
now covers any such function $w$ that is known to far.

The function from \cite{fs:convexity}
\begin{align*}
  w=\frac{(a-b)^2(a+b)^{2\sigma}}{(a\,b)^2}
\end{align*}
is an example of a MPF for a normal velocity $F=H^\sigma$ with $1<\sigma\leq 5$. 

Our main question is when MPF cease to exist
for the large set of contracting and expanding 
normal velocities $F^\sigma_\xi$.
We are particularly interested in normal velocities
$F^\sigma_0=\sgn(\sigma)\cdot K^{\sigma/2}$,
$F^\sigma_1=\sgn(\sigma)\cdot H^\sigma$, and 
$F^\sigma_2=\sgn(\sigma)\cdot |A|^\sigma$,
for all powers $\sigma\in\R\setminus\{0\}$. 

Here, we either are able to prove the non-existence of MPF, or give an example 
of a MPF. We present MPF from literature, and
newly discovered MPF.
Small gaps only exist for $H^\sigma$ with $5.17<\sigma<5.98$, and $|A|^\sigma$ with $8.15<\sigma<9.89$. 

We summarize our main results in the table for $F^\sigma_\xi$ (p. \pageref{thm:necessary}), and in the tables
for $F^\sigma_0$ (p. \pageref{cor:gauss curvature}), for $F^\sigma_1$ (p. \pageref{cor:mean curvature}),
for $F^\sigma_2$ (p. \pageref{cor:|A|}), and for $F^\sigma_\sigma$ (p. \pageref{cor:trA}). \\

\textit{The paper is structured as follows:} 

In the chapter on notation, we 
give a brief introduction to differential geometric quantities
like induced metric, second fundamental form, and principal curvatures. 

The next chapter is on $F^\sigma_\xi$.
This is the set of normal velocities for which we investigate 
the existence and non-existence of MPF. 

We proceed with a chapter on the MPF, 
motivating the definition and taking a closer
look at the linear operator $L$ and the $\alpha$-conditions. 

Chapter 5 contains main Theorem \ref{thm:necessary}.
Using Euler's Theorem on homogeneous functions
we are also able to prove the non-existence of MPF 
for any normal velocity with $0<\hom F\leq 1$. 

In chapter 6 we present vanishing functions, 
which are sometimes MPF, 
depending on the given normal velocity $F^\sigma_\xi$.
Furthermore, we present a few technical lemmas on
vanishing functions, which play an important role 
in the following chapters
on $F^\sigma_0$, $F^\sigma_1$, $F^\sigma_2$, and $F^\sigma_\sigma$. 

We conclude our paper with an outlook suggesting
an improved MPF Ansatz.

\section{Notation}

For a brief introduction of the standard notation we adopt the 
a corresponding chapter from \cite{os:surfacesA2}. 

We use $X=X(x,\,t)$ to denote the embedding vector of a $2$-manifold $M_t$ into $\R^3$ and 
$\frac{d}{dt} X=\dot{X}$ for its total time derivative. 
It is convenient to identify $M_t$ and its embedding in $\R^3$.
The normal velocity $F$ is a homogeneous symmetric function of the principal curvatures.
We choose $\nu$ to be the outer unit normal vector to $M_t$. 
The embedding induces a metric $g_{ij} := \langle X_{,i},\, X_{,j} \rangle$ and
the second fundamental form $h_{ij} := -\langle X_{,ij},\,\nu \rangle$ for all $i,\,j = 1,\,2$. 
We write indices preceded by commas to indicate differentiation with respect to space components, 
\eg $X_{,k} = \frac{\partial X}{\partial x_k}$ for all $k=1,\,2$.

We use the Einstein summation notation. 
When an index variable appears twice in a single term it implies summation of that term over all the values of the index.

Indices are raised and lowered with respect to the metric
or its inverse $\left(g^{ij}\right)$, 
\eg $h_{ij} h^{ij} = h_{ij} g^{ik} h_{kl} g^{lj} = h^k_j h^j_k$.

The principal curvatures $a,\,b$ are the eigenvalues of the second fundamental
form $\left(h_{ij}\right)$ with respect to the induced metric $\left(g_{ij}\right)$. A surface is called strictly convex,
if all principal curvatures are strictly positive.
We will assume this throughout the paper.
Therefore, we may define the inverse of the second fundamental
form denoted by $(\tilde h^{ij})$.

Symmetric functions of the principal
curvatures are well-defined, we will use the Gauss curvature
$K=\frac{\det h_{ij}}{\det g_{ij}} = a\cdot b$, the mean curvature
$H=g^{ij} h_{ij} = a+b$, the square of the norm of the second fundamental form
$|A|^2= h^{ij} h_{ij} = a^2+b^2$, and the trace of powers of the second fundamental form 
$\tr A^{\sigma} = \tr \left(h^i_j\right)^{\sigma} = a^\sigma+b^\sigma$. We write indices preceded by semi-colons 
to indicate covariant differentiation with respect to the induced metric,
e.\,g.\ $h_{ij;\,k} = h_{ij,k} - \Gamma^l_{ik} h_{lj} - \Gamma^l_{jk} h_{il}$, 
where $\Gamma^k_{ij} = \frac{1}{2} g^{kl} \left(g_{il,j} + g_{jl,i} - g_{ij,l}\right)$.
It is often convenient to choose normal coordinates, \ie coordinate systems such
that at a point the metric tensor equals the Kronecker delta, $g_{ij}=\delta_{ij}$,
in which $\left(h_{ij}\right)$ is diagonal, $(h_{ij})=\diag(a,\,b)$. 
Whenever we use this notation, we will also assume that we have 
fixed such a coordinate system. We will only use a Euclidean metric
for $\R^3$ so that the indices of $h_{ij;\,k}$ commute according to
the Codazzi-Mainardi equations.

A normal velocity $F$ can be considered as a function of $(a,\,b)$
or $(h_{ij},\,g_{ij})$. We set $F^{ij}=\fracp{F}{h_{ij}}$,
$F^{ij,\,kl}=\fracp{^2F}{h_{ij}\partial h_{kl}}$. 
Note that in coordinate
systems with diagonal $h_{ij}$ and $g_{ij}=\delta_{ij}$ as mentioned
above, $F^{ij}$ is diagonal.

\section{Contracting and expanding normal velocities $F^\sigma_\xi$}

In this chapter, we specify what we mean by a contracting and an expanding normal velocity
and define the important quantity $\beta=\frac{F_a}{F_b}$.
In Remark \ref{rem:normal velocity xi}, 
we introduce the set of normal velocities $F^\sigma_\xi$.
This set includes
powers of the Gauss curvature, $F^\sigma_0$, 
powers of the mean curvature, $F^\sigma_1$,
powers of the norm of the second fundamental form, $F^\sigma_2$,
and the trace of powers of the second fundamental form, $F^\sigma_\sigma$.
Throughout this paper, our goal is to determine,
when MPF exist and cease to exist for $F^\sigma_\xi$.

\begin{definition}[Normal velocity $F$]\label{def:normal velocity}
  Let $a$ and $b$ be principal curvatures.
  Let $F(a,b)\in C^2\left(\R^2_{+}\right)$ be a symmetric homogeneous function of degree $\sigma\in\R\setminus\{0\}$.
  In this paper, we call $F$ a \textit{normal velocity} if
  \begin{align*}
    F_a,\;F_b > 0
  \end{align*}
  for all $0<a,b$.
  Furthermore, we call $F$ \textit{contracting} if
  $F>0$ for all $0<a,b$,
  and we call $F$ \textit{expanding} if
  $F<0$ for all $0<a,b$.
\end{definition}  
  
\begin{definition}[Quantity $\beta$]\label{def:beta}
  Let $F$ be a normal velocity.
  We define the quantity $\beta$ as
  \begin{align}
    \beta = \frac{F_a}{F_b}
  \end{align}
  for all $0<a,b$. 
  We later choose $(a,b)=(\rho,1)$ and write $\beta_F(\rho)$.
\end{definition} 

\begin{remark}[Normal velocity $F^\sigma_\xi$]\label{rem:normal velocity xi}
  In this paper, we investigate the normal velocity $F^\sigma_\xi$.
  We define $F^\sigma_\xi$ as
  \begin{eqnarray}\label{def:fxi}
    F^\sigma_\xi(a,b) =
    \begin{cases}
      \sgn(\sigma)\cdot\left(a^\xi+b^\xi\right)^{\sigma/\xi}, & \text{if } \xi\neq 0, \\
      \sgn(\sigma)\cdot(a\,b)^{\sigma/2}, & \text{if } \xi=0,
      \end{cases}
  \end{eqnarray}
  for all $\sigma\in\R\setminus\{0\}$.
  Calculating $\big(F^\sigma_\xi\big)_a$, $\big(F^\sigma_\xi\big)_b$, we obtain
  \begin{eqnarray}\label{def:fxia}
      \left(F^\sigma_\xi\right)_a =
      \begin{cases}
        |\sigma|\cdot\left(a^\xi+b^\xi\right)^{\sigma/\xi-1}a^{\xi-1}, & \text{if } \xi\neq 0, \\
        |\sigma|\cdot(a\,b)^{\sigma/2-1}\frac{b}{2}, & \text{if } \xi=0,
        \end{cases}
    \end{eqnarray}
  and
  \begin{eqnarray}\label{def:fxib}
      \left(F^\sigma_\xi\right)_b =
      \begin{cases}
        |\sigma|\cdot\left(a^\xi+b^\xi\right)^{\sigma/\xi-1}b^{\xi-1}, & \text{if } \xi\neq 0, \\
        |\sigma|\cdot(a\,b)^{\sigma/2-1}\frac{a}{2}, & \text{if } \xi=0,
        \end{cases}
    \end{eqnarray}
  for all $0<a,b$. Therefore,
  \begin{align*}
    \left(F^\sigma_\xi\right)_a,\,\left(F^\sigma_\xi\right)_b>0,
  \end{align*}
  for all $0<a,b$. Hence, $F^\sigma_\xi$ is a normal velocity for all $\xi\in\R$.
  $F^\sigma_\xi$ is a \textit{contracting} normal velocity, if $\sigma>0$, and
  $F^\sigma_\xi$ is an \textit{expanding} normal velocity, if $\sigma<0$.
  Calculating the quantity $\beta(\rho)$ we obtain
  \begin{align*}
    \beta_{F^\sigma_\xi}(\rho) = \rho^{\xi-1},
  \end{align*}
  for all $\xi\in\R$. Note that $\beta_{F^\sigma_\xi}(\rho)$ is independent of $\sigma$.
\end{remark}

\begin{example}[Gauss curvature]
  Let $\xi=0$. Then we have $\beta_{F^\sigma_0}(\rho)=\rho^{-1}$, and
  \begin{align*}
    F^\sigma_0=&\,\sgn(\sigma)\cdot(a\,b)^{\sigma/2} \\
       =&\,\sgn(\sigma)\cdot K^{\sigma/2}.
  \end{align*}
\end{example}

\begin{example}[Mean curvature]
  Let $\xi=1$. Then we have $\beta_{F^\sigma_1}(\rho)=1$, and
  \begin{align*}
    F^\sigma_1=&\,\sgn(\sigma)\cdot(a+b)^\sigma \\
       =&\,\sgn(\sigma)\cdot H^\sigma.
  \end{align*}
\end{example}

\begin{example}[Norm of the second fundamental form]
  Let $\xi=2$. Then we have $\beta_{F^\sigma_2}(\rho)=\rho$, and
  \begin{align*}
    F^\sigma_2=&\,\sgn(\sigma)\cdot(a^2+b^2)^{\sigma/2} \\
       =&\,\sgn(\sigma)\cdot |A|^\sigma.
  \end{align*}
\end{example}

\begin{example}[Trace of the second fundamental form]
  Let $\xi=\sigma$. Then we have $\beta_{F^\sigma_\sigma}(\rho)=\rho^{\sigma-1}$, and
  \begin{align*}
    F^\sigma_\sigma=&\,\sgn(\sigma)\cdot(a^\sigma+b^\sigma)^{\sigma/\sigma} \\
       =&\,\sgn(\sigma)\cdot \tr A^\sigma.
  \end{align*}
\end{example}

\section{Definition and motivation of rational MPF, and MPF}

In the context of geometric evolution equations \eqref{eq:flow},
\begin{align*}
  \frac{d}{dt} X = -F\nu,
\end{align*}
the question arises, if after appropriate rescaling, 
closed strictly convex surfaces $M_t$ converge to spheres.
This is also referred to as convergence to a round point
or convergence to a sphere at $\infty$.

The answer is affirmative for many normal velocities $F$,
including the Gauss curvature flow, $F^2_0=K$, B. Andrews \cite{ba:gauss},
the mean curvature flow, $F^1_1=H$, G. Huisken \cite{gh:flow}, and
the inverse Gauss curvature flow, $F^{-1}_0=-\frac{1}{K}$, O. Schn\"urer \cite{os:surfaces}.

For many normal velocities $F$ proofs rely on the fact that a certain geometrically meaningful quantity $w$
is monotone, \ie $\max_{M_t} w$ is non-increasing in time.

In \cite{os:surfacesA2} and \cite{os:surfaces}, O. Schn\"urer proposes criteria 
for selecting such monotone quantities for contracting flows and for expanding flows, respectively.
To date, to the author's knowledge, all known quantities which fulfill these criteria
can be used to show convergence to a round point or to a sphere at $\infty$.
Here, we decided to work with these criteria as a definition.
Their monotonicity is proven using the maximum-principle.
This is why we name these quantities \textit{rational maximum-principle functions} (RMPF).

\begin{definition}[RMPF]\label{def:rmpf}
  Let $a$ and $b$ be principal curvatures.
  Let $w(a,b)\in C^2(\R^2_+)$ be a symmetric homogeneous function of degree $\chi\in\R\setminus\{0\}$. 
  We call $w$ a \textit{rational maximum-principle function} for a normal velocity $F$ if
  \begin{enumerate}
    \item\label{rI}
      \begin{enumerate}
        \item $w>0$ for all $0<a,b$, $a\neq b$,
        \item $w=0$ for all $0<a=b$.
      \end{enumerate}
    \item\label{rII}
      \begin{enumerate}
        \item $\chi>0$ if $F$ is contracting, 
        \item $\chi<0$ if $F$ is expanding.
      \end{enumerate}
    \item\label{rIII}
      \begin{enumerate}
        \item $w_a<0$ for all $0<a<b$,
        \item $w_a>0$ for all $0<b<a$.
      \end{enumerate}
    \item\label{rIV}
      Let $Lw:=\frac{d}{dt} w-F^{ij}w_{;\,ij}$ be the linear operator 
      corresponding to the geometric flow equation \eqref{eq:flow}.      
      We achieve $Lw\leq 0$ for all $0<a,b$ by assuming that
      \begin{enumerate}
        \item terms without derivatives of $\left(h_{ij}\right)$ are non-positive for all $0<a,b$, and
        \item terms involving derivatives of $\left(h_{ij}\right)$, at a critical point of $w$, 
        \textit{i.e.} $w_{;i}=0$ for all $i=1,\,2$, are non-positive for all $0<a,b$.
      \end{enumerate}
    \item\label{rV}
      $w(a,b)$ is a rational function.
   \end{enumerate}
\end{definition}

The RMPF conditions \eqref{rI} through \eqref{rV} as in \cite{os:surfacesA2}, \cite{os:surfaces} are motivated as follows:

For all geometric flow equations \eqref{eq:flow} we assume that spheres stay spherical. 
For contracting (expanding) normal velocities they contract to a point (expand to infinity).
So we can only aim to find monotone quantities, if $w(a,a)=0$ for all $a>0$, or if $\chi\leq 0$ (if $\chi\geq 0$).
If $\chi\leq 0$ (if $\chi\geq 0$), we obtain that $w$ is non-increasing on any self-similarly contracting (expanding) surface.
So this does not imply convergence to a round point.
RMPF condition \eqref{rIII} ensures that the quantity decreases if the principal curvatures approach each other.
By RMPF condition \eqref{rIV}, we check that we can apply the maximum principle to prove monotonicity.
The first found monotone quantities are all rational functions, \eg $w=(a-b)^2$ for $F^2_0=K$ in \cite{ba:gauss}, or
$w=\frac{(a-b)^2(a+b)}{(a\,b)}$ for $F^2_2=|A|^2$ in \cite{os:surfacesA2}.
This motivates RMPF condition \eqref{rV}.

\begin{definition}[MPF]\label{def:mpf}
  Let $a$ and $b$ be principal curvatures.
  Let $w(a,b)\in C^2(\R^2_+)$ be a symmetric homogeneous function of degree $\chi\in\R\setminus\{0\}$. 
  We call $w$ a \textit{maximum-principle function} for a normal velocity $F$ if
  \begin{enumerate}
    \item\label{I}
      \begin{enumerate}
        \item $w>0$ for all $0<a,b$, $a\neq b$,
        \item $w=0$ for all $0<a=b$.
      \end{enumerate}
    \item\label{II}
      \begin{enumerate}
        \item $\chi>0$ if $F$ is contracting, 
        \item $\chi<0$ if $F$ is expanding.
      \end{enumerate}
    \item\label{III}
      \begin{enumerate}
        \item $w_a<0$ for all $0<a<b$,
        \item $w_a>0$ for all $0<b<a$.
      \end{enumerate}
    \item\label{IV}
      We assume that the \textit{constant terms} $C_w(a,b)$ and the \textit{gradient terms} $E_w(a,b)$ and $G_w(a,b)$ are non-positive for all $0<a,b$.
        \begin{align*}
          C_w(a,b):=&\, w_a\,a \left(\left(F_a\,a^2+F_b\,b^2\right) + \left(F-F_a\,a-F_b\,b\right)a\right) \\
                                 &\quad + w_b\,b \left(\left(F_a\,a^2+F_b\,b^2\right) + \left(F-F_a\,a-F_b\,b\right)b\right), \\                
          E_w(a,b):=&\, w_a \left(F_{aa} + 2\,F_{ab}\,\alpha + F_{bb}\,\alpha^2\right) \\
                                   &\quad -F_a \left(w_{aa} + 2\,w_{ab}\,\alpha + w_{bb}\,\alpha^2\right) \\
                                   &\quad +2\,\frac{w_b\,F_a-w_a\,F_b}{a-b}\,\alpha^2, \\
          G_w(a,b):=&\, \frac{1}{\alpha^2}\cdot\bigg(w_b \left(F_{aa} + 2\,F_{ab}\,\alpha + F_{bb}\,\alpha^2\right)\bigg. \\
                                 &\qquad\qquad -F_b \left(w_{aa} + 2\,w_{ab}\,\alpha + w_{bb}\,\alpha^2\right) \\
                                 &\qquad\qquad \bigg. +2\, \frac{w_b\,F_a - w_a\,F_b}{a-b}\bigg).
        \end{align*}
    \item\label{V}
    We define
      \begin{align*}
        \alpha = -\frac{w_a}{w_b},\;
        \alpha_a = \frac{\partial\alpha}{\partial a}.
      \end{align*}
    Let $c>0$, $d\in\R$ be some constants.
    Set $(a,b)=(\rho,1)$, where $0<\rho<1$. 
    If $F$ is contracting, we assume for $\rho\rightarrow 0$
    \begin{align*}
      \alpha =&\, c, \qquad \qquad \quad \; \; \, \, \alpha_a = 0, \\
      \text{or} \qquad \alpha =&\, \frac{c}{\rho}+o\left(\rho^{-1}\right), \quad 
      \alpha_a = -\frac{c}{\rho^2}+o\left(\rho^{-1}\right).
    \end{align*}
    If $F$ is expanding, we assume for $\rho\rightarrow 0$
    \begin{align*}
      \alpha =&\, \frac{c}{\rho}+o\left(\rho^{-1}\right), \qquad \quad \; \,
      \alpha_a = -\frac{c}{\rho^2}+o\left(\rho^{-1}\right), \\
      \text{or} \qquad \alpha =&\, \frac{c+d\,\rho}{\rho^2}+o\left(\rho^{-1}\right), \quad
      \alpha_a = -\frac{2\,c+d\,\rho}{\rho^2}+o\left(\rho^{-1}\right).
    \end{align*}
   \end{enumerate}
\end{definition}

F. Schulze and O. Schn\"urer present in \cite{fs:convexity}  
one of the first non-rational monotone quantities. They use
\begin{align*}
  w=\frac{(a-b)^2(a+b)^{2\sigma}}{(a\,b)^2}
\end{align*}
to show convergence to a round point for $F^\sigma_1=H^\sigma$, for all $1<\sigma\leq 5$.
Here, $w$ is not rational, if $2\sigma\notin\N$.
Therefore, in MPF Definition \ref{def:mpf},
we extend RMPF Definition \ref{def:rmpf} to not necessarily rational functions.
We call them \textit{maximum-principle functions} (MPF).
In Lemma \ref{lem:rmpf}, we prove that MPF condition \eqref{V} holds for any RMPF
in the case of contracting flows $F^\sigma_\xi$, if $\xi>0$, or $\xi=0$ and $1<\sigma\leq 2$. 
For other $F^\sigma_\xi$,
a similar Lemma is in progress. 
A first step in the this direction is Lemma \ref{lem:alpha o estimates}.
To date, all known RMPF fulfill MPF condition \eqref{V}. 

In this paper, the main question is, when MPF cease to exist.
As it turns out, 
this is more an algebra than a differential geometry question. 
Correspondingly, in MPF condition \eqref{IV}
we formulate three inequality conditions. 
In Lemmas \ref{lem:1stDerivatives} through \ref{lem:LwCriticalPoint}, we show that these three inequality conditions
are equivalent to RMPF condition \eqref{rIV}. 
In condensed form, MPF Definition \ref{def:mpf} is a more algebraic formulation of RMPF Definition \ref{def:rmpf},
which is extended to what seems to be the proper class of non-rational functions.

\subsection{Linear operator L}\label{sc:terms}
In Lemmas \ref{lem:1stDerivatives} through \ref{lem:LwCriticalPoint}, we show that the three inequality conditions 
$C_w(a,b),\,E_w(a,b),\,G_w(a,b)\leq 0$
in condition MPF \eqref{IV} are equivalent
to RMPF condition \eqref{rIV} as follows:

In Lemma \ref{lem:1stDerivatives}, we show two helpful identities for the first derivatives 
with respect to the induced metric $g_{ij}$, and for the derivatives with respect to the second fundamental form $h_{ij}$.
Next, we cite the evolution equations for $g_{ij}$ and $h_{ij}$ from O. Schn\"urer \cite{os:alpbach}.
We need both Lemmas in Lemma \ref{lem:Lw}. 

Here, we compute the linear operator $Lw$, where $w$ is a function of the principal curvatures $a$ and $b$.
Sometimes, we also denote them by $\lambda_1$ and $\lambda_2$.
In Lemma \ref{lem:2ndDerivatives}, we cite another helpful identity for the second derivatives
with respect to the second fundamental form. 
Finally, in Lemma \ref{lem:LwCriticalPoint} we compute the constant terms $C_w(a,b)$, and the two gradient terms $E_w(a,b)$ and $G_w(a,b)$.
This concludes the proof of our claim that condition \eqref{IV} of RMPF Definition \ref{def:rmpf} and of MPF Definition \ref{def:mpf}
are equivalent.

\begin{lemma}[First derivatives]\label{lem:1stDerivatives}
  Let $f$ be a normal velocity $F$ or a function $w$ of the principal curvatures $\lambda_1$ and $\lambda_2$.
  Then we have
  \begin{align}
    f^{kl} =&\, f^k_j g^{lj}, \label{eq:fhkl} \\
    \frac{\partial f}{\partial g_{kl}} =&\, -f^{il}h^k_i. \label{eq:fgkl}
  \end{align}
  The matrices $\left(f^{kl}\right)$ and $\left(\frac{\partial f}{\partial g_{kl}}\right)$ are symmetric.
\end{lemma}
\begin{proof}
  We consider $f = f\left(h^j_i\big((h_{kl}),(g_{kl})\big)\right)$.
  The matrices $\left(h_{ij}\right)$ and $\left(g_{ij}\right)$ are symmetric. 
  So we have $h_{ij} = \frac{1}{2}\left(h_{ij} + h_{ji}\right)$ and $g^{ij} = \frac{1}{2}\left(g^{ij} + g^{ji}\right)$. \\
  Now, we differentiate $f$ with respect to $h_{kl}$.
  \begin{align*}
    f^{kl} =&\, \frac{\partial f}{\partial h_{kl}} \umbruch \\
             =&\,\frac{\partial f}{\partial h^j_i} \frac{\partial h^j_i}{h_{kl}} \umbruch \\
             =&\, f^i_j\, \frac{\partial}{\partial h_{kl}} \frac{1}{2}\left(h_{im}+h_{mi}\right)\frac{1}{2}\left(g^{mj}+g^{jm}\right) \umbruch \\
             =&\, f^j_i\, \frac{1}{4} \frac{\partial}{\partial h_{kl}} \left(h_{im} g^{mj} + h_{im} g^{jm}+h_{mi} g^{mj}+h_{mi} g^{jm}\right) \umbruch \\
             =&\, f^j_i\, \frac{1}{4} \left(\delta^k_i\delta^l_m g^{mj} + \delta^k_i\delta^l_m g^{jm} + \delta^k_m\delta^l_i g^{mj} + \delta^k_m \delta^l_i g^{jm}\right) \umbruch \\
             =&\, f^i_j\, \frac{1}{4} \left(2\,\delta^k_i g^{jl} + 2\,\delta^l_i g^{jk}\right) \umbruch \\
             =&\, \frac{1}{2} \left(f^k_j g^{jl} + f^l_j g^{jk}\right). 
  \end{align*}
  Since $f^{kl} - f^{lk} = 0$ for all $1 \leq k,\, l \leq n$, the matrix $\left(f^{kl}\right)$ is symmetric and formula \eqref{eq:fhkl} follows. \\
  Next, we differentiate $f$ with respect to $g_{kl}$.
  \begin{align*}
    \frac{\partial f}{\partial g_{kl}} 
    =&\, \frac{\partial f}{\partial h^j_i} \frac{\partial}{\partial g_{kl}} \frac{1}{2} 
         \left(h_{im} + h_{mi}\right) \frac{1}{2} \left(g^{mj} + g^{jm}\right) \umbruch\\
    =&\, f^i_j\,\frac{\partial}{\partial g_{kl}}\,\frac{1}{4} \left(h_{im} g^{mj} + h_{im} g^{jm} + h_{mi} g^{mj} + h_{mi} g^{jm}\right) \umbruch\\
    =&\, -f^i_j\,\frac{1}{4}\,\left(h_{im} g^{mk} g^{jl} + h_{im} g^{jk} g^{ml} + h_{mi} g^{mk} g^{jl} + h_{mi} g^{jk} g^{ml}\right) \\
     &\, \left(\textit{use } \frac{\partial g^{ij}}{\partial g_{kl}} = -g^{ik} g^{jl}\right) \umbruch\\
    =&\, -f^i_j\,\frac{1}{2}\,\left(h^k_i g^{jl} + h^l_i g^{jk}\right)
  \end{align*}
  Since $\frac{\partial f}{\partial g_{kl}} - \frac{\partial f}{\partial g_{lk}} = 0$ for all $1 \leq k,\,l \leq n$ the matrix $\left(\frac{\partial f}{\partial g_{kl}}\right)$ is symmetric and formula \eqref{eq:fgkl} follows.
\end{proof}

\begin{lemma}[Evolution equations]\label{lem:evolution}
  Let $X$ be a solution to a geometric flow equation \eqref{eq:flow}.
  Then we have the following evolution equations:
  \begin{align}
    \frac{d}{dt}g^{ij} =&\, 2 F h^{ij}, \label{eq:gij} \\
    Lh_{ij} =&\,F^{kl} h^m_k h_{lm} \cdot h_{ij} - \left(F - F^{kl} h_{kl}\right) h^m_i h_{jm} + F^{kl,rs} h_{kl;i} h_{rs;j}, \label{eq:hij}
  \end{align}
  where $L$ is the linear operator of RMPF Definition \ref{def:rmpf}.
\end{lemma}
\begin{proof}
  We refer to O. Schn\"urer \cite{os:alpbach}.
\end{proof}

\begin{lemma}[Linear operator]\label{lem:Lw}
  Let $w=w\big(h^j_i\big)$ be a function of the principal curvatures.
  Then we have 
  \begin{equation}
    \label{eq:whji} 
    \begin{split}
      Lw=&\, w^{ij} \left(h_{ij} F^{kl} h^m_k h_{lm} + h^m_i h_{jm} \left(F-F^{kl}h_{kl}\right) \right) \\
          &\quad + \left(w^{ij} F^{kl,rs} - F^{ij} w^{kl,rs} \right) h_{kl;i} h_{rs;j}, 
    \end{split}
  \end{equation}
  where $L$ is the linear operator of RMPF Definition \ref{def:rmpf}.
\end{lemma}
\begin{proof}
  We consider $w = w\left(h^j_i\big((h_{kl}),(g_{kl})\big)\right)$.
  \begin{align*}
    Lw =&\, \frac{d}{dt} w - F^{rs} w_{;rs} \umbruch\\
       =&\, \frac{\partial w}{\partial h^j_i} \frac{d}{dt} \left(h_{im} g^{mj}\right) - F^{rs}\bigg(\frac{\partial w}{\partial h^j_i} \left(h_{im} g^{mj}\right)_{;r}\bigg)_{;s} \umbruch\\
       =&\, w^i_j \left(\dot{h}_{im} g^{mj} + h_{im} \dot{g}^{mj}\right) - F^{rs} \left(w^i_j\left(h_{im;r} g^{mj} \right)\right)_{;s} \umbruch\\
       =&\, 2\,F\,w^i_j h^{mj}h_{im} + w^i_j g^{mj}\dot{h}_{im} - F^{rs} \left(w^i_j g^{mj} h_{im;r}\right)_{;s} \umbruch\\
        &\, \left(\textit{use evolution equation \eqref{eq:gij}}\right) \umbruch\\
       =&\, 2\,F\,w^{ij} h^m_j h_{im} + w^{im} \dot{h}_{im} - F^{rs} \left(w^{im} h_{im;r}\right)_{;s} \umbruch\\
        &\, \left(\textit{use formula \eqref{eq:fhkl}: } w^{kl} = w^k_j g^{lj} \right) \umbruch\\
       =&\, 2\,F\,w^{ij} h^m_j h_{im} + w^{im} \dot{h}_{im} - F^{rs} \frac{\partial w^{im}}{\partial h^l_k} \left(h_{kn} g^{nl}\right)_{;s} h_{im;r} 
           -F^{rs} w^{im} h_{im;rs} \umbruch\\
       =&\, 2\,F\,w^{ij} h^m_j h_{im} + w^{im} \dot{h}_{im} - F^{rs} \left(w^{im}\right)^k_l g^{nl} h_{kn;s} h_{im;r} - F^{rs} w^{im} h_{im;rs} \umbruch\\
       &\, \left(\textit{use formula \eqref{eq:fhkl}: } w^{kl} = w^k_j g^{lj} \right) \umbruch\\
       =&\, 2\,F\,w^{ij} h^m_j h_{im} + w^{im}\left(\dot{h}_{im} - F^{rs} h_{im;rs}\right) - F^{rs} w^{im,kn} h_{im;r} h_{kn;s} \umbruch\\
       =&\, 2\,F\,w^{ij} h^m_i h_{jm} + w^{ij}\left(\dot{h}_{ij} - F^{kl} h_{ij,kl}\right) - F^{ij} w^{kl,rs} h_{kl;i} h_{rs;j} \umbruch\\
        &\, \left(\textit{rename indices}\right) \umbruch\\
       =&\, w^{ij}\left(F^{kl} h^m_k h_{lm}\cdot h_{ij} + \left(F-F^{kl}h_{kl}\right)h^m_i h_{jm} + F^{kl,rs} h_{kl;i} h_{rs;j}\right) \\
        &\quad - F^{ij} w^{kl,rs} h_{kl;i} h_{rs;j} \umbruch\\
        &\, \left(\textit{use evolution equation \eqref{eq:hij}}\right) \umbruch\\
       =&\, w^{ij} \left(h_{ij} F^{kl} h^m_k h_{lm} + h^m_i h_{jm} \left(F - F^{kl} h_{kl}\right)\right) \\
        &\quad +\left(w^{ij} F^{kl,rs} - F^{ij} w^{kl,rs}\right) h_{kl;i} h_{rs;j}.\qedhere
  \end{align*}
\end{proof}

\begin{lemma}[Second derivatives]\label{lem:2ndDerivatives}
  Let $f$ be a normal velocity $F$ or a function $w$ of the principal curvatures $\lambda_1$ and $\lambda_2$.
  Then we have
    \begin{align}\label{eq:fijkl}
      f^{ij,kl} \eta_{ij} \eta_{kl} 
      = \sum_{i,\,j} \frac{\partial^2 f}{\partial \lambda_i \partial \lambda_j} \eta_{ii} \eta_{jj} 
      + \sum_{i\neq j} \frac{\frac{\partial f}{\partial \lambda_i}-\frac{\partial f}{\partial \lambda_j}}{\lambda_i - \lambda_j} \eta_{ij}^2
    \end{align}
    for any symmetric matrix $\left(\eta_{ij}\right)$ and $\lambda_1 \neq \lambda_2$, or $\lambda_1 = \lambda_2$ and the last term is interpreted as a limit.
\end{lemma}
\begin{proof}
  We refer to C. Gerhardt \cite{cg:curvature}.
\end{proof}

\begin{lemma}[Linear operator at a critical point]\label{lem:LwCriticalPoint}
  Let $w=w(h^j_i)$ be a symmetric function of the principal curvatures $a$ and $b$.
  At a critical point of $w$, \ie $w_{;i}=0$ for all $i=1,2$, 
  we choose normal coordinates,
  \ie $g_{ij} = \delta_{ij}$ and $\left(h_{ij}\right) = \diag\left(a,\,b\right)$. \\
  Then we have 
  \begin{align*}
    Lw = C_w\left(a,\,b\right) + E_w\left(a,\,b\right) h_{11;1}^2 + G_w\left(a,\,b\right) h_{22;2}^2,
  \end{align*}
  where $L$ is the linear operator of RMPF Definition \ref{def:rmpf}, and
  $\alpha$ is the quantity of MPF Definition \ref{def:mpf}. \\
  The constant terms are 
  \begin{align*}
    C_w\left(a,\,b\right) =&\, w_a\,a \left(\left(F_a\,a^2+F_b\,b^2\right) + \left(F-F_a\,a-F_b\,b\right)a\right) \\
                           &\quad + w_b\,b \left(\left(F_a\,a^2+F_b\,b^2\right) + \left(F-F_a\,a-F_b\,b\right)b\right) \\
                          =&\,F\left(w_a\,a^2+w_b\,b^2\right) \\
                           &\quad + F_a\,w_b\,a\,b\left(a-b\right) \\
                           &\quad -F_b\,w_a\,a\,b\left(a-b\right),
  \end{align*}
  and the two gradient terms are
  \begin{align*}
    E_w(a,b):=&\, w_a \left(F_{aa} + 2\,F_{ab}\,\alpha + F_{bb}\,\alpha^2\right) \\
              &\quad -F_a \left(w_{aa} + 2\,w_{ab}\,\alpha + w_{bb}\,\alpha^2\right) \\
              &\quad +2\,\frac{w_b\,F_a-w_a\,F_b}{a-b}\,\alpha^2, \umbruch \\
    G_w(a,b):=&\, \frac{1}{\alpha^2}\cdot\bigg(w_b \left(F_{aa} + 2\,F_{ab}\,\alpha + F_{bb}\,\alpha^2\right)\bigg. \\
              &\qquad\qquad -F_b \left(w_{aa} + 2\,w_{ab}\,\alpha + w_{bb}\,\alpha^2\right) \\
              &\qquad\qquad \bigg. +2\, \frac{w_b\,F_a - w_a\,F_b}{a-b}\bigg).
  \end{align*}
  We obtain $G_w\left(a,b\right) = E_w\left(b,a\right)$ for all $0<a,b$.
\end{lemma}
\begin{proof}
  We consider $w = w\big(h^j_i\big)$.
  At a critical point of $w$, we have
  \begin{align*}
    w_{;k} = \frac{\partial w}{\partial h^j_i} \big(h^j_i\big)_{;k} 
           = w^i_j \left(h_{il} g^{lj}\right)_{;k} 
           = w^i_j g^{lj} h_{il;k}
           = w^{ij} h_{ij;k}
           = 0.
  \end{align*}
  Here, we also choose normal coordinates and get
  $w_{;1} = w_a\,h_{11;1} + w_b\,h_{22;1} = 0$, implying
  \begin{align}
    h_{22;1} = -\frac{w_a}{w_b}\,h_{11;1} = \alpha\cdot h_{11;1}, \label{eq:h111}
  \end{align}
  and $w_{;2} = w_a\,h_{11;2} + w_b\,h_{22;2} = 0$, implying
  \begin{align}
    h_{11;2} = -\frac{w_b}{w_a}\,h_{22;2} = \frac{1}{\alpha}\cdot h_{22;2}. \label{eq:h222}
  \end{align}
  Now, we compute the linear operator $Lw$ of Lemma \ref{lem:Lw} 
  at a critical point of $w$, where we choose normal coordinates.
  \begin{align*}
    Lw =&\, w^{ij} \left(h_{ij}\,F^{kl} h^m_k h_{lm} + h^m_i h_{jm} \left(F - F^{kl} h_{kl}\right)\right) \\
        &\quad +\left(w^{ij}\,F^{kl,rs} - F^{ij}\,w^{kl,rs}\right) h_{kl;i} h_{rs;j} \\
       =&\, w_a\left(\left(F_a\,a^2+F_b\,b^2\right)a+\left(F-F_a\,a-F_b\,b\right)a^2\right) \\
        &\quad +w_b\left(\left(F_a\,a^2+F_b\,b^2\right)b+\left(F-F_a\,a-F_b\,b\right)b^2\right) \\
        &\quad +w_a\left(F_{aa}\,h^2_{11;1} + 2\,F_{ab}\,h_{11;1} h_{22;1} + F_{bb}\,h^2_{22;1} + 2\,\frac{F_a - F_b}{a-b}\,h^2_{12;1}\right) \\
        &\quad -F_a\left(w_{aa}\,h^2_{11;1} + 2\,w_{ab}\,h_{11;1} h_{22;1} + w_{bb}\,h^2_{22;1} + 2\,\frac{w_a - w_b}{a-b}\,h^2_{12;1}\right) \\
        &\quad +w_b\left(F_{aa}\,h^2_{11;2} + 2\,F_{ab}\,h_{11;2} h_{22;2} + F_{bb}\,h^2_{22;2} + 2\,\frac{F_a - F_b}{a-b}\,h^2_{12;2}\right) \\
        &\quad -F_b\left(w_{aa}\,h^2_{11;2} + 2\,w_{ab}\,h_{11;2} h_{22;2} + w_{bb}\,h^2_{22;2} + 2\,\frac{w_a - w_b}{a-b}\,h^2_{12;2}\right) \\
        &\qquad \left(\textit{use formula \eqref{eq:fijkl}}\right). 
  \end{align*}
  Clearly, the constant terms $C_w\left(a,b\right)$ are the first two lines of terms after the last equation equal sign.
  It remains to compute the gradient terms using identities \eqref{eq:h111} and \eqref{eq:h222}.
  We obtain the first gradient terms
  \begin{align*}
    E_w\left(a,b\right)\cdot h^2_{11;1} =&\,w_a\left(F_{aa}\,h^2_{11;1} + 2\,F_{ab}\,\alpha\,h^2_{11;1} + F_{bb}\,\alpha^2\,h^2_{11;1}\right) \\
                      &\quad -F_a\left(w_{aa}\,h^2_{11;1} + 2\,w_{ab}\,\alpha\,h^2_{11;1} + w_{bb}\,\alpha^2\,h^2_{11;1}\right) \\
                      &\quad +2\,\frac{1}{a-b}\,\left(w_b\,F_a-w_b\,F_b-w_a\,F_b+w_b\,F_b\right)\alpha^2\,h^2_{11;1} \\
                          =&\,\bigg(w_a\left(F_{aa} + 2\,F_{ab}\,\alpha + F_{bb}\,\alpha^2\right) \bigg. \\
                      &\quad -F_a\left(w_{aa} + 2\,w_{ab}\,\alpha + w_{bb}\,\alpha^2\right) \\
                      &\quad \left.\,+2\,\frac{w_b\,F_a - w_a\,F_b}{a-b}\,\alpha^2\right)\,h^2_{11;1},
  \end{align*}
  and we obtain second gradient terms
  \begin{align*}
    G_w\left(a,b\right)\cdot h^2_{22;2} 
                     =&\,w_b\left(F_{aa}\,\frac{1}{\alpha^2}\,h^2_{22;2} + 2\,F_{ab}\,\frac{1}{\alpha}\,h^2_{22;2} + F_{bb}\,h^2_{22;2} \right) \\
                 &\quad -F_b\left(w_{aa}\,\frac{1}{\alpha^2}\,h^2_{22;2} + 2\,w_{ab}\,\frac{1}{\alpha}\,h^2_{22;2} + w_{bb}\,h^2_{22;2} \right) \\
                 &\quad +2\,\frac{1}{a-b}\left(w_a\,F_a - w_a\,F_b - w_a\,F_a + w_b\,F_a\right)\frac{1}{\alpha^2}\,h^2_{22;2} \\
                     =&\, \frac{1}{\alpha^2}\cdot\bigg(w_b \left(F_{aa} + 2\,F_{ab}\,\alpha + F_{bb}\,\alpha^2\right)\bigg. \\
                 &\qquad\qquad -F_b \left(w_{aa} + 2\,w_{ab}\,\alpha + w_{bb}\,\alpha^2\right) \\
                 &\qquad\qquad \bigg. +2\, \frac{w_b\,F_a - w_a\,F_b}{a-b}\bigg)h_{22;2}^2. \qedhere
  \end{align*}
\end{proof}

\subsection{$\alpha$-conditions}\label{sc:rmpf} 
RMPF Definition \ref{def:rmpf} and MPF Definition \ref{def:mpf} are the same up to condition \eqref{III},
and in Lemma \ref{lem:LwCriticalPoint} we have shown that RMPF and MPF conditions \eqref{IV} are equivalent.
So it only remains to motivate MPF condition \eqref{V}. 
RMPF condition \eqref{rV} assumes any such function to be rational. 
As mentioned before, the example of
\begin{align*}
  w = \frac{(a-b)^2(a+b)^{2\sigma}}{(a\,b)^2},
\end{align*}
given by O. Schn\"urer and F. Schulze \cite{fs:convexity} fulfills all RMPF conditions 
except RMPF condition \eqref{rV} for any $2\sigma\notin \N$.
But it can be used to show convergence to a round point 
using the maximum-principle.

As a first step, we analyzed all functions that were used 
to show convergence to a round point and convergence to a sphere at $\infty$ in 
\cite{ba:gauss}, \cite{ac:surfaces}, \cite{ql:surfaces},
\cite{os:surfacesA2}, \cite{os:surfaces}, and \cite{fs:convexity}.
All of them have in common that they fulfill MPF condition \eqref{V}. 

As a second step, we seek to prove that all RMPF are MPF.
This is Lemma \ref{lem:rmpf}.
Here, we show that for contracting $F^\sigma_\xi$ 
with $\xi>0$, or $\xi=0$ and $1<\sigma\leq 2$,
RMPF fulfill MPF condition \eqref{V}.
This includes the contracting normal velocities
$F^\sigma_0=K^{\sigma/2}$, $F^\sigma_1=H^\sigma$, and $F^\sigma_2=|A|^\sigma$.
A similar Lemma for the remaining contracting and the expanding $F^\sigma_\xi$ 
is in progress. 

In \cite{mf:on}, we show that there are no RMPF for any $F=K^{\sigma/2}$
with $\sigma>2$.
And due to B. Andrews we know that convex surfaces do not necessarily
converge to a round point for any power $\frac{1}{2} \leq \sigma \leq 1$.
For the power $\sigma=\frac{1}{2}$, they converge to ellipsoids \cite{ba:contraction}.
For any power $\frac{1}{2}\leq \sigma \leq 1$, surfaces
contract homothetically in the limit \cite{ba:motion}.
This is why we cannot expect any RMPF to exist for any $F=K^{\sigma/2}$
with $0<\sigma\leq 1$. \\

We begin this chapter by showing that the quantity $\alpha=-\frac{w_a}{w_b}$ 
is strictly positive for 
all $0<a,b$, if $w$ is a RMPF or a MPF.

In Remark \ref{rem:examples condition}, we give an example for each of the $\alpha$-conditions in MPF condition \eqref{V}. 
Then we proceed with Lemma \ref{lem:alpha o estimates},
which contains some calculations for Lemma \ref{lem:rmpf}.
By this Lemma, we motivate MPF condition \eqref{V}.

\begin{lemma}[Quantity $\alpha$ is strictly positive]\label{lem:alpha}
  Let $w$ be a RMPF or a MPF.
  Then we have
  \begin{align*}
    \alpha>0
  \end{align*}
  for all $0<a,\,b$.   
  Note that $\alpha=-\frac{w_a}{w_b}$ as in MPF Definition \ref{def:mpf}. 
\end{lemma}
\begin{proof}
  Let $0<a,\,b$.
  MPF condition \eqref{III} and the symmetry of $w$ in $a$ and $b$ imply
  \begin{align*}
    w_a(a,b) <&\, 0 \text{ for all } 0 < a < b, \\
    w_a(a,b) >&\, 0 \text{ for all } 0 < b < a, \\  
    w_b(a,b) >&\, 0 \text{ for all } 0 < a < b, \\
    w_b(a,b) <&\, 0 \text{ for all } 0 < b < a.
  \end{align*}
  Recall, $\alpha=-\frac{w_a}{w_b}$. 
  Hence, the claim follows for all $0<a,b$ with $a\neq b$. 
  
  Now, let $0<a,b$ with $a=b$. 
  By MPF condition \eqref{IV}, in particular $C_w\leq 0$ for all $0<a,b$, we have
  \begin{align*}
     &\, w_a\,\underbrace{a\left(\left(F_a\,a^2+F_b\,b^2\right) + \left(F-F_a\,a-F_b\,b\right)a\right)}_{:=\phi} \\
     &\quad+w_b\,\underbrace{b\left(\left(F_a\,a^2+F_b\,b^2\right) + \left(F-F_a\,a-F_b\,b\right)b\right)}_{:=\psi} \\
 \leq&\, 0
  \end{align*}
  for all $0<a,b$. This implies
  \begin{align*}
    w_a\,\phi\leq-w_b\,\psi
  \end{align*} 
  for all $0<a,b$. \\
  Letting $0<a<b$ yields  $\alpha\cdot\phi \geq \psi$, and 
  letting $0<b<a$ yields $\alpha\cdot\phi \leq \psi$. 
  
  Hence, for $a\rightarrow b$ we get
  \begin{align*}
    \alpha\cdot \phi \leq \psi \leq \alpha\cdot \phi.
  \end{align*}
  This implies $\psi=\alpha\cdot\phi$ for all $0<a,b$ with $a=b$.
  We get
  \begin{align*}
    \phi =&\, \alpha\cdot\phi, \text{ since } \psi=\phi \text{ at } a=b, \\
    \text{and } \alpha =&\, 1, \text{ if } \phi\neq 0 \text{ at } a=b.
  \end{align*}
  Let $a=b=1$, then
  \begin{align*}
           \phi(1,1)
      =&\, 1\cdot ((F_a\cdot 1+F_b\cdot 1)+(F-F_a\cdot 1-F_b\cdot 1)\cdot 1) \\
      =&\, F \neq 0 \text{ by Definition \ref{def:normal velocity} of a normal velocity } F.
  \end{align*}
  Since $\phi$ is homogeneous, we have $\phi\neq 0$ for all $0<a,b$ with $a=b$. 
 
  This concludes the proof.  
\end{proof}

\begin{remark}[$\alpha$-conditions]\label{rem:examples condition}
  We give an example for each of the $\alpha$-conditions in MPF condition \eqref{V}. 
  Let $F=F^\sigma_0=\sgn(\sigma)\cdot K^{\sigma/2}$. Then we have 
  \begin{center}
    \begin{tabular}{| l | c || c | c | c |}
      \hline 
      & MPF & & & $\alpha$-condition \\
      \hline \hline  
      & & & & \\
      $\Bigg.\Bigg.$ $\sigma=2$ & $(a-b)^2$ & B. Andrews & \cite{ba:gauss} & $\alpha = 1$ \\
      \hline
      & & & & \\
      $\Bigg.\Bigg.$ $\sigma\in(1,2)$ & $\frac{(a-b)^2(a\,b)^\sigma}{(a\,b)^2}$ & B. Andrews, X. Chen & \cite{ac:surfaces} & $\alpha = \frac{-\sigma+2}{\sigma\rho} + o\left(\rho^{-1} \right)$ \\
      \hline \hline 
      & & & & \\
      $\Bigg.\Bigg.$ $\sigma\in(-2,0)$ & $\frac{(a-b)^2(a\,b)^{\sigma/2}}{(a\,b)}$ & Q. Li & \cite{ql:surfaces} & $\alpha = \frac{-\sigma+2}{(\sigma+2)\rho} + o\left(\rho^{-1} \right)$ \\
      \hline
      & & & & \\
      $\Bigg.\Bigg.$ $\sigma=-2$ & $\frac{(a-b)^2}{(a\,b)^2}$ & O. Schn\"urer & \cite{os:surfaces} & $\alpha = \frac{1}{\rho^2}$\\
      \hline  
    \end{tabular}
  \end{center}  
\end{remark}

\begin{lemma}[$o$-estimates for the quantity $\alpha$]\label{lem:alpha o estimates}
  Let $F^\sigma_\xi$ be defined as in Remark \ref{rem:normal velocity xi}.
  Assume $w$ to be a RMPF or a MPF for a $F^\sigma_\xi$. 
  
  Let $F^\sigma_\xi$ be a contracting normal velocity, \ie $\sigma>0$. 
  
  Then for $\rho\rightarrow 0$ we get
  \begin{align}
    \Big. \Big. \alpha \geq&\, \frac{1}{\sigma\,\rho} + o\left(\rho^{-1}\right), \text{ if } \xi>0, \umbruch \\
    \Big. \Big. \alpha \geq&\, \frac{-\sigma+2}{\rho} + o\left(\rho^{-1}\right), \text{ if } \xi=0 \text{ and } \sigma\neq 2, \umbruch \\
    \Big. \Big. \alpha \geq&\, 1, \text{ if } \xi=0 \text{ and } \sigma=2, \umbruch \\
    \Big. \Big. \alpha \geq&\, \frac{-\sigma+1}{\sigma\,\rho^{-\xi}} + \sigma\left(\rho^{-1}\right), \text{ if } -1<\xi<0, \umbruch \\
    \Big. \Big. \alpha \geq&\, \frac{-\sigma+1}{(\sigma+1)\rho^2} + \frac{1}{\rho} + o\left(\rho^{-1}\right), \text{ if } \xi=-1, \umbruch \\
    \Big. \Big. \alpha \geq&\, \frac{-\sigma+1}{\rho^2} + \frac{\sigma}{\rho} + o\left(\rho^{-1}\right), \text{ if } \xi<-1. 
  \end{align}
  
  Let $F^\sigma_\xi$ be an expanding normal velocity, \ie $\sigma<0$. 
  
  Then for $\rho\rightarrow 0$ we get
  \begin{align}
    \Big. \Big. \alpha \geq&\, \frac{1}{\sigma\,\rho} + o\left(\rho^{-1}\right), \text{ if } \xi>0, \umbruch \\
    \Big. \Big. \alpha \geq&\, \frac{-\sigma+2}{\rho} + o\left(\rho^{-1}\right), \text{ if } \xi=0, \umbruch \\
    \Big. \Big. \alpha \geq&\, \frac{-\sigma+1}{\sigma\,\rho^{-\xi}} + o\left(\rho^{-1}\right), \text{ if } -1<\xi<0, \umbruch \\
    \Big. \Big. \alpha \geq&\, \frac{-\sigma+1}{(\sigma+1)\rho^2} + \frac{1}{\rho} + o\left(\rho^{-1}\right), \text{ if } \xi=-1 \text{ and } \sigma<-1, \umbruch \\
    \Big. \Big. \alpha \leq&\, \frac{1}{\rho^3}, \text{ if } \xi=-1 \text{ and } \sigma=-1, \umbruch \\
    \Big. \Big. \alpha \leq&\, \frac{-\sigma+1}{(\sigma+1)\rho^2} + \frac{1}{\rho} + o\left(\rho^{-1}\right), \text{ if } \xi=-1 \text{ and } -1<\sigma<0, \umbruch \\
    \Big. \Big. \alpha \leq&\, \frac{-\sigma+1}{\rho^2} + \frac{\sigma}{\rho} + o\left(\rho^{-1}\right), \text{ if } \xi<-1.
  \end{align}
\end{lemma}
\begin{proof}
  Assume $w$ to be a RMPF or a MPF for a $F^\sigma_\xi$.
  By Lemma \ref{lem:necessary}, we have
  \begin{align*}
    0 \geq&\, C^\alpha_\beta(\rho) \\
         =&\, \sgn(\sigma)\cdot\left( -\alpha\,\rho\left(\sigma+\rho(-\sigma+1)+\beta\,\rho^2 \right)
                                    +1 +\beta\,\rho(1-\sigma+\rho\,\rho) \right),
  \end{align*}
  which implies
  \begin{align}\label{id:alpha sign}
    \begin{split} 
          &\, \sgn(\sigma)\cdot\Big( \alpha\,\rho \big(\overbrace{ \sigma+\rho(-\sigma+1) + \beta\,\rho^2 }^{=:A}\big)\Big) \\
      \geq&\, \sgn(\sigma)\cdot \Big( \underbrace{1+\beta\,\rho\big(1-\sigma+\rho\,\sigma\big)}_{=:B}\Big).
    \end{split} 
  \end{align}
  We seek estimates for $\alpha$ in some zero neighborhood of $\rho$.
  So we need to know the sign of $A$ in such a neighborhood. 
  For our convenience, we calculate the quantity $\alpha$ for 
  \begin{align*}
    \beta = \rho^\xi.
  \end{align*}
  Towards the end of the proof we just need to add one to $\xi$
  to obtain the proper results for $F^\sigma_\xi$.
  This yields 
  \begin{align*}
    A = \sigma + \rho(-\sigma+1) + \rho^{\xi+2}.
  \end{align*}
  We get
  \begin{align*}
    \limr A =&\, \sigma, \text{ if } \xi>-2, \\
    \limr A =&\, \sigma+1, \text{ if } \xi=-2 \text{ and } \sigma\neq -1, \\
    \limr \rho^{-1}\cdot A =&\, 2, \text{ if } \xi=-2 \text{ and } \sigma=-1, \\
    \limr \rho^{-\xi-2}\cdot A =&\, 1, \text{ if } \xi<-2. 
  \end{align*}
  Therefore, $A$ is strictly positive in some zero neighborhood of $\rho$, if $F^\sigma_\xi$
  is a contracting normal velocity, \ie $\sigma>0$. \\
  If $F^\sigma_\xi$ is an expanding normal velocity, \ie $\sigma<0$, we obtain
  a zero neighborhood of $\rho$, where
  \begin{itemize}
    \item $A$ is strictly negative, if $\xi>-2$, 
    \item $A$ is strictly negative, if $\xi=-2$ and $\sigma<-1$,
    \item $A$ is strictly positive, if $\xi=-2$ and $\sigma=-1$,
    \item $A$ is strictly positive, if $\xi=-2$ and $-1<\sigma<0$,
    \item $A$ is strictly positive, if $\xi<-2$.
  \end{itemize}
  Note that by \eqref{id:alpha sign}, we have 
  \begin{align*}
    \alpha\,\rho\,A \geq B, \text{ if } \sigma>0, \\
    \alpha\,\rho\,A \leq B, \text{ if } \sigma<0,
  \end{align*}
  for all $0<\rho<1$. Now, we use the results on $A$. 
  
  Let $F^\sigma_\xi$ be a contracting normal velocity. 
  Then we have
  \begin{align}\label{id:alpha estimate}
    \alpha \geq \frac{B}{\rho\,A}
  \end{align}
  in some zero neighborhood of $\rho$. 
  
  Let $F^\sigma_\xi$ be an expanding normal velocity.
  Then we have
  \begin{align}\label{id:alpha estimate I}
    \alpha \geq \frac{B}{\rho\,A}
  \end{align}  
  in some zero neighborhood of $\rho$,
  if $\sigma>-2$, or $\xi=-2$ and $\sigma<-1$, \\
  and we have
  \begin{align}\label{id:alpha estimate II}
    \alpha \leq \frac{B}{\rho\,A}
  \end{align}
  in some zero neighborhood of $\rho$, 
  if $\xi=-2$ and $-1\leq \sigma<0$, or $\xi<-2$. \\
  In the following, we can calculate the desired $o$-estimates for the quantity $\alpha$. \\
  First, let $F^\sigma_\xi$ be a contracting normal velocity. \\
  By \eqref{id:alpha estimate}, we have
  \begin{align*}
    \alpha \geq \frac{1+\rho^{\xi+1}(1-\sigma+\rho\,\sigma)}{\rho\left( \sigma + \rho(-\sigma+1) + \rho^{\xi+2} \right)}.
  \end{align*}
  
  Let $\xi>-1$. For $\rho\rightarrow 0$ we get
  \begin{align*}
    \alpha \geq \frac{1}{\sigma\,\rho} + o\left(\rho^{-1}\right).
  \end{align*}
  Let $\xi=-1$. For $\rho\rightarrow 0$ we get
  \begin{align*}
    \alpha \geq&\, \frac{-\sigma+2}{\rho} + o\left(\rho^{-1}\right), \text{ if } \sigma\neq 2, \\
    \alpha \geq&\, 1, \text{ if } \sigma=2.
  \end{align*}
  Let $-2<\xi<-1$. For $\rho \rightarrow 0$ we get
  \begin{align*}
    \alpha \geq \frac{-\sigma+1}{\sigma\,\rho^{-\xi}} + o\left(\rho^{-1}\right).
  \end{align*}
  Let $\xi=-2$. For $\rho\rightarrow 0$ we get
  \begin{align*}
    \alpha \geq \frac{-\sigma+1}{(\sigma+1)\rho^2} + \frac{1}{\rho} + o\left(\rho^{-1}\right).
  \end{align*}
  Let $\xi<-2$. For $\rho\rightarrow 0$ we get
  \begin{align*}
    \alpha \geq \frac{-\sigma+1}{\rho^2} + \frac{\sigma}{\rho} + o\left(\rho^{-1}\right).
  \end{align*}
  Next, let $F^\sigma_\xi$ be an expanding normal velocity. 
  Here, we use \eqref{id:alpha estimate I}, and \eqref{id:alpha estimate II}. 
  
  Let $\xi>-1$. For $\rho\rightarrow 0$ we get
  \begin{align*}
    \alpha \geq \frac{1}{\sigma\,\rho} + o\left(\rho^{-1}\right).
  \end{align*}
  Let $\xi=-1$. For $\rho\rightarrow 0$ we get
  \begin{align*}
    \alpha\geq \frac{-\sigma+2}{\rho} + o\left(\rho^{-1}\right).
  \end{align*}
  Let $-2<\xi<-1$. For $\rho\rightarrow 0$ we get
  \begin{align*}
    \alpha \geq \frac{-\sigma+1}{\sigma\,\rho^{-\xi}} + o\left(\rho^{-1}\right).
  \end{align*}
  Let $\xi=-2$ and $\sigma<-1$. For $\rho\rightarrow 0$ we get
  \begin{align*}
    \alpha \geq \frac{-\sigma+1}{(\sigma+1)\rho^2} + \frac{1}{\rho} + o\left(\rho^{-1}\right).
  \end{align*} 
  Let $\xi=-2$ and $\sigma=-1$. For $\rho\rightarrow 0$ we get
  \begin{align*}
    \alpha \leq \frac{1}{\rho^3}.
  \end{align*}
  Let $\xi=-2$ and $-1<\sigma<0$. For $\rho\rightarrow 0$ we get
  \begin{align*}
    \alpha \leq \frac{-\sigma+1}{(\sigma+1)\rho^2} + \frac{1}{\rho} + o\left(\rho^{-1}\right).
  \end{align*}
  Let $\xi<-2$. For $\rho\rightarrow 0$ we get
  \begin{align*}
    \alpha \leq \frac{-\sigma+1}{\rho^2} + \frac{\sigma}{\rho} + o\left(\rho^{-1}\right).
  \end{align*}
  This concludes the proof.
\end{proof}

\begin{lemma}[RMPF fulfill $\alpha$-condition]\label{lem:rmpf}
  Let $F^\sigma_\xi$ be defined as in Remark \ref{rem:normal velocity xi}.
  Let $w$ be a RMPF for a contracting normal velocity $F^\sigma_\xi$, \ie $\sigma>0$,
  with $\xi>0$, or $\xi=0$, and $1<\sigma<2$. 
  
  Then for $\rho\rightarrow 0$ we get
  \begin{align*}
    \alpha = \frac{c}{\rho} + o\left(\rho^{-1}\right)
  \end{align*}
  for some $c>0$, which is the second part of MPF condition \eqref{V}. 
  
  The function $w=(a-b)^2$ from \cite{ba:gauss} is a RMPF 
  for $F^\sigma_\xi$ with $\xi=0$ and $\sigma=2$. 
  
  For $\rho\rightarrow 0$ we get
  \begin{align*}
    \alpha=1,
  \end{align*}
  which is the first part of MPF condition \eqref{V}.
\end{lemma}
\begin{proof}
  Let $w \not\equiv 0$ be a rational symmetric homogeneous function of degree $\chi\in\R\setminus\{0\}$.
  Then $w$ has the form
  \begin{align*}
    w = \frac{p(a,b)}{q(a,b)\cdot(a\,b)^l},
  \end{align*}
  for some $l\in\Z$, and some symmetric homogeneous polynomials $p(a,b)$, and $q(a,b)$, of degree $g$, and $h$, respectively.
  \begin{align*}
                     p\left(a,b\right) =&\,\sum_{i=0}^{\lfloor g/2\rfloor} c_{i+1}\left(a^{g-i}b^i + a^ib^{g-i}\right), \\
                     q\left(a,b\right) =&\,\sum_{j=0}^{\lfloor h/2\rfloor} d_{j+1}\left(a^{h-j}b^j + a^jb^{h-j}\right),
  \end{align*}
  where $c_1>0$, and $d_1>0$, respectively.
  
  Next, we calculate
  \begin{align*}
    \alpha = -\frac{w_a}{w_b} 
           = -\frac{b\left(p\,q\,l+a\left(p\,q_a-p_a\,q\right)\right)}{a\left(p\,q\,l+b\left(p\,q_b-p_b\,q\right)\right)}.
  \end{align*}
  Setting $(a,b)=(\rho,1)$, where $0<\rho<1$, yields
  \begin{align}\label{id:alpha p q}
      \alpha =&\, -\frac{p\,q\,l+\rho\left(p\,q_a-p_a\,q\right)}{\rho\left(p\,q\,l+p\,q_b-p_b\,q\right)}.
    \end{align}
  Firstly, we calculate
  \begin{align*}
    p(a,b)=&\,\sum_{i=0}^{\lfloor g/2\rfloor} c_{i+1}\left(a^{g-i}b^i + a^ib^{g-i}\right), \\
    p(\rho,1)=&\,\sum_{i=0}^{\lfloor g/2\rfloor} c_{i+1}\left(\rho^{g-i} + \rho^i\right), \\
    p(\rho,0)=&\,c_1. \Bigg .\Bigg.
  \end{align*}
  Secondly, we calculate
  \begin{align*}
    p_a(a,b) =&\,\sum_{i=0}^{\lfloor g/2\rfloor} c_{i+1}\left(\left(g-i\right)a^{g-i-1}b^i + i\,a^{i-1}b^{g-i}\right), \\
    p_a(\rho,1) =&\,\sum_{i=0}^{\lfloor g/2\rfloor} c_{i+1}\left(\left(g-i\right)\rho^{g-i-1} + i\,\rho^{i-1}\right), \\
    p_a(0,1) =&\, c_2. \Bigg .\Bigg.
  \end{align*}
  Thirdly, we calculate
  \begin{align*}  
    p_b(a,b) =&\,\sum_{i=0}^{\lfloor g/2\rfloor} c_{i+1}\left(i\,a^{g-i}b^{i-1} + \left(g-i\right)a^i b^{g-i-1}\right), \\
    p_b(\rho,1) =&\,\sum_{i=0}^{\lfloor g/2\rfloor} c_{i+1}\left(i\,\rho^{g-i} + \left(g-i\right)\rho^i\right), \\
    p_b(0,1) =&\,g\,c_1. \Bigg .\Bigg.
  \end{align*}
  Analogously, we obtain
  \begin{align*}
    q(0,1) =&\, d_1, \\
    q_a(0,1) =&\, d_2, \\
    q_b(0,1) =&\, h\,d_1.
  \end{align*}
  Now, we can combine identity \eqref{id:alpha p q}, and the calculations for $p$, $p_a$, $p_b$, 
  and $q$, $q_a$, $q_b$. 
  Hence, for $\rho\rightarrow 0$ we get
  \begin{align}\label{id:alpha l}
    \alpha=\frac{l}{\rho(g-h-l)} + o\left(\rho^{-1}\right).
  \end{align}
  Due to RMPF condition \eqref{rII}, we have $\chi>0$, if $F$ is a contracting normal velocity. 
  Also we have $l\geq 0$, which is a direct consequence of
  RMPF conditions \eqref{rI} and \eqref{rIII}.
  This implies,
  \begin{align*}
    g-h-l\geq g-h-2l=\chi>0.
  \end{align*}
  By Lemma \ref{lem:alpha o estimates}, for $\rho \rightarrow 0$ we get
  \begin{align*}
    \Big. \Big. \alpha \geq&\, \frac{1}{\sigma\,\rho} + o\left(\rho^{-1}\right), \text{ if } \xi>0, \\
    \Big. \Big. \alpha \geq&\, \frac{-\sigma+2}{\rho} + o\left(\rho^{-1}\right), \text{ if } \xi=0 \text{ and } 0<\sigma<2, 
  \end{align*}
  Combining these $o$-estimates with identity \eqref{id:alpha l} yields $l>0$, if $\xi>0$, or $\xi=0$ and $0<\sigma<2$.
  This translates into the second part of MPF condition \eqref{V}.
  Here, the condition on $\alpha$ implies the condition on $\alpha_a$ since $w$ is assumed to be a rational function. 
  
  The part on $w=(a-b)^2$ follows by direct calculations. 
  
  This concludes the proof.
\end{proof}

\section{Necessary conditions for the existence of MPF}

In this chapter, we start with Euler's Theorem on homogeneous functions.
This is Theorem \ref{thm:euler} and implies Corollary \ref{cor:euler}. 

We use Euler's Theorem \ref{thm:euler} in Lemma \ref{lem:one}, and in Lemma \ref{lem:zeroone}.
Here, we show that there are no MPF for contracting normal velocities
which are homogeneous of degree $\sigma\in(0,1]$. 

In Lemma \ref{lem:necessary}, we use Euler's Theorem \ref{thm:euler} and 
its Corollary \ref{cor:euler}.
Here, we reformulate MPF condition \eqref{IV}, \ie
$C_w(a,b)\leq 0$, 
$E_w(a,b)\leq 0$, $G_w(a,b)\leq 0$ for all $0<a,b$.
We replace the function $w$ and the normal velocity $F$ 
by the quantities $\alpha=-\frac{w_a}{w_b}$, and
$\beta=\frac{F_a}{F_b}$. We set $(a,b)=(\rho,1)$, where $0<\rho<1$. 
We obtain the three necessary conditions 
$C^\alpha_\beta(\rho)\leq 0$, $E^\alpha_\beta(\rho)\leq 0$, and $G^\alpha_\beta(\rho)\leq 0$
for all $0<\rho<1$.
In particular, these new constant and gradient terms now match the form of 
MPF condition \eqref{V} ($\alpha$-conditions). 

Combining the new constant and gradient terms from Lemma \ref{lem:necessary} and 
MPF condition \eqref{V} yields Theorem \ref{thm:necessary}.
We are interested in MPF $w$ for normal velocities $F^\sigma_\xi$ as defined in 
Remark \ref{rem:normal velocity xi}.
Theorem \ref{thm:necessary} gives us necessary conditions for the existence of MPF
in terms of $\xi$ and $\sigma$. 

In particular, there exist no MPF for contracting $F^\sigma_\xi$ with $\xi<0$
by Theorem \ref{thm:necessary}. 
We discuss further implications of Theorem \ref{thm:necessary}
in the following chapters
on $F^\sigma_0=\sgn(\sigma)\cdot K^{\sigma/2}$, 
$F^\sigma_1=\sgn(\sigma)\cdot H^\sigma$,
$F^\sigma_2=\sgn(\sigma)\cdot |A|^\sigma$, and
$F^\sigma_\sigma=\sgn(\sigma)\cdot \tr A^\sigma$.

\subsection{Euler's Theorem on homogeneous functions}

\begin{definition}[Cone]\label{def:set}
  Let $\Omega \subset \R^n\setminus\{0\}$ be an open set.
  We call $\Omega$ a \textit{cone}
  if $\forall x\in\Omega : \{r\,x : 0 < r < \infty \} \subset \Omega$.
\end{definition}

\begin{definition}[Homogeneous function]\label{def:homogeneous}
  Let $\Omega$ be a cone.
  We call $f:\Omega\to\R$ a \textit{homogeneous function} of degree $\chi\in\R$ if
  \begin{align*}
    f(s\,x) = s^\chi\,f(x), 
  \end{align*}
  for all $s>0$, and for all $x\in\Omega$.
\end{definition}

\begin{theorem}[Euler's Theorem on homogeneous functions]\label{thm:euler}
  Let $\Omega\subset\R^2$ be a cone. 
  Let $f(a,b)\in C^1\left(\Omega\right)$ be a function.
  Then $f$ is homogeneous of degree $\chi$, if and only if 
  \begin{align}\label{eq:euler}
    a\,f_a+b\,f_b=\chi\,f.
  \end{align}
\end{theorem}
\begin{proof}
  We refer to S. Hildebrandt \cite{sh:analysis2}.
\end{proof}

\begin{corollary}[Euler's Corollary]\label{cor:euler}
  Let $f(a,b)\in C^2\left(\R^2\setminus\{0\}\right)$ be
  a homogeneous function of degree $\chi$.
  Then we have
  \begin{align}
    a\,f_{aa}+b\,f_{ab}=(\chi-1)\,f_a, \label{eq:euleraa}\\
    a\,f_{ab}+b\,f_{bb}=(\chi-1)\,f_b. \label{eq:eulerbb}
  \end{align}
\end{corollary}
\begin{proof}
  We use Theorem \ref{thm:euler}. First, by differentiating equation \eqref{eq:euler} with respect to $a$, we obtain 
  equation \eqref{eq:euleraa}. Next, by differentiating equation \eqref{eq:euler} with respect to $b$, we obtain
  equation \eqref{eq:eulerbb}.
\end{proof}

\subsection{No MPF for normal velocities with $0 < \hom\,F \leq 1$} 

\begin{lemma}[Contracting normal velocities homogeneous of degree one]\label{lem:one}
  There exists no MPF for any contracting normal velocity $F$
  homogeneous of degree one.
\end{lemma}
\begin{proof}
  Suppose there exists a MPF $w$ for some contracting normal velocity
  homogeneous of degree one. Then we have by MPF condition \eqref{IV}
  \begin{align*}
    C_w(a,b)\leq 0
  \end{align*}
  for all $0<a,b$. Furthermore,
  Euler's Theorem \ref{thm:euler} implies
  \begin{align*}
    \chi\,w =&\, a\,w_a+b\,w_b, \\
    \text{and}\qquad F=&\, a\,F_a+b\,F_b.
  \end{align*}
  Combining these identities with MPF condition \eqref{IV} yields
  \begin{align*}
    0\geq&\, C_w(a,b) \\
        =&\, a\,w_a\left(\left(a^2\,F_a+b^2\,F_b\right)+a\left(F-a\,F_a-b\,F_b\right)\right) \\
         &\quad +b\,w_b\left(\left(a^2\,F+b^2\,F_b\right)+b\left(F-a\,F_a-b\,F_b\right)\right) \\
        =&\, a\,w_a\left(a^2\,F_a+b^2\,F_b\right)+b\,w_b\left(a^2\,F_a+b\,F_b\right) \\
        =&\, \chi\,w\left(a^2\,F_a+b^2\,F_b\right) \\
        >&\, 0
  \end{align*}
  for all $0<a,b$ with $a\neq b$.
  Note that $\chi>0$ for contracting normal velocities by MPF condition \eqref{II}, 
  $w>0$ for all $0<a,b$ with $a\neq b$ by MPF condition \eqref{I},
  and $F_a,\,F_b>0$ for all $0<a,b$, since $F$ is a normal velocity.
  Hence, the claim follows.
\end{proof}

\begin{lemma}[Contracting normal velocities homogeneous of degree between zero and one]\label{lem:zeroone}
  There exists no MPF for any contracting normal velocity $F$ homogeneous of degree $\sigma\in(0,1)$.
\end{lemma}
\begin{proof}
  Suppose there exists a MPF $w$ for some contracting normal velocity homogeneous of degree $\sigma\in(0,1)$.
  Then we have by MPF condition \eqref{IV}
  \begin{align*}
    C_w(a,b) \leq 0
  \end{align*}
  for all $0<a,b$. Furthermore, Euler's Theorem \ref{thm:euler} implies
  \begin{align*}
    \chi\,w =&\, a\,w_a+b\,w_b, \\
    \text{and}\qquad \sigma\,F =&\, a\,F_a+b\,F_b.
  \end{align*}
  Combining these identities with MPF condition \eqref{IV} yields
  \begin{align*}
    0 \geq&\, C_w(a,b) \\
         =&\, a\,w_a\left(\left(a^2\,F_a+b^2\,F_b\right)+a\left(F-a\,F_a-b\,F_b\right)\right) \\
          &\quad +b\,w_b\left(\left(a^2\,F_a+b^2\,F_b\right)+b\left(F-a\,F_a-b\,F_b\right)\right) \\
         =&\,(a\,w_a+b\,w_b)\left(a^2\,F_a+b^2\,F_b\right)+\left(a^2\,w_a+b^2\,w_b\right)(F-a\,F_a-b\,F_b) \\
         =&\, \chi\,w\left(a^2\,F_a+b^2\,F_b\right)+(1-\sigma)F\left(a^2\,w_a+b^2\,w_b\right).
  \end{align*}
  Note that $\chi>0$ for contracting normal velocities by MPF condition \eqref{II}, 
  $w>0$ for all $0<a,b$ with $a\neq b$ by MPF condition \eqref{I},
  and $F,\,F_a,\,F_b>0$ for all $0<a,b$, since $F$ is a contracting normal velocity. 
  
  So we get 
  \begin{align*}
    0\geq a^2\,w_a+b^2\,w_b
  \end{align*}
  for all $0<a,b$.
  By MPF condition \eqref{III}, we have
  \begin{align*}
    w_a <&\, 0, \\
    w_b >&\, 0
  \end{align*}
  for all $0<a<b$, see proof of Lemma \ref{lem:alpha} for details.
  Using the last three inequalities we obtain
  \begin{align*}
    \alpha = -\frac{w_a}{w_b} \geq \frac{b^2}{a^2}
  \end{align*}
  for all $0<a<b$.
  Setting $(a,b)=(\rho,1)$, where $0<\rho<1$, yields
  \begin{align}\label{ineq:rho2}
    \alpha \geq \frac{1}{\rho^2}
  \end{align}
  for all $0<\rho<1$.
  If $F$ is contracting, we have for $\rho\rightarrow 0$
  \begin{align*}
    \alpha =&\, c, \\
    \text{or}\quad \alpha =&\, \frac{c}{\rho} + o\left(\rho^{-1}\right)
  \end{align*}
  by MPF condition \eqref{V}.
  This results in a contradiction to inequality \eqref{ineq:rho2}.
  Hence, the claim follows.
\end{proof}

\begin{lemma}[Necessary conditions for the existence of MPF]\label{lem:necessary}
  Let $w$ be a MPF for a normal velocity $F$. Set $(a,b)=(\rho,1)$, where $0<\rho<1$.
  Then the constant terms $C^\alpha_\beta(\rho)$, and the gradient terms 
  $E^\alpha_\beta(\rho)$ and $G^\alpha_\beta$ are non-positive for all $0<\rho<1$. 
  
  The constant terms are
  \begin{align}
    C^\alpha_\beta(\rho) = \sgn(\sigma)\cdot\left(\left(-\alpha\,\rho^2+1\right)\left(\beta\,\rho+1\right)+\left(\alpha+\beta\right)\left(-1+\rho\right)\rho\,\sigma\right).
  \end{align}
  The gradient terms are
  \begin{align}
    \begin{split}
      E^\alpha_\beta(\rho) =&\,-\alpha_a\,\beta(\alpha\,\rho-1)(\beta\,\rho+1)(1-\rho) \\
                            &\,\quad-\alpha\,\beta_a(\alpha\,\rho-1)^2(1-\rho) \\
                            &\,\quad-\alpha(\alpha+\beta)(\alpha(1+\rho+(1-\rho)\sigma)
                              +\beta(1-\rho)(\sigma-1)+2\,\alpha\,\beta\,\rho),
    \end{split}
  \end{align}
  and
  \begin{align}
    \begin{split}  
      G^\alpha_\beta(\rho) =&\,-\alpha_a(\alpha\,\rho-1)(\beta\,\rho+1)(1-\rho) \\
                            &\,\quad+\beta_a(\alpha\,\rho-1)^2(1-\rho) \\
                            &\,\quad-(\alpha+\beta)(\alpha(1-\rho)(1-\sigma)
                              +\beta(1+\rho-(1-\rho)\sigma)+2).
    \end{split}
  \end{align}
\end{lemma}
\begin{proof}
  In this proof, we make the general assumption $0<a<b$. 
  
  We begin with a few preliminary calculations. 
  
  Recall the quantities $\alpha=-\frac{w_a}{w_b}$, and $\beta=\frac{F_a}{F_b}$
  from MPF Definition \ref{def:mpf}, and Definition \ref{def:beta}, respectively. 
  They imply
  \begin{align}
    w_a =&\, -\alpha\,w_b, \label{id:wa} \\
    F_a =&\, \beta\,F_b. \label{id:fa} 
  \end{align}
  Furthermore, we deduce from Euler's Corollary \ref{cor:euler} the identities
  \begin{align}
    w_{aa} =&\, \frac{(\chi-1)w_a-b\,w_{ab}}{a}, \label{id:waaHelp} \\
    w_{bb} =&\, \frac{(\chi-1)w_b-a\,w_{ab}}{b}, \label{id:wbbHelp} \\
    F_{aa} =&\, \frac{(\sigma-1)F_a-b\,F_{ab}}{a}, \label{id:faaHelp} \\
    F_{bb} =&\, \frac{(\sigma-1)F_b-a\,F_{ab}}{b}. \label{id:fbbHelp}
  \end{align}
  Calculations for $w_{ab}$ in $\alpha,\alpha_a$: 
  
  We have
  \begin{align*}
    \alpha_a = \frac{\partial}{\partial a}\alpha = \frac{-w_b\,w_{aa}+w_a\,w_{ab}}{w_b^2},
  \end{align*}
  which implies
  \begin{align*}
    w_{ab} = \frac{w_b(w_{aa}+w_b\,\alpha_a)}{w_a}.
  \end{align*}
  Using identities \eqref{id:waaHelp} and \eqref{id:wa} we get
  \begin{align}
    w_{ab} = \frac{w_b((\chi-1)\alpha-\alpha_a\,a)}{\alpha\,a-b}. \label{id:wab}
  \end{align}
  Calculations for $w_{aa}$ in $\alpha$, $\alpha$: 
  
  Using identities \eqref{id:waaHelp}, \eqref{id:wab}, and \eqref{id:wa} we get
  \begin{align}
    w_{aa} = \frac{w_b\left(-(\chi-1)\alpha^2+\alpha_a\,b\right)}{\alpha\,a-b}. \label{id:waa}
  \end{align}
  Calculations for $w_{bb}$ in $\alpha$, $\alpha_a$: 
  
  Using identities \eqref{id:wbbHelp}, \eqref{id:wab}, and \eqref{id:wa} we get
  \begin{align}
    w_{bb} = \frac{w_b\left(\alpha_a\,a^2-(\chi-1)\right)}{(\alpha\,a-b)b}. \label{id:wbb}
  \end{align}
  Calculations for $F_{ab}$ in $\beta$, $\beta_a$: 
  
  We have
  \begin{align*}
    \beta_a = \frac{\partial}{\partial a} \beta = \frac{F_b\,F_{aa}-F_a\,F_{ab}}{F_b^2},
  \end{align*}
  which implies
  \begin{align*}
    F_{ab} = \frac{F_b(F_{aa}-F_b\,\beta_a)}{F_a}.
  \end{align*}
  Using identities \eqref{id:faaHelp} and \eqref{id:fa} we get
  \begin{align}
    F_{ab} = \frac{F_b(\beta(\sigma-1)+\beta_a\,a)}{\beta\,a+b}. \label{id:fab}
  \end{align}
  Calculations for $F_{aa}$ in $\beta$, $\beta_a$: 
  
  Using identities \eqref{id:faaHelp}, \eqref{id:fab}, and \eqref{id:fa} we get
  \begin{align}
    F_{aa} = \frac{F_b\left(\beta^2(\sigma-1)+\beta_a\,b\right)}{(\beta\,a+b)b}. \label{id:faa}
  \end{align}
  Calculations for $F_{bb}$ in $\beta$, $\beta_a$: 
  
  Using identities \eqref{id:fbbHelp}, \eqref{id:fab}, and \eqref{id:fa} we get
  \begin{align}
    F_{bb} = \frac{F_b\left(\beta_a\,a^2+(\sigma-1)b\right)}{(\beta\,a+b)b}. \label{id:fbb}
  \end{align}
  Firstly, we combine identities \eqref{id:waa}, \eqref{id:wab}, and \eqref{id:wbb}
  to obtain the useful identity 
  \begin{align}
    w_{aa} + 2\,w_{ab}\,\alpha + w_{bb}\,\alpha^2
    = \frac{(\alpha\,a-b)\alpha_a}{b}\,w_b. \label{id:I}
  \end{align}
  Secondly, we combine identities \eqref{id:faa}, \eqref{id:fab}, and \eqref{id:fbb}
  to obtain the useful identity
  \begin{align}
    F_{aa} + 2\,F_{ab}\,\alpha + F_{bb}\,\alpha^2
    = \frac{(\alpha+\beta)^2(\sigma-1)b+\beta_a(\alpha\,a-b)^2}{(\beta\,a+b)b}\,F_b. \label{id:II}
  \end{align}
  Thirdly, we use identities \eqref{id:wa} and \eqref{id:fa} 
  to obtain the useful identity
  \begin{align}
    2\,\frac{w_b\,F_a-w_a\,F_b}{a-b} = 2\,\frac{\alpha+\beta}{a-b}\,w_b\,F_b. \label{id:III}
  \end{align}
  Now, we can further investigate the constant terms $C_w(a,b)$, and 
  the gradient terms $E_w(a,b)$ and $G_w(a,b)$ from MPF Definition \ref{def:mpf}.
  We compute necessary conditions for the existence of MPF:
  $C^\alpha_\beta(\rho)$, $E^\alpha_\beta(\rho)$, and $G^\alpha_\beta(\rho)$ have to be
  non-positive for all $0<\rho<1$.
  
  \textbf{Constant terms \textit{C}.} By MPF Definition \ref{def:mpf} we have
  \begin{align*}
    C_w(a,b) = F\left(w_a\,a^2+w_b\,b^2\right)+F_a\,w_b\,a\,b(a-b)-F_b\,w_a\,a\,b(a-b) \leq 0
  \end{align*}
  for all $0<a,b$. 
  
  Using Euler's Theorem \ref{thm:euler} for normal velocity $F$,
  \ie $F=\frac{1}{\sigma}\left(a\,F_a+b\,F_b\right)$, and identities \eqref{id:wa} and \eqref{id:fa} we obtain
  \begin{align*}
    C_w(a,b) = \left(\left(-\alpha\,a^2+b^2\right)(\beta\,a+b)-(\alpha+\beta)a\,b(b-a)\sigma\right)\frac{w_b\,F_b}{\sigma} \leq 0
  \end{align*}
  for all $0<a,b$. 
  Setting $(a,b)=(\rho,1)$, where $0<\rho<1$, yields
  \begin{align*}
    C^\alpha_\beta(\rho) = &\, \sgn(\sigma)\cdot\left(\left(-\alpha\,\rho^2+1\right)(\beta\,\rho+1)+(\alpha+\beta)(-1+\rho)\rho\,\sigma\right) \\
                         = &\, C_w(\rho,1) \frac{|\sigma|}{w_b\,F_b} \\
                      \leq &\, 0
  \end{align*}
  for all $0 < \rho < 1$. Note that $w_b,\,F_b>0$ for all $0<\rho<1$. \\
  
  \textbf{Gradient terms \textit{E}.} By MPF Definition \ref{def:mpf}, we have
  \begin{align*}
    E_w(a,b) = &\, w_a \left(F_{aa} + 2\,F_{ab}\,\alpha + F_{bb}\,\alpha^2\right) \\
               &\quad -F_a \left(w_{aa} + 2\,w_{ab}\,\alpha + w_{bb}\,\alpha^2\right) \\
               &\quad +2\,\frac{w_b\,F_a-w_a\,F_b}{a-b}\,\alpha^2, \\
          \leq &\, 0
  \end{align*}
  for all $0<a,b$. \\
  Using identities \eqref{id:I}, \eqref{id:II}, and \eqref{id:III},
  we obtain
  \begin{align*}
    E_w(a,b) = &\, w_a \left(\frac{(\alpha+\beta)^2(\sigma-1)b+\beta_a(\alpha\,a-b)^2}{(\beta\,a+b)b}\,F_b\right) \\
               &\quad -F_a \left(\frac{(\alpha\,a-b)\alpha_a}{b}\,w_b\right) \\
               &\quad +\left(2\,\frac{\alpha+\beta}{a-b}\,w_b\,F_b\right)\alpha^2, \\
          \leq &\, 0
  \end{align*}
  for all $0<a,b$. 
  
  We use identities \eqref{id:wa} and \eqref{id:fa} and set $(a,b)=(\rho,1)$, where $0<\rho<1$.
  This yields
  \begin{align*}
    E^\alpha_\beta(\rho) = &\, -\alpha_a\,\beta(\alpha\,\rho-1)(\beta\,\rho+1)(1-\rho) \\
                           &\, \quad-\alpha\,\beta_a(\alpha\,\rho-1)^2(1-\rho) \\
                           &\, \quad-\alpha(\alpha+\beta)(\alpha(1+\rho+(1-\rho)\sigma)
                             + \beta(1-\rho)(\sigma-1)+2\,\alpha\,\beta\,\rho), \\
                         = &\, E_w(\rho,1)\,\frac{(\beta\,\rho+1)(1-\rho)}{w_b\,F_b} \\
                      \leq &\, 0
  \end{align*}
  for all $0<\rho<1$. Note that $w_b,\,F_b>0$ for all $0<\rho<1$. \\
  
  \textbf{Gradient terms \textit{G}.} By MPF Definition \eqref{def:mpf}, we have
  \begin{align*}
    G_w(a,b) = &\, \frac{1}{\alpha^2}\cdot\bigg(w_b \left(F_{aa} + 2\,F_{ab}\,\alpha + F_{bb}\,\alpha^2\right)\bigg. \\
               & \qquad\qquad -F_b \left(w_{aa} + 2\,w_{ab}\,\alpha + w_{bb}\,\alpha^2\right) \\
               & \qquad\qquad \bigg. +2\, \frac{w_b\,F_a - w_a\,F_b}{a-b}\bigg) \\
          \leq &\, 0
  \end{align*}
  for all $0<a,b$. 
  
  We use identities \eqref{id:wa} and \eqref{id:fa} and set $(a,b)=(\rho,1)$, where $0<\rho<1$.
  This yields
  \begin{align*}
    G_w(a,b) = &\, \frac{1}{\alpha^2}\cdot\bigg(w_b \left(\frac{(\alpha+\beta)^2(\sigma-1)b+\beta_a(\alpha\,a-b)^2}{(\beta\,a+b)b}\,F_b\right)\bigg. \\
               & \qquad\qquad -F_b \left(\frac{(\alpha\,a-b)\alpha_a}{b}\,w_b\right) \\
               & \qquad\qquad \bigg. +2\, \frac{\alpha+\beta}{a-b}\,w_b\,F_b\bigg) \\
          \leq &\, 0  
  \end{align*}
  for all $0<a,b$. 
  
  We use identities \eqref{id:wa} and \eqref{id:fa} and set $(a,b)=(\rho,1)$, where $0<\rho<1$.
  This yields
  \begin{align*}
    G^\alpha_\beta(\rho) = &\, -\alpha_a(\alpha\,\rho-1)(\beta\,\rho+1)(1-\rho) \\
                           &\, \quad+\beta_a(\alpha\,\rho-1)^2(1-\rho) \\
                           &\, \quad-(\alpha+\beta)(\alpha(1-\rho)(1-\sigma)
                             + \beta(1+\rho-(1-\rho)\sigma)+2) \\
                         = &\, G_w(\rho,1)\,\frac{\alpha^2(\beta\,\rho+1)(1-\rho)}{w_b\,F_b} \\
                      \leq &\, 0  
  \end{align*}  
  for all $0<\rho<1$. Note that $w_b,\,F_b>0$ for all $0<\rho<1$. 
  
  This concludes the proof.
\end{proof}

\subsection{When MPF cease to exist for normal velocities $F^\sigma_\xi$}

\begin{theorem}[Necessary conditions for the existence of MPF]\label{thm:necessary}
  Let $w$ be a MPF.
  Let $F^\sigma_\xi$ be defined as in Remark \ref{rem:normal velocity xi}.
  Then the necessary conditions from Lemma \ref{lem:necessary}
  imply the following necessary conditions for the existence of MPF
  for contracting normal velocities $F^\sigma_\xi$, and 
  for the existence of MPF
  for expanding normal velocities $F^\sigma_\xi$.
  \begin{center}
    \begin{tabular}{| l | c || c |}
      \hline
      & $\Big. \Big.$ contracting $F^\sigma_\xi$, $\sigma>1$ & expanding $F^\sigma_\xi$, $\sigma<0$ \\
      \hline \hline
      $\Big. \Big.$ $\xi>0$ & $c=1/\sigma$ & $\sigma\in [-1,0)$ \\
      \hline
      $\Big. \Big.$ $\xi=0$ & $\sigma\in(1,2]$ & $\sigma\in[-2,0)$ \\
      \hline  
      $\Big. \Big.$ $\xi<0$ & MPF non-existent & open problem \\
      \hline
    \end{tabular}
  \end{center}
  Note that $c>0$ is from MPF Definition \ref{def:mpf}.
\end{theorem}
\begin{proof}
  We recall from Remark \ref{rem:normal velocity xi} that
  \begin{align}
    \beta_{F^\sigma_\xi} = \rho^{\xi-1}.
  \end{align}
  For our convenience, we calculate the constant terms $C^\alpha_\beta(\rho)$, and
  the gradient terms $E^\alpha_\beta(\rho)$ and $G^\alpha_\beta(\rho)$
  from Lemma \ref{lem:necessary} for
  \begin{align*}
    \beta = \rho^\xi,\;\beta_a=\xi\,\rho^{\xi-1}.
  \end{align*}
  Towards the end of the proof we just need to add one to $\xi$
  to obtain the proper results for $F^\sigma_\xi$.
  Furthermore, we recall from MPF Definition \ref{def:mpf}
  the $\alpha$-conditions:
  \begin{align*}
    &\,\alpha_1\text{-condition (contracting)}: \\
    &\,\qquad \alpha=c,
    \;\alpha_a=0, \\
    &\,\alpha_2\text{-condition (contracting/expanding)}: \\
    &\,\qquad\alpha=\frac{c}{\rho}+o\left(\rho^{-1}\right),
    \;\alpha_a=-\frac{c}{\rho^2}+o\left(\rho^{-1}\right), \\
    &\,\alpha_3\text{-condition (expanding)}: \\
    &\,\qquad\alpha=\frac{c+d\,\rho}{\rho^2}+o\left(\rho^{-1}\right),
    \;\alpha_a=-\frac{2\,c+d\,\rho}{\rho^3}+o\left(\rho^{-1}\right).
  \end{align*}
  The proof idea is the following.
  For all $\alpha$-conditions and all $\beta=\rho^{\xi}$, where $\xi\in\R$, we calculate
  \begin{align*}
    \limr \rho^{\xi_C}\cdot C^\alpha_\beta(\rho), \\
    \limr \rho^{\xi_E}\cdot E^\alpha_\beta(\rho), \\
    \limr \rho^{\xi_G}\cdot G^\alpha_\beta(\rho), 
  \end{align*}
  for some $\xi_C,\xi_E,\xi_G\in\R$, which depend on $\xi$.
  These calculations are somewhat tedious.
  Since $w$ is MPF, these limits have to be non-negative for continuity reasons.
  Calculating these limits is the first part of the proof.
  In the second part, we draw conclusions from the results 
  of the first part. Here, we get the necessary conditions
  for $F^\sigma_\xi$ in terms of $\xi$, $\sigma$, and $c>0$.
  
  \textbf{Limits of \textit{C}, \textit{E}, and \textit{G}.} \\
  
  We begin by inserting $\beta=\rho^\xi$ and $\beta_a=\xi\,\rho^{\xi-1}$
  into $C^\alpha_\beta(\rho)$, $E^\alpha_\beta(\rho)$, and $G^\alpha_\beta(\rho)$.
  This yields
  \begin{align*}
    C^\alpha_\xi(\rho):=&\,
    \sgn(\sigma)\left(\left(-\alpha\,\rho^2+1\right)\left(\rho^{\xi+1}+1\right)+\left(\alpha+\rho^\xi\right)\left(-1+\rho\right)\rho\,\sigma\right), \\
    E^\alpha_\xi(\rho):=&\,-\alpha_a\,\rho^\xi(\alpha\,\rho-1)(\rho^{\xi+1}+1)(1-\rho) \\
                       &\,\quad-\alpha\,\xi\,\rho^{\xi-1}(\alpha\,\rho-1)^2(1-\rho) \\
                       &\,\quad-\alpha(\alpha+\rho^\xi)(\alpha(1+\rho+(1-\rho)\sigma)
                               +\rho^\xi(1-\rho)(\sigma-1)+2\,\alpha\,\rho^{\xi+1}), \\
    G^\alpha_\xi(\rho):=&\,-\alpha_a(\alpha\,\rho-1)(\rho^{\xi+1}+1)(1-\rho) \\
                                &\,\quad+\xi\,\rho^{\xi-1}(\alpha\,\rho-1)^2(1-\rho) \\
                                &\,\quad-(\alpha+\rho^\xi)(\alpha(1-\rho)(1-\sigma)
                                  +\rho^\xi(1+\rho-(1-\rho)\sigma)+2).
  \end{align*}
  Now we calculate the desired limits for each $\alpha$-condition. \\
  
  \textit{Calculate the constant terms $C^\alpha_\xi(\rho)$ for the $\alpha_1$-condition.} \\
  Let $\xi>-1$:
  \begin{align}\label{eq:x1}
    \limr C^{\alpha_1}_{\xi>-1}(\rho) = \sgn(\sigma).
  \end{align}
  Let $\xi=-1$:
  \begin{align}\label{eq:x2}
    \limr C^{\alpha_1}_{\xi=-1}(\rho) = \sgn(\sigma)\cdot(-\sigma+2).
  \end{align}
  Let $\xi<-1$:
  \begin{align}
    \limr \rho^{-(\xi+1)}\cdot C^{\alpha_1}_{\xi<-1}(\rho) = \sgn(\sigma)\cdot(-\sigma+1).
  \end{align}
  
  \textit{Calculate the gradient terms $E^\alpha_\xi(\rho)$ and $G^\alpha_\xi(\rho)$ for the $\alpha_1$-condition.} \\
  Let $\xi>1$:
  \begin{align}
    \limr E^{\alpha_1}_{\xi>1}(\rho) =&\, -c^3(\sigma+1), \\
    \limr G^{\alpha_1}_{\xi>1}(\rho) =&\, c(c(\sigma-1)-2).
  \end{align} 
  Let $\xi=1$:
  \begin{align}
    \limr E^{\alpha_1}_{\xi=1}(\rho) =&\, -c(c^2(\sigma+1)+1), \\
    \limr G^{\alpha_1}_{\xi=1}(\rho) =&\, c(c(\sigma-1)-2)+1.
  \end{align}
  Let $0<\xi<1$:
  \begin{align}
    \limr \rho^{-(\xi-1)}\cdot E^{\alpha_1}_{0<\xi<1}(\rho) =&\, -c\,\xi, \\
    \limr \rho^{-(\xi-1)}\cdot G^{\alpha_1}_{0<\xi<1}(\rho) =&\, \xi.
  \end{align}
  Let $\xi=0$:
  \begin{align}
    \limr E^{\alpha_1}_{\xi=0}(\rho) =&\, -c(c+1)(c(\sigma+1)+\sigma-1), \\
    \limr G^{\alpha_1}_{\xi=0}(\rho) =&\, (c+1)(c(\sigma-1)+\sigma-3).
  \end{align}
  Let $-1<\xi<0$:
  \begin{align}
    \limr \rho^{-(\xi-1)}\cdot E^{\alpha_1}_{-1<\xi<0}(\rho) =&\, -c\,\xi, \\
    \limr \rho^{-(\xi-1)}\cdot G^{\alpha_1}_{-1<\xi<0}(\rho) =&\, \xi.
  \end{align}
  Let $\xi=-1$:
  \begin{align}
    \limr \rho^2\cdot E^{\alpha_1}_{\xi=-1}(\rho) =&\, -c(\sigma-2), \label{eq:x3} \\
    \limr \rho^2\cdot G^{\alpha_1}_{\xi=-1}(\rho) =&\, \sigma-2. \label{eq:x4}
  \end{align}
  Let $\xi<-1$:
  \begin{align}
    \limr \rho^{-2\,\xi}\cdot E^{\alpha_1}_{\xi<-1}(\rho) =&\, -c(\sigma-1), \\
    \limr \rho^{-2\,\xi}\cdot G^{\alpha_1}_{\xi<-1}(\rho) =&\, \sigma-1. \label{eq:x5}
  \end{align}
  
  \textit{Calculate the constant terms $C^\alpha_\xi(\rho)$ for the $\alpha_2$-condition.} \\
  Let $\xi>-1$:
  \begin{align}\label{eq:x6}
    \limr C^{\alpha_2}_{\xi>-1}(\rho) =&\, \sgn(\sigma)(-(c\,\sigma-1)).
  \end{align}
  Let $\xi=-1$:
  \begin{align}\label{eq:x9}
    \limr C^{\alpha_2}_{\xi=-1}(\rho) =&\, \sgn(\sigma)(-(c\,\sigma+\sigma-2)).
  \end{align}
  Let $\xi<-1$
  \begin{align}\label{eq:x12}
    \limr C^{\alpha_2}_{\xi<-1}(\rho) =&\, \sgn(\sigma)(-\sigma+1).
  \end{align}
  
  \textit{Calculate the gradient terms $G^\alpha_\xi(\rho)$ and $E^\alpha_\xi(\rho)$ for the $\alpha_2$-condition.} \\
  Let $\xi>-1$:
  \begin{align}
    \limr \rho^3\cdot E^{\alpha_2}_{\xi>-1}(\rho) =&\, -c^3(\sigma+1), \label{eq:x7} \\
    \limr \rho^2\cdot G^{\alpha_2}_{\xi>-1}(\rho) =&\, -(-c\,\sigma+1). \label{eq:x8}
  \end{align}
  Let $\xi=-1$:
  \begin{align}
    \limr \rho^3\cdot E^{\alpha_2}_{\xi=-1}(\rho) =&\, -c(c+1)(c\,\sigma+2\,c+\sigma), \label{eq:x10} \\
    \limr \rho^2\cdot G^{\alpha_2}_{\xi=-1}(\rho) =&\, (c+1)(c\,\sigma+\sigma-2). \label{eq:x11}
  \end{align}
  Let $\xi<-1$:
  \begin{align}
    \limr \rho^{-(2\,\xi-1)}\cdot E^{\alpha_2}_{\xi<-1}(\rho) =&\, -c(c+\sigma), \label{eq:x13} \\
    \limr \rho^{-2\,\xi}\cdot G^{\alpha_2}_{\xi<-1}(\rho) =&\, \sigma-1. \label{eq:x14}
  \end{align}
  
  \textit{Calculate the constant terms $C^\alpha_\xi(\rho)$ for the $\alpha_3$-condition.} \\
  Let $\xi>-2$:
  \begin{align}\label{eq:x15}
    \limr \rho\cdot C^{\alpha_3}_{\xi>-2}(\rho) =&\, c\,\sigma. 
  \end{align}
  Let $\xi=-2$:
  \begin{align}\label{eq:x22}
    \limr \rho\cdot C^{\alpha_3}_{\xi=-2}(\rho) =&\, c\,\sigma+c+\sigma-1. 
  \end{align}
  Let $\xi<-2$:
  \begin{align}\label{eq:x25}
    \limr \rho^{-(\xi+1)}\cdot C^{\alpha_3}_{\xi<-2}(\rho) =&\, c+\sigma-1.
  \end{align}
  
  \textit{Calculate the gradient terms $E^\alpha_\xi(\rho)$ and $G^\alpha_\xi(\rho)$ for the $\alpha_3$-conditions.} \\
  Let $\xi>-1$:
  \begin{align}
    \limr \rho^6\cdot E^{\alpha_3}_{\xi>-1}(\rho) =&\, -c^3(\sigma+1), \label{eq:x16} \\
    \limr \rho^4\cdot G^{\alpha_3}_{\xi>-1}(\rho) =&\, c^2(\sigma+1). \label{eq:x17}
  \end{align}
  Let $\xi=-1$:
  \begin{align}
    \limr \rho^6\cdot E^{\alpha_3}_{\xi=-1}(\rho) =&\, -c^3(\sigma+2), \label{eq:x18} \\
    \limr \rho^4\cdot G^{\alpha_3}_{\xi=-1}(\rho) =&\, c^2(\sigma+2). \label{eq:x19}
  \end{align}
  Let $-2<\xi<-1$:
  \begin{align}
    \limr \rho^{-(\xi-5)}\cdot E^{\alpha_3}_{-2<\xi<-1}(\rho) =&\, -c^3(\xi+2), \label{eq:x20} \\
    \limr \rho^{-(\xi-2)}\cdot G^{\alpha_3}_{-2<\xi<-1}(\rho) =&\, c^2(\xi+2). \label{eq:x21}
  \end{align}
  Let $\xi=-2$:
  \begin{align}
    \limr \rho^6\cdot E^{\alpha_3}_{\xi=-2}(\rho) =&\, -c(c^2(\sigma+3)+2\,c(\sigma+2)+\sigma+1+d), \label{eq:x23} \\
    \limr \rho^4\cdot G^{\alpha_3}_{\xi=-2}(\rho) =&\, c^2(\sigma+1)+2\,c\,\sigma+\sigma-1-c\,d. \label{eq:x24}
  \end{align}
  Let $-3<\xi<-2$:
  \begin{align}
    \limr \rho^{-(\xi-5)}\cdot E^{\alpha_3}_{-3<\xi<-2}(\rho) =&\, -c^3(\xi+2), \label{eq:x26} \\
    \limr \rho^{-(\xi-3)}\cdot G^{\alpha_3}_{-3<\xi<-2}(\rho) =&\, -c^3(\xi+2). \label{eq:x27}
  \end{align}
  Let $\xi=-3$:
  \begin{align}
    \limr \rho^8\cdot E^{\alpha_3}_{\xi=-3}(\rho) =&\, c(c^2-2\,c-\sigma-1-d), \label{eq:x28} \\
    \limr \rho^6\cdot G^{\alpha_3}_{\xi=-3}(\rho) =&\, -c^2+\sigma-1. \label{eq:x29}
  \end{align}
  Let $\xi<-3$:
  \begin{align}
    \limr \rho^{-(2\,\xi-2)}\cdot E^{\alpha_3}_{\xi<-3}(\rho) =&\, -c(2\,c+\sigma+1+d), \label{eq:x30} \\
    \limr \rho^{-2\,\xi}\cdot G^{\alpha_3}_{\xi<-3}(\rho) =&\, \sigma-1. \label{eq:x31}
  \end{align}
  
  \textbf{Conclusions for the limits of \textit{C}, \textit{E}, and \textit{G}.} \\
  
  \textit{Conclusions for the $\alpha_1$-condition in the contracting case, $\sigma>1$.} \\
  
  Let $\xi>-1$: \\
  By \eqref{eq:x1} we get the necessary condition
  \begin{align*}
    1 \leq 0,
  \end{align*}
  thus, here a MPF $w$ cannot exist. \\
  Let $\xi=-1$: \\
  By \eqref{eq:x2}, \eqref{eq:x3}, \eqref{eq:x4}, we get the necessary conditions
  \begin{align*}
    -\sigma+2\leq&\, 0, \\
    -\sigma+2\leq&\, 0, \\
    \sigma-2\leq&\, 0.
  \end{align*}  
  Thus, here a MPF $w$ cannot exist if $\sigma\neq 2$. \\
  Let $\xi<-1$: \\
  By \eqref{eq:x5}, we get the necessary condition
  \begin{align*}
    \sigma-1 \leq&\, 0,
  \end{align*}
  thus, here a MPF $w$ cannot exist. \\
  
  Therefore, a MPF $w$ can only exist if
  \begin{align}\label{eq:xI}
    \xi=-1, \text{ and } \sigma=2.
  \end{align}
  \newline
  \textit{Conclusions for the $\alpha_2$-condition in the contracting case, $\sigma>1$.} \\
  
  Let $\xi>-1$:
  By \eqref{eq:x6}, \eqref{eq:x7}, \eqref{eq:x8}, we get the necessary conditions
  \begin{align*}
    -(c\,\sigma-1) \leq&\, 0, \\
    -(\sigma+1) \leq&\, 0, \\
    c\,\sigma-1 \leq&\, 0.
  \end{align*}
  Thus, here a MPF $w$ cannot exist if $\sigma\neq \frac{1}{\sigma}$. \\
  Let $\xi=-1$: \\
  By \eqref{eq:x9}, \eqref{eq:x10}, \eqref{eq:x11}, we get the necessary conditions,
  \begin{align*}
    -(c\,\sigma+\sigma-2) \leq 0, \\
    -(c\,\sigma+2\,c+\sigma) \leq 0, \\
    c\,\sigma+\sigma-2 \leq 0.
  \end{align*}
  This implies $\frac{2-\sigma}{\sigma}=c>0$.
  Thus, here a MPF $w$ cannot exist if $\sigma\not\in(0,2)$. \\
  Let $\xi<-1$: \\
  By \eqref{eq:x14} we get the necessary condition,
  \begin{align*}
    \sigma-1 \leq 0,
  \end{align*}
  thus, here a MPF $w$ cannot exist. \\
  
  Therefore, a MPF $w$ can only exist if
  \begin{align}\label{eq:xII}
    \begin{split}
      \xi>&\, -1, \text{ and } c=\frac{1}{\sigma}, \\
      \xi=&\, -1, \text{ and } \sigma \in (0,2), \\
      \xi<&\, -1, \text{ MPF non-existent}.
    \end{split}
  \end{align}
  
  \textit{Conclusions for the $\alpha_2$-condition in the expanding case, $\sigma<0$.} \\
  
  Let $\xi>-1$: \\
  By \eqref{eq:x6}, \eqref{eq:x7}, \eqref{eq:x8}, we get the necessary conditions, 
  \begin{align*}
    c\,\sigma-1 \leq&\, 0, \\
    -(\sigma-1) \leq&\, 0, \\
    c\,\sigma-1 \leq&\, 0,
  \end{align*}
  thus, here MPF $w$ cannot exist if $\sigma\not\in [-1,0)$. \\
  Let $\xi=-1$: \\
  By \eqref{eq:x9}, \eqref{eq:x10}, \eqref{eq:x11}, we get the necessary conditions
  \begin{align*}
    c\,\sigma+\sigma-2 \leq&\, 0, \\
    -c\,\sigma-2\,c-\sigma \leq&\, 0 \\
    c\,\sigma+\sigma-2 \leq&\, 0,
  \end{align*}
  thus, here MPF cannot exist if $\sigma\not\in(-2,0)$. \\
  Let $\xi<-1$: \\
  By \eqref{eq:x12}, \eqref{eq:x13}, \eqref{eq:x14}, we get the necessary conditions
  \begin{align*}
    \sigma-1 \leq&\, 0, \\
    -(c+\sigma) \leq&\, 0, \\
    \sigma-1 \leq&\, 0,
  \end{align*}
  thus, here we cannot determine any further necessary conditions. \\
  
  Therefore, a MPF $w$ can only exist if
  \begin{align}
    \begin{split}\label{eq:xIII}
      \xi>&\, -1, \text{ and } \sigma\in[-1,0), \\
      \xi=&\, -1, \text{ and } \sigma\in(-2,0), \\
      \xi<&\, -1, \text{ no further necessary conditions}.
    \end{split}
  \end{align}
  
  \textit{Conclusions for the $\alpha_3$-condition in the contracting case, $\sigma<0$.} \\
  
  Let $\xi>-1$: \\
  By \eqref{eq:x15}, \eqref{eq:x16}, \eqref{eq:x17}, we get the necessary conclusions
  \begin{align*}
    c\,\sigma \leq&\, 0, \\
    -(\sigma+1) \leq&\, 0, \\
    (\sigma+1) \leq&\, 0,
  \end{align*}
  thus, here a MPF $w$ cannot exist if $\sigma\neq -1$. \\
  Let $\xi=-1$: \\
  By \eqref{eq:x15}, \eqref{eq:x18}, \eqref{eq:x19}, we get the necessary conditions
  \begin{align*}
    c\,\sigma \leq&\, 0, \\
    -(\sigma+2) \leq&\, 0, \\
    \sigma+2\leq&\, 0, 
  \end{align*}
  thus, here a MPF $w$ cannot exist if $\sigma\neq -2$. \\
  Let $-2<\xi<-1$: \\
  By \eqref{eq:x15}, \eqref{eq:x20}, \eqref{eq:x21}, we get the necessary conditions
  \begin{align*}
    c\,\sigma \leq&\, 0, \\
    -(\xi+2) \leq&\, 0, \\
    \xi+2 \leq&\, 0,
  \end{align*}
  thus, here a MPF $w$ cannot exist. \\
  Let $\xi=-2$: \\
  By \eqref{eq:x22}, \eqref{eq:x23}, \eqref{eq:x24}, we get the necessary conditions 
  \begin{align*}
    c\,\sigma+c+\sigma-1 \leq&\, 0, \\
    -(c^2(\sigma+3)+2\,c(\sigma+2)+\sigma+1+d) \leq&\, 0, \\
    c^2(\sigma+1)+2\,c\,\sigma+\sigma-1-c\,d \leq&\, 0,
  \end{align*}
  thus, here we cannot determine any further necessary conditions. \\
  Let $-3<\xi<-2$: \\
  By \eqref{eq:x25}, \eqref{eq:x26}, \eqref{eq:x27}, we get the necessary conditions
  \begin{align*}
    c+\sigma-1\leq&\, 0, \\
    -(\xi+2) \leq&\, 0, \\
    \xi+2 \leq&\, 0,
  \end{align*}
  thus, here a MPF $w$ cannot exist. \\
  Let $\xi=-3$: \\
  By \eqref{eq:x25}, \eqref{eq:x28}, \eqref{eq:x29}, we get the necessary conditions
  \begin{align*}
    c+\sigma-1\leq&\, 0, \\
    c^3-2\,c-\sigma-1-d\leq&\, 0, \\
    -c^2+\sigma-1\leq&\, 0,
  \end{align*} 
  thus, here we cannot determine any further necessary conditions. \\
  Let $\xi<-3$: \\
  By \eqref{eq:x25}, \eqref{eq:x30}, \eqref{eq:x31}, we get the necessary conditions
  \begin{align*}
    c+\sigma-1 \leq&\, 0, \\
    -2\,c-\sigma+1+d \leq&\, 0, \\
    \sigma-1 \leq&\, 0,
  \end{align*}
  thus, here we cannot determine any further necessary conditions. \\
  
  Therefore, a MPF $w$ can only exist if
  \begin{align}\label{eq:xIV}
    \begin{split}
      \xi>&\, -1, \text{ and } \sigma=-1, \\
      \xi=&\, -1, \text{ and } \sigma=-2, \\
      -2<&\,\xi<-1, \text{ MPF non-existent}, \\
      \xi=&\,-2, \text{ no further necessary conditions}, \\
      -3<&\,\xi<-2: \text{ MPF non-existent}, \\
      \xi=&\,-3, \text{ no further necessary conditions}, \\
      \xi<&\,-3, \text{ no further ncessary conditions}.
    \end{split}
  \end{align}
  
  As mentioned in the beginning of the proof, we need to add one to $\xi$
  to obtain the proper results for $F^\sigma_\xi$ since $\beta_{F^\sigma_\xi} = \rho^{\xi-1}$. \\
  
  \textbf{Conclusions in the contracting case, and in the expanding case.} \\
  Consider $F^\sigma_\xi$, where $\sigma>1$. 
  MPF $w$ can only exist in the contracting case if the $\alpha_1$-condition, or if the corresponding $\alpha_2$-condition is met.
  Combining \eqref{eq:xI} and \eqref{eq:xII}, 
  we obtain the necessary conditions 
  \begin{align*}
    \xi>&\, 0:\;c=\frac{1}{\sigma}, \\
    \xi=&\, 0:\;\sigma\in(0,2], \\
    \xi<&\, 0:\;\text{MPF non-existent}.
  \end{align*}
  Consider $F^\sigma_\xi$, where $\sigma<0$. 
  MPF $w$ can only exist in the expanding case if the corresponding $\alpha_2$-condition, or if the $\alpha_3$-condition is met.
  Combining \eqref{eq:xIII} and \eqref{eq:xIV},
  we obtain the necessary conditions
  \begin{align*}
    \xi>&\, 0:\;\sigma\in[-1,0), \\
    \xi=&\, 0:\;\sigma\in[-2,0), \\
    \xi<&\, 0:\;\text{no further necessary conditions}. 
  \end{align*}
  This concludes the proof.
\end{proof}

\section{Vanishing functions}

G. Huisken \cite{gh:flow} uses a function similar to
\begin{align*}
  v=\frac{(a-b)^2}{(a+b)^2}
\end{align*}
to show convergence to a round point.
Here, $v$ fulfills all MPF conditions except condition \eqref{II},
since $v$ is homogeneous of degree zero.
But it is a vanishing function,
\ie a function whose constant terms $C_v(a,b)$ of 
MPF condition \eqref{IV} vanish for all $0<a,b$.
B. Andrews \cite{ba:gauss} uses the MPF
\begin{align*}
  v=(a-b)^2
\end{align*}
to show convergence to a round point for $F=K$,
and O. Schn\"urer and F. Schulze \cite{fs:convexity} 
use the MPF 
\begin{align}\label{id:MPF Schnulze}
  v=\frac{(a-b)^2(a+b)^{2\sigma}}{(a\,b)^2}
\end{align}
to show a corresponding result for $F=H^\sigma$ with $1<\sigma\leq 5$.
All of these functions are vanishing functions.
In fact, vanishing functions seem to be the natural set of functions,
when it comes to MPF for contracting normal velocities $F^\sigma_\xi$
with $\xi\geq 0$.
We conjecture that, if a vanishing function is not MPF
for a contracting normal velocity $F^\sigma_\xi$ with $\xi\geq 0$, then there just exist 
no MPF for that $F^\sigma_\xi$. 

For other contracting and expanding $F^\sigma_\xi$, 
this is different. Here,
vanishing functions do not play such an important role.
By Lemma \ref{lem:alpha o estimates}, we know that
vanishing functions cannot be MPF in most cases.
We see this by interchanging the $\leq$ and $\geq$
by equal signs in the $o$-estimates of this Lemma.
That is basically what we do in Lemma \ref{lem:alpha vanishing V}. 

However, for contracting normal velocities $F^\sigma_\xi$ with $\xi\geq 0$,
vanishing functions remain interesting.
In later chapters, the MPF tables for normal velocities 
$F^\sigma_0=\sgn(\sigma)\cdot K^{\sigma/2}$,
$F^\sigma_1=\sgn(\sigma)\cdot H^\sigma$,
$F^\sigma_2=\sgn(\sigma)\cdot |A|^\sigma$,
and 
$F^\sigma_\sigma=\sgn(\sigma)\cdot \tr A^\sigma$,
all contain vanishing functions.
In Lemma \ref{lem:alpha vanishing VI} we see that in some sense 
they are the optimal fit.

In Remark \ref{rem:not vanishing}, we present
some MPF for $F=H^\sigma$ with $\sigma=3,4$
by O. Schn\"urer \cite{os:surfacesA2}, and with $\sigma=2,5$,
which are not vanishing functions.
We want to stress that vanishing functions are useful
but not the only existing MPF. 
The rest of the somewhat technical Lemmas 
on vanishing functions
is needed in later chapters.

\begin{definition}[Vanishing functions]\label{def:vanishing}
  Let $a$ and $b$ be principal curvatures.
  Let $v(a,b)\in C^2\left(\R^2_{+}\right)$ with $v\not\equiv 0$.
  Let $C_w(a,b)$ be defined as in MPF condition \eqref{IV}. 
  We call $v$ a \textit{vanishing function}
  for a normal velocity $F$ if
  $C_v(a,b)=0$
  for all $0<a,b$.
\end{definition}

\begin{example}[Vanishing functions]\label{exm:vanishing}
  By O. Schn\"urer and F. Schulze \cite{fs:convexity},
  and B. Andrews and X. Chen \cite{ac:surfaces},
  we have the following example of a 
  vanishing function for a normal velocity $F$
  \begin{align*}
    v=\frac{(a-b)^2\,F^2}{(a\,b)^2}.
  \end{align*}
  Note that $F$ can be an arbitrary normal velocity.
\end{example}

\begin{remark}[MPF which are not vanishing functions]\label{rem:not vanishing}
  We give examples of MPF which are not vanishing functions.
  Let $F=F^\sigma_1=\sgn(\sigma)\cdot H^\sigma$. Then we have 
  \begin{center}
    \begin{tabular}{| l | c || c | c |}
      \hline 
      & MPF \, $w_{H^\sigma}$ & &  \\
      \hline \hline  
      & & & \\
      $\Bigg.\Bigg.$ $\sigma=2$ & $\frac{(a-b)^2\left(a^2+4a\,b+b^2\right)}{(a+b)(a\,b)}$ & & \\
      \hline
      & & & \\
      $\Bigg.\Bigg.$ $\sigma=3$ & $\frac{(a-b)^2(a+b)^2\left(a^2+2a\,b+b^2\right)}{\left(a^2-a\,b+b^2\right)(a\,b)}$ & O. Schn\"urer & \cite{os:surfacesA2} \\ 
      \hline \hline 
      & & & \\
      $\Bigg.\Bigg.$ $\sigma=4$ & $\frac{(a-b)^2(a+b)^6\left(a^2+a\,b+b^2\right)}{(a\,b)^2}$ & O. Schn\"urer & \cite{os:surfacesA2} \\ 
      \hline
      & & & \\
      $\Bigg.\Bigg.$ $\sigma=5$ & $\frac{(a-b)^2(a+b)^2\left(16(a+b)^8-(a\,b)^4\right)}{(a\,b)^2}$ & & \\
      \hline  
    \end{tabular}
  \end{center}  
\end{remark}

\begin{lemma}[Quantity $\alpha$ of a vanishing function I]\label{lem:alpha vanishing I}
  Let $v$ be a vanishing function for a normal velocity $F$.
  Then we have 
  \begin{align}\label{id:alphavF}
    \alpha_{v,F}(a,b) = \frac{b\left(b\,F-a(-a+b)F_a\right)}{a\left(a\,F+b(-a+b)F_b\right)}.
  \end{align}
  Set $(a,b)=(\rho,1)$, where $0<\rho<1$.
  Furthermore, we have
  \begin{align}\label{id:alphavFrho}
    \alpha_{v,F}(\rho) = \frac{F-\rho(1-\rho)F_a}{\rho\left(\rho\,F+(1-\rho)F_b\right)}.
  \end{align}
  Note that $\alpha_{v,F}(a,b)$ is unique for all vanishing functions $v$
  for a normal velocity $F$.  
\end{lemma}
\begin{proof}
  Let $v$ be a vanishing function for a normal velocity $F$.
  Then we have $C_v(a,b)=0$ for all $0<a,b$.
  Using $\alpha_{v,F}:=-\frac{v_a}{v_b}$ yields identity \eqref{id:alphavF}.
  Now we set $(a,b)=(\rho,1)$, where $0<\rho<1$, and obtain identity \eqref{id:alphavFrho}.
  This concludes the proof.
\end{proof}

\begin{lemma}[Quantity $\alpha$ of a vanishing function II]\label{lem:alpha vanishing II}
  Let $v$ be a vanishing function for a normal velocity $F$.
  Set $(a,b)=(\rho,1)$, where $0<\rho<1$.
  Then we have
  \begin{align}\label{id:alphavFbeta}
    \alpha_{v,\beta,\sigma}(\rho) = \frac{1+\beta\,\rho(1-(1-\rho)\sigma)}{\rho((1-\rho)\sigma+\rho(\beta\,\rho+1))}
  \end{align}
\end{lemma}
\begin{proof}
  Let $v$ be a vanishing function for a normal velocity $F$.
  Set $(a,b)=(\rho,1)$, where $0<\rho<1$.
  Let $C^\alpha_\beta(\rho)$ be defined as in Lemma \ref{lem:necessary}.
  For vanishing function $v$, we get $C^\alpha_\beta(\rho)=0$ for all $0<\rho<1$.
  This implies identity \eqref{id:alphavFbeta}. 
\end{proof}

\begin{lemma}[Quantity $\alpha$ of a vanishing function III]\label{lem:alpha vanishing III}
  Let $F^\sigma_\xi$ be defined as in Remark \ref{rem:normal velocity xi}.
  Let $v$ be a vanishing function for a contracting normal velocity $F^\sigma_\xi$, \ie $\sigma>0$.
  Set $(a,b)=(\rho,1)$, where $0<\rho<1$.
  Let $\alpha_{v,\beta,\sigma}(\rho)$ be defined as in Lemma \ref{lem:alpha vanishing II}.
  Then $\alpha_{v,\beta,\sigma}(\rho)$ is strictly decreasing in $\sigma$ 
  for all $0<\rho<1$.
\end{lemma}
\begin{proof}
  Let $\sigma,\,\psi>0$.
  Let $\beta=\rho^{\xi-1}$.
  We calculate $\alpha_{v,\beta,\sigma}(\rho)$ and $\alpha_{v,\beta,\sigma+\psi}(\rho)$, 
  and we subtract the second term from the first term.
  We obtain
  \begin{align*}
    0<\frac{(1-\rho)(1+\rho^{\xi})(1+\rho^{1+\xi})\psi}{\rho\left(\rho+\rho^{1+\xi}+(1-\rho)\sigma\right)
    \left(\rho+\rho^{1+\xi}+(1-\rho)(\sigma+\psi)\right)}
  \end{align*}
  for all $0<\rho<1$. This concludes the proof.
\end{proof}

\begin{lemma}[Quantity $\alpha$ of a vanishing function IV]\label{lem:alpha vanishing IV}
  Let $F^\sigma_\xi$ be defined as in Remark \ref{rem:normal velocity xi}.
  Let $v$ be a vanishing function for a contracting normal velocity $F^\sigma_\xi$ with $\sigma>1$.
  Set $(a,b)=(\rho,1)$, where $0<\rho<1$.
  Let $\alpha_{v,\beta,\sigma}(\rho)$ be defined as in Lemma \ref{lem:alpha vanishing II}.
  Then we have
  \begin{align*}
    \alpha_{v,\beta,\sigma}(\rho)<\frac{1}{\rho}
  \end{align*}
  for all $0<\rho<1$.
\end{lemma}
\begin{proof}
  By Lemma \ref{lem:alpha vanishing III}, $\alpha_{v,\beta,\sigma}(\rho)$ is strictly decreasing in $\sigma$
  for all $0<\rho<1$ and all $\sigma>0$.
  We get
  \begin{align*}
    \alpha_{v,\beta,1}(\rho)=\frac{1}{\rho}
  \end{align*}
  for all $0<\rho<1$. Therefore, the claim follows.
\end{proof}

\begin{figure}[t] 
  \vspace*{2mm} 
  \center\includegraphics[width=8.3cm]{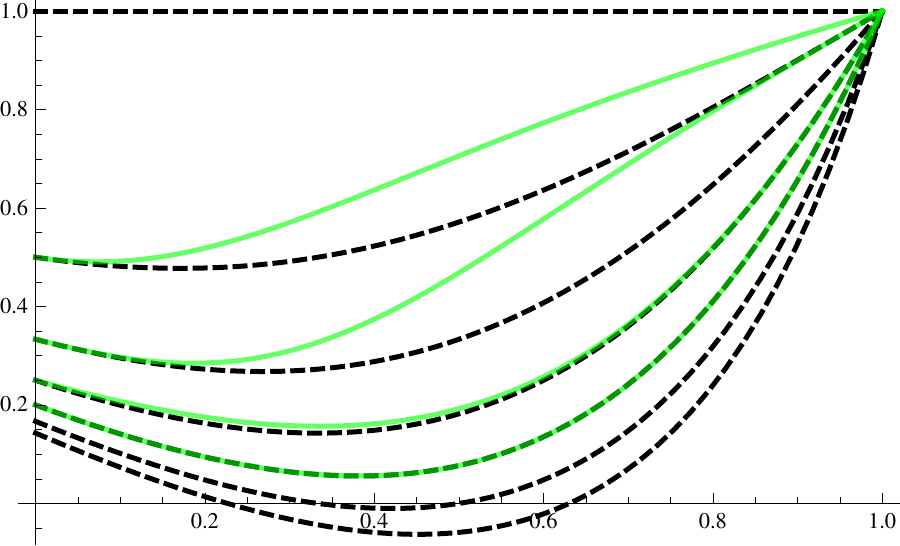} 
  \caption{
  Let $v_{H^\sigma}$ be vanishing functions as in \eqref{id:MPF Schnulze}, and let $w_{H^\sigma}$ be MPF as defined as in Remark \ref{rem:not vanishing}, which are not vanishing functions.
  Set $(a,b)=(\rho,1)$, where $0\leq\rho\leq 1$.
  We plot the quantity $\rho\cdot \alpha$ for $v_{H^\sigma}$ with $\sigma=1,\ldots,7$, (dashed lines), and for $w_{H^\sigma}$ with $\sigma=2,\ldots,5$, on the interval $[0,1]$. 
  Note that we have $\limr \rho\cdot \alpha_{w,H^\sigma} = \limr \rho\cdot \alpha_{v,H^\sigma} = \frac{1}{\sigma}$.
  }
\end{figure} 

\begin{lemma}[Quantity $\alpha$ of a vanishing function V]\label{lem:alpha vanishing V}
  Let $F^\sigma_\xi$ be defined as in Remark \ref{rem:normal velocity xi}.
  Let $v$ be a vanishing function for a contracting normal velocity $F^\sigma_\xi$.
  Set $(a,b)=(\rho,1)$, where $0<\rho<1$.
  Let $\alpha_{v,\beta,\sigma}(\rho)$ be defined as in Lemma \ref{lem:alpha vanishing II}.
  Then we have for $\rho\rightarrow 0$
  \begin{align*}
    \alpha_{v,\beta,\sigma}(\rho) = \frac{1}{\sigma\,\rho} + o\left(\rho^{-1}\right),&\, \text{ if } \xi > 0, \\
    \alpha_{v,\beta,\sigma}(\rho) = \frac{2-\sigma}{\sigma\,\rho} + o\left(\rho^{-1}\right),&\, \text{ if } \xi = 0,\,\sigma\neq2, \\
    \alpha_{v,\beta,\sigma}(\rho) = 1 + o\left(\rho^{-1}\right),&\, \text{ if } \xi = 0,\,\sigma=2.
  \end{align*}
  Furthermore, we have
  \begin{align*}
    \alpha_{v,\beta,\sigma}(1) = 1
  \end{align*}
  for all $\sigma>0$.
\end{lemma}
\begin{proof}
  Let $\alpha_{v,\beta,\sigma}(\rho)$ be defined as in Lemma \ref{lem:alpha vanishing II}.
  The claim follows by direct calculations.
\end{proof}

\begin{lemma}[Quantity $\alpha$ of a vanishing function VI]\label{lem:alpha vanishing VI}
  Let $v,\,w$ be a vanishing and a maximum-principle function, respectively,
  for a contracting normal velocity $F$.
  Set $(a,b)=(\rho,1)$, where $0<\rho<1$.
  Then we have 
  \begin{align*}
    \alpha_{v,F}(\rho) \leq \alpha_{w,F}(\rho)
  \end{align*}
  for all $0<\rho<1$.
\end{lemma}
\begin{proof}
  Let $w$ be a MPF for a normal velocity $F$. Then we have $C_w(a,b)\leq 0$ for all $0<a,b$. 
  Set $(a,b)=(\rho,1)$, where $0<\rho<1$. Using $\alpha_{w,F} := -\frac{w_a}{w_b}$ yields
  \begin{align*}
    \frac{F-\rho(1-\rho)F_a}{\rho\left(\rho\,F+(1-\rho)F_b\right)} \leq \alpha_{w,F}.
  \end{align*}
  By Lemma \ref{lem:alpha vanishing I}, the left term equals $\alpha_{v,F}$.
  Hence, the claim follows.
  Note that the denominator of $\alpha_{v,F}$ is strictly positive for all $0<\rho<1$.
\end{proof}

\begin{lemma}[Quantity $\alpha$ of a vanishing function VII]\label{lem:alpha vanishing VII}
  Let $v$ be a vanishing function for some $F=H^\sigma$ with $\sigma>1$,
  \ie $\beta(a,b)=1$ for all $0<a,b$.
  Set $(a,b)=(\rho,1)$, where $0<\rho<1$.
  Let $\alpha_{v,\beta,\sigma}(\rho)$ be defined as in Lemma \ref{lem:alpha vanishing II}.
  Then we have 
  \begin{align*}
    \alpha_{v,1,\sigma}(\rho)=\frac{1+\rho-(1-\rho)\rho\,\sigma}{\rho\left(\rho+\rho^2+(1-\rho)\sigma\right)}
  \end{align*}
  Furthermore, $\alpha_{v,1,\sigma}(\rho)$ has these roots
  \begin{align*}
    \rho_{\mp} = \frac{-1+\sigma\mp \sqrt{1-6\,\sigma+\sigma^2}}{2\,\sigma}.
  \end{align*}
  In particular, we have
  \begin{align*}
    \rho_{-} \in \left(0,-1+\sqrt{2}\right], 
    \quad \rho_{+} \in&\, \left[-1+\sqrt{2},1\right),
    \quad \text{for all } \sigma\geq 3+\sqrt{2}, \\
    \rho_{-} \in \left(0,1/3\right], 
    \quad \rho_{+} \in&\, \left[1/2,1\right),
    \quad \text{for all } \sigma\geq 6,
  \end{align*}
  where $-1+\sqrt{2}\approx 0.414$, and $3+\sqrt{2} \approx 5.828$. 
  
  Note that $\alpha_{v,1,3+\sqrt{3}}(\rho)$ 
  has one root $\rho_{\mp}=-1+\sqrt{2}$, and
  $\alpha_{v,1,6}(\rho)$ has two roots $\rho_{-}=1/3$ and $\rho_{+}=1/2$.
\end{lemma}
\begin{proof}
  We set $\beta=1$ in $\alpha_{v,\beta,\sigma}(\rho)$ from Lemma \ref{lem:alpha vanishing II}
  to obtain $\alpha_{v,1,\sigma}(\rho)$. 
  
  Note that the denominator of $\alpha_{v,1,\sigma}(\rho)$ is strictly positive 
  for all $0<\rho<1$.
  Since the numerator is a quadratic equation in $\rho$, we obtain the roots 
  \begin{align*}
    \rho_{\mp} = \frac{-1+\sigma\mp \sqrt{1-6\,\sigma+\sigma^2}}{2\,\sigma}.
  \end{align*}
  The discriminant is also a quadratic equation, this time in $\sigma$.
  Thus, $\alpha_{v,1,\sigma}(\rho)$ has 
  \begin{itemize}
    \item no roots, if $3-2\,\sqrt{2} < \sigma < 3+\sqrt{2}$, 
    \item one root, if $\sigma=3-2\,\sqrt{2}$ or if $\sigma = 3+\sqrt{2}$,
    \item two roots, if $\sigma<3-2\,\sqrt{2}$ or if $3+\sqrt{2}<\sigma$,
  \end{itemize}
  where $3-2\,\sqrt{2} \approx 0.172$, and $3+\sqrt{2} \approx 5.828$. 
  We obtain the roots for $\sigma=3+\sqrt{2}$ and for $\sigma=6$ by direct calculations. 
  
  By Lemma \ref{lem:alpha vanishing V}, $\alpha_{v,1,\sigma}(\rho)$ is 
  strictly positive in some zero neighborhood of $\rho$, and 
  $\alpha_{v,1,\sigma}(1)=1$.
  And by Lemma \ref{lem:alpha vanishing III}, $\alpha_{v,1,\sigma}(\rho)$ is 
  strictly decreasing in $\sigma$ for all $0<\rho<1$.
  Therefore, we get $\rho_{-} \in \left(0,-1+\sqrt{2}\right]$ and $\rho_{+}\in[-1+\sqrt{2},1)$
  for all $\sigma\geq 3+\sqrt{2}$, and 
  $\rho_{-} \in \left(0,1/3\right]$ and $\rho_{+} \in \left[1/2,1\right)$
  for all $\sigma\geq 6$. This concludes the proof.
\end{proof}

\begin{lemma}[Quantity $\alpha$ of a vanishing function VIII]\label{lem:alpha vanishing VIII}
  Let $v$ be a vanishing function for some $F=|A|^\sigma$ with $\sigma>1$,
  \ie $\beta(a,b)=\frac{a}{b}$ for all $0<a,b$.
  Set $(a,b)=(\rho,1)$, where $0<\rho<1$.
  Let $\alpha_{v,\beta,\sigma}(\rho)$ be defined as in Lemma \ref{lem:alpha vanishing II}.
  Then we have 
  \begin{align*}
    \alpha_{v,\rho,\sigma}(\rho)=\frac{1+\rho^2-(1-\rho)\rho^2\,\sigma}{\rho\left(\rho+\rho^3+(1-\rho)\sigma\right)}
  \end{align*}
  Furthermore, $\alpha_{v,\rho,\sigma}(\rho)$ has these roots
  \begin{align*}
    \rho_0 =&\, \frac{1}{3\sigma} \left( 
                                            \sigma-1 
                                           -\frac{1+\sqrt{3}\,\imath}{2^{2/3}}(\sigma-1)^2\frac{1}{\varrho} 
                                           -\frac{1-\sqrt{3}\,\imath}{2\cdot 2^{1/3}}\varrho
                                      \right), \umbruch \\
    \rho_{-} =&\, \frac{1}{3\sigma} \left(
                                            \sigma-1
                                           -\frac{1-\sqrt{3}\,\imath}{2^{2/3}}(\sigma-1)^2\frac{1}{\varrho}
                                           -\frac{1+\sqrt{3}\,\imath}{2\cdot 2^{1/3}}\varrho
                                      \right), \umbruch \\
    \rho_{+} =&\, \frac{1}{3\sigma} \left(
                                           \sigma-1
                                          +\frac{2}{2^{2/3}}(\sigma-1)^2\frac{1}{\varrho}
                                          +\frac{1}{2^{1/3}}\varrho
                                      \right), \umbruch \\
             &\Big. \Big. \qquad \text{where} \\
    \varrho =&\, \left( -2+6\sigma-33\sigma^2+2\sigma^3+3\sqrt{3}\sigma\sqrt{4-12\sigma+39\sigma^2-4\sigma^3} \right)^{1/3}. 
  \end{align*}

  In particular, we have
  \begin{align*}
    \rho_{-} \in \left(0,\rho_\star\right], 
    \quad \rho_{+} \in&\, \left[\rho_\star,1\right),
    \quad \text{for all } \sigma\geq \sigma_\star, \umbruch \\
    \rho_{-} \in \left(0,1/2\right], 
    \quad \rho_{+} \in&\, \left[1/5\left(1+\sqrt{6}\right),1\right),
    \quad \text{for all } \sigma\geq 10,
  \end{align*}
  where $\rho_\star\approx 0.596$, 
  $\sigma_{\star} := \frac{1}{12} \left( 39+ \left( 51759-5832\sqrt{2}\right)^{1/3} + 8\left( 71+8\sqrt{2} \right)^{1/3} \right) \approx 9.444$,
  $1/5\left(1+\sqrt{6}\right) \approx 0.69$. 
  
  Note that $\alpha_{v,\rho,\sigma_\star}(\rho)$ has one root $\rho_{\mp}=\rho_\star$, 
  and $\alpha_{v,\rho,10}(\rho)$ has two roots $\rho_{-}=0.5$, and $\rho_{+}=1/5\left(1+\sqrt{6}\right)$. 
  
  Some results were obtained using a compute algebra program.
\end{lemma}
\begin{proof}
  We set $\beta=\rho$ in $\alpha_{v,\beta,\sigma}(\rho)$ from Lemma \ref{lem:alpha vanishing II}
  to obtain $\alpha_{v,\rho,\sigma}(\rho)$. 
  
  Note that the denominator of $\alpha_{v,\rho,\sigma}(\rho)$ is strictly positive
  for all $0<\rho<1$.
  Since the numerator is a cubic equation in $\rho$, we obtain the three roots
  $\rho_0$, $\rho_{-}$, and $\rho_{+}$.
  The radicand in $\varrho$ is also a cubic equation, this time in $\sigma$.
  Thus, $\alpha_{v,\rho,\sigma}(\rho)$ has
  \begin{itemize}
    \item no roots, if $1<\sigma<\sigma_{\star}$, 
    \item one root, if $\sigma=\sigma_{\star}$,
    \item two roots, if $\sigma_{\star} < \sigma$.
  \end{itemize}
  where $\sigma_{\star} := \frac{1}{12} \left( 39+ \left( 51759-5832\sqrt{2}\right)^{1/3} + 8\left( 71+8\sqrt{2} \right)^{1/3} \right) \approx 9.444$. 
  
  $\alpha_{v,\rho,\sigma_\star}(\rho)$ has one root $\rho_\star\approx 0.596$ for $0<\rho<1$. 
  
  All of the above computations are obtained using a computer algebra program. 
  
  We rewrite by hand 
  \begin{align*}
    \alpha_{v,\rho,10}(\rho) = \left(2\rho-1\right)\left(5\rho-\left(1-\sqrt{6}\right)\right)\left(\rho-1/5\left(1+\sqrt{6}\right)\right).
  \end{align*}
  obtaining the roots
  \begin{align*}
    \rho_0 =&\, \frac{1}{5}\left( 1-\sqrt{6}\right) \approx -0.29, \\
    \rho_{-} =&\, \frac{1}{2} = 0.5, \\
    \rho_{+} =&\, \frac{1}{5}\left( 1+\sqrt{6} \right) \approx 0.69.
  \end{align*}  
  
  By Lemma \ref{lem:alpha vanishing V}, $\alpha_{v,\rho,\sigma}(\rho)$ is 
  strictly positive in some zero neighborhood of $\rho$, and 
  $\alpha_{v,\rho,\sigma}(1)=1$.
  And by Lemma \ref{lem:alpha vanishing III}, $\alpha_{v,\rho,\sigma}(\rho)$ is 
  strictly decreasing in $\sigma$ for all $0<\rho<1$.
  Therefore, we get $\rho_{-} \in \left(0,\rho_\star\right]$ and $\rho_{+}\in\left[\rho_\star,1\right)$
  for all $\sigma\geq \sigma_\star$, and 
  $\rho_{-} \in \left(0,1/2\right]$ and $\rho_{+} \in \left[1/5\left(1+\sqrt{6}\right),1\right)$
  for all $\sigma\geq 10$. This concludes the proof.
\end{proof}

\section{Gauss curvature, $F^\sigma_0$}

In this chapter, we discuss $F^\sigma_0=K^{\sigma/2}$ for all $\sigma\in R\setminus\{0\}$. 

For any $\sigma<-2$ and any $\sigma>2$, we get the non-existence of MPF.
This is a direct consequence of the necessary conditions in Theorem \ref{lem:necessary}. 
For any $0<\sigma\leq 1$, we also get the non-existence of MPF.
This result is implied by Lemma \ref{lem:one}, and by Lemma \ref{lem:zeroone}. 
For all other powers $\sigma$, we state MPF by 
B. Andrews \cite{ba:gauss}, B. Andrews and X. Chen \cite{ac:surfaces}, Q. Li \cite{ql:surfaces}, 
and O. Schn\"urer \cite{os:surfaces}. 

We make a summary of all results on $F^\sigma_0=\sgn(\sigma)\cdot K^{\sigma/2}$
in Corollary \ref{cor:gauss curvature}. 

We should point out that C. Gerhardt \cite{cg:non} proves convergence to a sphere at $\infty$
for all $\sigma<0$ despite the fact that there exist no MPF for all $\sigma<-2$.
We discuss this further in chapter 'Outlook'. 

\begin{theorem}[Gauss curvature I]\label{thm:gauss curvature I}
  Let $F=F^\sigma_0=\sgn(\sigma)\cdot K^{\sigma/2}$. 
  Then there are no $MPF$, if $\sigma<-2$.
\end{theorem}
\begin{proof}
  Let $F^\sigma_\xi$ be defined as in Remark $\ref{rem:normal velocity xi}$.
  For powers of the Gauss curvature, we get 
  $F^\sigma_0=\sgn(\sigma)\cdot K^{\sigma/2}$.
  By Theorem \ref{thm:necessary}, we have the necessary condition $\sigma\in[-2,0)$
  for the existence of MPF in case of expanding normal velocities, \ie $\sigma<0$.
  Hence, the claim follows.
\end{proof}

\begin{theorem}[Gauss curvature II]\label{thm:gauss curvature II}
  Let $F=F^\sigma_0=\sgn(\sigma)\cdot K^{\sigma/2}$. 
  Then there are no $MPF$, if $0<\sigma\leq 1$.
\end{theorem}
\begin{proof}
  This is a direct consequence of Lemma \ref{lem:one} and Lemma \ref{lem:zeroone}.
\end{proof}

\begin{theorem}[Gauss curvature III]\label{thm:gauss curvature III}
  Let $F=F^\sigma_0=\sgn(\sigma)\cdot K^{\sigma/2}$. 
  Then there are no $MPF$, if $\sigma>2$.
\end{theorem}
\begin{proof}
  Let $F^\sigma_\xi$ be defined as in Remark $\ref{rem:normal velocity xi}$.
  For powers of the Gauss curvature, we get 
  $F^\sigma_0=\sgn(\sigma)\cdot K^{\sigma/2}$.
  By Theorem \ref{thm:necessary}, we have the necessary condition $\sigma\in(1,0]$
  for the existence of MPF in case of contracting normal velocities, \ie $\sigma>0$.
  Hence, the claim follows.
\end{proof}

\newpage

\begin{corollary}[Gauss curvature]\label{cor:gauss curvature}
  Let $F=F^\sigma_0=\sgn(\sigma)\cdot K^{\sigma/2}$. 
  Then we have
  \begin{center}
    \begin{tabular}{| l | c || c | c |}
      \hline 
      & MPF & & \\
      \hline \hline  
      & & & \\                
      $\sigma>2$ & non-existent & & Theorem \ref{thm:gauss curvature III} \\
      \hline
      & & & \\
      $\sigma=2$ & $(a-b)^2$ & B. Andrews & \cite{ba:gauss} \\
      \hline
      & & & \\
      $\sigma\in(1,2)$ & $\frac{(a-b)^2(a\,b)^\sigma}{(a\,b)^2}$ & B. Andrews, X. Chen & \cite{ac:surfaces} \\
      \hline
      & & & \\
      $\sigma\in(0,1]$ & non-existent & & Theorem \ref{thm:gauss curvature II} \\
      \hline \hline 
      & & & \\
      $\sigma\in(-2,0)$ & $\frac{(a-b)^2(a\,b)^{\sigma/2}}{(a\,b)}$ & Q. Li & \cite{ql:surfaces} \\
      \hline
      & & & \\
      $\sigma=-2$ & $\frac{(a-b)^2}{(a\,b)^2}$ & O. Schn\"urer & \cite{os:surfaces} \\
      \hline
      & & & \\
      $\sigma<-2$ & non-existent & & Theorem \ref{thm:gauss curvature I} \\
      \hline  
    \end{tabular}
  \end{center}
\end{corollary}

\section{Mean curvature, $F^\sigma_1$}

In this chapter, we discuss $F^\sigma_1=\sgn(\sigma)\cdot H^\sigma$
for all $\sigma\in\R\setminus\{0\}$. 

For $\sigma\geq 5.98$, we show the non-existence of MPF. 
This is Theorem \ref{thm:mean curvature III}.
We show this by comparing an assumed MPF to a vanishing function.
Since vanishing functions develop roots for any
$\sigma\geq 3+\sqrt{2} \approx 5.828$,
we obtain a contradiction.
Note that the quantity $\alpha$ has to be strictly positive
for a MPF, see Lemma \ref{lem:alpha}. 

Our proof technique does not work for $5.17<\sigma\leq 5.98$
since the quantity $\alpha$ of vanishing functions
does not develop any roots there. 

For $5.89\leq\sigma<5.98$, $\alpha$ has roots but
our technique still does not work, since it involves 
some estimates which apparently are not sharp. 

For any $1<\sigma\leq 5.17$, we have a MPF by 
O. Schn\"urer and F. Schulze \cite{fs:convexity}, 
which is also a vanishing function. 

By Theorem \ref{thm:mean curvature II}, we get
the non-existence of MPF for any $0<\sigma\leq 1$. 

In Lemma \ref{lem:H expanding}, we give a MPF, 
which we present fully in a separate paper. 
This is $-1<\sigma<0$. 

And for $\sigma=-1$, we have another MPF by O. Schn\"urer. 

For any $\sigma<-1$, the non-existence of MPF is shown 
in Theorem \ref{thm:mean curvature I}.
We prove this using the necessary conditions for $F^\sigma_\xi$,
which we obtained in Theorem \ref{thm:necessary}. 

We summarize our results for $F^\sigma_1=\sgn(\sigma)\cdot H^\sigma$
in the table of Corollary \ref{cor:mean curvature}.

\begin{theorem}[Mean curvature I]\label{thm:mean curvature I}
  Let $F=F^\sigma_1=\sgn(\sigma)\cdot H^\sigma$. 
  Then there are no $MPF$, if $\sigma<-1$.
\end{theorem}
\begin{proof}
  Let $F^\sigma_\xi$ be defined as in Remark $\ref{rem:normal velocity xi}$.
  For powers of the mean curvature, we get 
  $F^\sigma_1=\sgn(\sigma)\cdot H^\sigma$.
  By Theorem \ref{thm:necessary}, we have the necessary condition $\sigma\in[-1,0)$
  for the existence of MPF in case of expanding normal velocities, \ie $\sigma<0$.
  Hence, the claim follows.
\end{proof}

\begin{theorem}[Mean curvature II]\label{thm:mean curvature II}
  Let $F=F^\sigma_1=\sgn(\sigma)\cdot H^\sigma$. 
  Then there are no $MPF$, if if $0<\sigma\leq 1$.
\end{theorem}
\begin{proof}
  This is a direct consequence of Lemma \ref{lem:one} and Lemma \ref{lem:zeroone}.
\end{proof}

\begin{theorem}[Mean curvature III]\label{thm:mean curvature III}
  Let $F=F^\sigma_1=\sgn(\sigma)\cdot H^\sigma$. 
  Then there are no $MPF$, if $\sigma\geq \sigma_\delta$ with $\sigma_\delta=7$. 
  
  Using a computer algebra system, we even have $\sigma_\delta\approx 5.98$.
\end{theorem}
\begin{proof}
  Let $\sigma_\delta := \sigma_0 + \delta$ for some $\sigma_0>1$ and some $\delta>0$.
  Suppose there exists a MPF $w$ for any normal velocity $F=H^{\sigma_\delta}$. 
  
  Set $(a,b)=(\rho,1)$, where $0<\rho<1$.
  We recall the gradient terms $G^\alpha_\beta(\rho)$ from Lemma \ref{lem:necessary}.
  For $F=F^{\sigma_\delta}_1=\sgn(\sigma_\delta)\cdot H^{\sigma_\delta}$, we have $\beta=1$, and $\beta_a=0$.
  By Lemma \ref{lem:necessary}, $G^\alpha_1(\rho)$ is non-positive for all $0<\rho<1$.
  \begin{align*}  
    G^\alpha_1(\rho) =&\,-\alpha_a(\alpha\,\rho-1)(\beta\,\rho+1)(1-\rho) \\
                          &\,\quad+\beta_a(\alpha\,\rho-1)^2(1-\rho) \\
                          &\,\quad-(\alpha+\beta)(\alpha(1-\rho)(1-\sigma_0-\delta)
                            +\beta(1+\rho-(1-\rho)(\sigma_0+\delta))+2) \umbruch \\
                         =&\,-\alpha_a(\alpha\,\rho-1)(\rho+1)(1-\rho) \\
                          &\,\quad-(\alpha+\rho)(\alpha(1-\rho)(1-\sigma_0-\delta)
                            +(1+\rho-(1-\rho)(\sigma_0+\delta))+2)
  \end{align*}
  Since $w$ is MPF, using Theorem \ref{thm:necessary}, we get for $\rho\rightarrow 0$
  \begin{align*}
    \alpha_{w,1,\sigma_\delta}(\rho) =&\, \frac{c}{\rho} + o\left(\rho^{-1}\right) \\
             =&\, \frac{1}{\sigma_\delta\,\rho} + o\left(\rho^{-1}\right).
  \end{align*}
  Now, let $v$ be a vanishing function for $F=H^{\sigma_0}$.
  By Lemma \ref{lem:alpha vanishing V}, we get for $\rho\rightarrow 0$
  \begin{align*}
    \alpha_{v,1,\sigma_0}(\rho) = \frac{1}{\sigma_0\,\rho} + o\left(\rho^{-1}\right).
  \end{align*}
  By Lemma \ref{lem:alpha vanishing VII}, $\alpha_{v,1,\sigma_0}(\rho)$ has one root $\rho_0\in(0,-1+\sqrt{2}]$,
  if $\sigma_0\geq 3+2\sqrt{2}$. 
  
  Combining the results on $\alpha_{w,1,\sigma_\delta}(\rho)$ and $\alpha_{v,1,\sigma_0}(\rho)$
  for $\rho\rightarrow 0$ yields the existence of some zero neighborhood of $\rho$,
  where
  \begin{align}\label{id:alpha v w}
    \alpha_{v,1,\sigma_0}(\rho) \geq \alpha_{w,1,\sigma_\delta}(\rho)>0.
  \end{align}
  We get the last inequality by Lemma \ref{lem:alpha}. 
  
  Thus, there exists a first boundary/intersection point $\rho_1\in(0,\rho_0)$
  of $\alpha_{w,1,\sigma_\delta}(\rho)$ and $\alpha_{v,1,\sigma_0}(\rho)$, 
  such that 
  \begin{align}
    \begin{split}\label{id:replace alpha}
      \alpha_{w,1,\sigma_\delta}(\rho_1) =&\, \alpha_{v,1,\sigma_0}(\rho_1), \\
      \alpha_{a\,(w,1,\sigma_\delta)}(\rho_1) =&\, \alpha_{a\,(v,1,\sigma_0)}(\rho_1) + \epsilon, 
    \end{split}
  \end{align}
  for some $\epsilon\geq 0$. 
  Otherwise, we have $\alpha_{w,1,\sigma_\delta}(\rho_2)=0$ for some $\rho_2\in(0,\rho_0)$,
  which results in a contradiction to inequality \eqref{id:alpha v w}. 
  
  In the following, we calculate the gradient terms $G^\alpha_1(\rho)$ 
  at a first boundary/ intersection point $\rho_1\in(0,\rho_0)$.
  As a preliminary step, we replace $\alpha_a$ by $\alpha_a+\epsilon$.
  This yields
  \begin{align*}  
    G^\alpha_1(\rho_1) =&\,-(\alpha_a+\epsilon)(\alpha\,\rho_1-1)(\rho_1+1)(1-\rho_1) \\
                          &\,\quad-(\alpha+\rho_1)(\alpha(1-\rho_1)(1-\sigma_0-\delta)
                            +(1+\rho_1-(1-\rho_1)(\sigma_0+\delta))+2).
  \end{align*}
  By Lemma \ref{lem:alpha vanishing IV}, we have
  \begin{align*}
    \alpha_{v,1,\sigma_0}(\rho) \leq \frac{1}{\rho}
  \end{align*}
  for all $0<\rho<1$. Applying this to $G^\alpha_1(\rho_1)\leq 0$ yields
  \begin{align*}
    0\leq&\, \epsilon \\
     \leq&\,\frac{-\alpha_a(1-\alpha\,\rho_1)\left(1-\rho_1^2\right)-(\alpha+1)( \alpha(1-\rho_1)(\sigma_\delta-1)+(1-\rho_1)(\sigma_\delta+1)-4 )}{(1-\alpha\,\rho_1)\left(1-\rho_1^2\right)} \\
     =:&\, \Phi(\rho_1,\sigma_0,\delta).
  \end{align*}
  Therefore, the term $\Phi(\rho_1,\sigma_0,\delta)$ is non-negative at a first boundary/intersection point $\rho_1\in (0,\rho_0)$.
  So if $\Phi(\rho_1,\sigma_0,\delta)$ is strictly negative for all $\rho_1\in(0,\rho_0)$,
  no boundary/ intersection point $\rho_1$ can occur. 
  
  By Lemma \ref{lem:alpha vanishing VII}, we have
  \begin{align*}
    \alpha_{v,1,\sigma_0}(\rho)=\frac{1+\rho-(1-\rho)\rho\,\sigma_0}{\rho\left(\rho+\rho^2+(1-\rho)\sigma_0\right)}.
  \end{align*}
  Differentiating $\alpha_{v,1,\sigma_0}(\rho)$ with respect to $\rho$ yields
  \begin{align*}
    \alpha_{a\,(v,1,\sigma_0)}(\rho) = \frac{-\sigma_0+2\rho(\sigma_0-1)+2\rho^2(\sigma_0-2)+2\rho^3(\sigma_0-1)-\rho^4\sigma_0}{\rho^2\left(\rho+\rho^2+(1-\rho)\sigma_0\right)^2}.
  \end{align*}
  Here, we replace $\alpha$ by $\alpha_{v,1,\sigma_0}(\rho_1)$ and $\alpha_a$ by $\alpha_{a\,(v,1,\sigma_0)}(\rho_1)$ in $\Phi(\rho_1,\sigma_0,\delta)$.
  We can do so, since we assume the existence of a first boundary/intersection point $\rho_1\in(0,\rho_0)$, see \eqref{id:replace alpha}.
  We get
  \begin{align*}
     \Phi(\rho_1,\sigma_0,\delta) = \frac{\delta\cdot\Phi_1(\rho_1,\sigma_0) + \Phi_2(\rho_1,\sigma_0)}{\Phi_3(\rho_1,\sigma_0)},
  \end{align*}
  where 
  \begin{align*}
    \Phi_1(\rho_1,\sigma_0) =&\, -(1-\rho_1)\left(1+\rho_1^2\right)^2\left(\rho_1+\rho_1^2+(1-\rho_1)\sigma_0\right), \\
    \Phi_2(\rho_1,\sigma_0) =&\, \rho_1(3\sigma_0-1)
                                -4\rho_1^2\sigma_0(\sigma_0-1)
                                +\rho_1^3(\sigma_0(8\sigma_0-9)+7) \umbruch \\
                             &\quad -\rho_1^4(\sigma_0^2-1) 
                                +\rho_1^5(5\sigma_0+1)
                                -4\rho_1^6(\sigma_0-1)
                                +\rho_1^7(\sigma_0+1) \\
    \Phi_3(\rho_1,\sigma_0) =&\, \rho_1^2(1-\rho_1)^2(\sigma_0-1)\left(\rho_1+\rho_1^2+(1-\rho_1)\sigma_0\right)^2,
  \end{align*}
  Clearly, the term $\Phi_1(\rho_1,\sigma_0)$ is strictly negative and the term $\Phi_3(\rho_1,\sigma_0)$ is strictly positive for all $0<\rho_1<1$,
  and for all $\sigma_0>1$. 
  
  Recall that by Lemma \ref{lem:alpha vanishing VII}, $\alpha_{v,1,\sigma_0}(\rho)$ has one root $\rho_0\in(0,-1+\sqrt{2}]$,
  if $\sigma_0\geq 3+2\sqrt{2}$. 
  
  As defined in the beginning of this proof we have, $\sigma_\delta := \sigma_0 + \delta$.
  In the last two paragraphs of this proof, we show that 
  $\delta\cdot\Phi_1(\rho_1,\sigma_0) + \Phi_2(\rho_1,\sigma_0)$ is strictly negative for some fixed $\sigma_0>1$,
  and some fixed $\delta>0$ for all $\rho_1\in(0,\rho_0)$.
  As noted before, $\Phi_1(\rho_1,\sigma_0)$ is strictly negative for all $0<\rho_1<1$. Therefore, we have 
  $\delta\cdot\Phi_1(\rho_1,\sigma_0) + \Phi_2(\rho_1,\sigma_0)\geq \delta_1\cdot\Phi_1(\rho_1,\sigma_0) + \Phi_2(\rho_1,\sigma_0)$,
  if $\delta\leq\delta_1$, for all $\rho_1\in(0,\rho_0)$.
  So no boundary/intersection point can occur for any $\sigma_0+\delta_1\geq\sigma_0+\delta$.
  Hence, no MPF $w$ exists for any $\sigma\geq\sigma_\delta$. 
  
  Using a computer algebra system, we compute $\delta\cdot\Phi_1(\rho_1,\sigma_0) + \Phi_2(\rho_1,\sigma_0)$
  for $\sigma_0 = 3+2\sqrt{2}$, and $\delta=0.151$, \ie $\sigma_\delta=\sigma_0+\delta\approx 5.98$.
  We get that $\delta\cdot\Phi_1(\rho_1,\sigma_0) + \Phi_2(\rho_1,\sigma_0)$ is strictly negative for all
  $0<\rho_1\leq-1+\sqrt{2}$.
  By Lemma \ref{lem:alpha vanishing VII}, $\alpha_{v,1,\sigma_0}(\rho)$ has one root $\rho_0\in(0,-1+\sqrt{2}]$,
  if $\sigma_0\geq 3+2\sqrt{2}$.
  Hence, there are no MPF $w$ for any $F=H^\sigma$ with $\sigma\geq 5.98$. 
  
  By Lemma \ref{lem:tedious H}, we calculate $\delta\cdot\Phi_1(\rho_1,\sigma_0) + \Phi_2(\rho_1,\sigma_0)$
  for $\sigma_0=6$, and $\delta=1$.
  We get that $\delta\cdot\Phi_1(\rho_1,\sigma_0) + \Phi_2(\rho_1,\sigma_0)$ is strictly negative for all
  $0<\rho_1\leq 1/3$.
  By, Lemma \ref{lem:alpha vanishing VII}, $\alpha_{v,1,\sigma}(\rho)$ has one root $\rho_0\in(0,1/3]$ for all $\sigma\geq 6$.
  Hence, there are no MPF for any $F=H^\sigma$ with $\sigma\geq 7$. 
  
  This concludes the proof.
\end{proof}

\begin{lemma}[Calculations for Theorem \ref{thm:mean curvature III}]\label{lem:tedious H}
  Define
  \begin{align*}
    \tilde{\Phi}(\rho,\sigma,\delta) := \delta\cdot\Phi_1(\rho,\sigma) + \Phi_2(\rho,\sigma),
  \end{align*}
  where
  \begin{align*}
    \Phi_1(\rho,\sigma) =&\, -(1-\rho)\left(1+\rho^2\right)^2\left(\rho+\rho^2+(1-\rho)\sigma\right), \\
    \Phi_2(\rho,\sigma) =&\, \rho(3\sigma-1)
                                -4\rho^2\sigma(\sigma-1)
                                +\rho^3(\sigma(8\sigma-9)+7) \\
                             &\quad -\rho^4(\sigma^2-1) 
                                +\rho^5(5\sigma+1)
                                -4\rho^6(\sigma-1)
                                +\rho^7(\sigma+1),
  \end{align*}
  as in Theorem \ref{thm:mean curvature III}. 
  Let $\sigma=6$, and $\delta=1$. 
  Then we have that $\tilde{\Phi}(\rho,6,1)$ is strictly negative for all $0<\rho\leq 1/3$.
\end{lemma}
\begin{proof}
  First, we calculate $\tilde{\Phi}(\rho,6,1)$. We get
  \begin{align*}
    \tilde{\Phi}(\rho,6,1) =&\, -6+28\rho-138\rho^2+264\rho^3-158\rho^4+44\rho^5-26\rho^6+8\rho^7.
  \end{align*}
  Next, we rewrite $\tilde{\Phi}(\rho,6,1)$ as a sum of four polynomials
  \begin{align*}
    \tilde{\Phi}(\rho,6,1) = \tilde{\Phi}_1(\rho) + \tilde{\Phi}_2(\rho) + \tilde{\Phi}_3(\rho) + \tilde{\Phi}_4(\rho),
  \end{align*}
  where
  \begin{align*}
    \tilde{\Phi}_1(\rho) =&\, -6+28\rho-38\rho^2,\\
    \tilde{\Phi}_2(\rho) =&\, -100\rho^2+264\rho^3,\\
    \tilde{\Phi}_3(\rho) =&\, -158\rho^4+44\rho^5,\\
    \tilde{\Phi}_4(\rho) =&\, -26\rho^6+8\rho^7.
  \end{align*}
  Now, we show that $\tilde{\Phi}_1(\rho)$ is strictly negative, and 
  $\tilde{\Phi}_i(\rho),\,i=2,3,4$, is non-positive for all $0<\rho\leq 1/3$. 
  \begin{itemize}
    \item $\tilde{\Phi}_1(\rho)$ has no roots, and $\tilde{\Phi}_1(0)=-6$,
    \item $\tilde{\Phi}_2(\rho)/\rho^2$ has one root at $\rho=22/66\approx 0.379$, and $\tilde{\Phi}_2(1/6)=-14/9$,
    \item $\tilde{\Phi}_3(\rho)/\rho^4$ has one root at $\rho=79/22\approx 3.591$, and $\tilde{\Phi}_3(1/6)=-113/972$,
    \item $\tilde{\Phi}_4(\rho)/\rho^6$ has one root at $\rho=13/4=3.25$, and $\tilde{\Phi}_4(1/6)=-37/69984$.
  \end{itemize}  
  This concludes the proof.
\end{proof}

\begin{lemma}[Mean curvature]\label{lem:H expanding}
  Let $F=F^\sigma_1=\sgn(\sigma)\cdot H^\sigma$. Then 
  \begin{align*}
    w=\frac{(a-b)^2(a^2+b^2)(a\,b)^{\sigma-2}}{(a+b)}
  \end{align*}
  is a maximum-principle function, if $\sigma\in(-1,0)$.
\end{lemma}
\begin{proof}
  It is rather easy to see that $w$ is a MPF except for MPF condition \eqref{IV}. 
  Here, we need a computer algebra program.
  We compute $C(a,b)$, $E(a,b)$, $G(a,b)$ for the given $w$ and $F$. 
  Then we use a Monte-Carlo method to check for non-positivity of these terms.
  Since $C(a,b)$, $E(a,b)$, $G(a,b)$ are homogeneous, $C(a,b)=C(b,a)$, and $E(a,b)=G(b,a)$,
  it suffices to check for non-positivity testing only on $(a,b)=(\rho,1)$, where $0<\rho\leq 1$.
  
  In a separate paper, using this MPF we will show convergence to a sphere at $\infty$.
\end{proof}

\newpage

\begin{corollary}[Mean curvature]\label{cor:mean curvature}
  Let $F=F^\sigma_1=\sgn(\sigma)\cdot H^\sigma$. 
  Then we have
  \begin{center}
    \begin{tabular}{| l | c || c | c |}
      \hline 
      & MPF & & \\
      \hline \hline           
      $\Bigg.\Bigg. \sigma\geq 5.
      \begin{small}
        \text{98}
      \end{small}$ & non-existent & & Theorem \ref{thm:mean curvature III} \\
      \hline
      $\Bigg.\Bigg. \sigma\in(5.
      \begin{small}
        \text{17}
      \end{small}
      ,5.
      \begin{small}
        \text{98}
      \end{small}
      )$ & open problem &  & \\
      \hline
      $\Bigg.\Bigg. \sigma\in(1,5.
      \begin{small}
        \text{17}
      \end{small}
      ]$ & $\frac{\left(a-b\right)^2\left(a+b\right)^{2\sigma}}{\left(a\,b\right)^2}$ & F. Schulze, O. Schn\"urer & \cite{fs:convexity} \\
      \hline
      $\Bigg.\Bigg. \sigma\in(0,1]$ & non-existent & & Theorem \ref{thm:mean curvature II} \\
      \hline \hline
      $\Bigg.\Bigg. \sigma\in(-1,0)$ & $\frac{\left(a-b\right)^2\left(a^2+b^2\right)\left(a\,b\right)^\sigma}{\left(a+b\right)\left(a\,b\right)^2}$ & & Lemma \ref{lem:H expanding} \\
      \hline
      $\Bigg.\Bigg. \sigma=-1$ & $\frac{\left(a-b\right)^2}{\left(a+b\right)\left(a\,b\right)}$ & O. Schn\"urer & \cite{os:surfaces} \\
      \hline
      $\Bigg.\Bigg. \sigma<-1$ & non-existent & & Theorem \ref{thm:mean curvature I} \\
      \hline  
    \end{tabular}
  \end{center}
\end{corollary}

\section{Norm of the second fundamental form, $F^\sigma_2$}

In this chapter, we discuss $F^\sigma_2=\sgn(\sigma)\cdot |A|^\sigma$
for all $\sigma\in\R\setminus\{0\}$. 

For $\sigma\geq 9.89$, we show the non-existence of MPF. 
This is Theorem \ref{thm:|A| III}.
We show this by comparing an assumed MPF to a vanishing function.
Since vanishing functions develop roots for any
$\sigma\geq 9.444$,
we obtain a contradiction.
Note that the quantity $\alpha$ has to be strictly positive
for a MPF, see Lemma \ref{lem:alpha}. 

Our proof technique does not work for $8.15<\sigma<9.444$
since the quantity $\alpha$ of vanishing functions
does not develop any roots there. 

For $9.444\leq\sigma<9.89$, $\alpha$ has roots but
our technique still does not work, since it involves 
some estimates, which apparently are not sharp. 

For any $1<\sigma\leq 8.15$, we have a MPF by 
B. Andrews and X. Chen \cite{ba:contraction}, 
which is also a vanishing function. 

By Theorem \ref{thm:|A| II}, we get
the non-existence of MPF for any $0<\sigma\leq 1$. 

In Lemma \ref{lem:|A| expanding}, we give a MPF, 
which we present fully in a separate paper. 
This is $-1\leq\sigma<0$. 

And for any $\sigma<-1$, the non-existence of MPF is shown 
in Theorem \ref{thm:|A| I}.
We prove this using the necessary conditions for $F^\sigma_\xi$,
which we obtained in Theorem \ref{thm:necessary}. 

We summarize our results for $F^\sigma_2=\sgn(\sigma)\cdot |A|^\sigma$
in the table of Corollary \ref{cor:|A|}.

\begin{theorem}[Norm of the second fundamental form I]\label{thm:|A| I} 
  Let $F=F^\sigma_2=\sgn(\sigma)\cdot |A|^\sigma$. 
  Then there are no $MPF$, if $\sigma<-1$.
\end{theorem}
\begin{proof}
  Let $F^\sigma_\xi$ be defined as in Remark \ref{rem:normal velocity xi}. 
  For the norm of the second fundamental form, we get
  $F^\sigma_2=\sgn(\sigma)\cdot |A|^\sigma$. 
  By Theorem \ref{thm:necessary}, we have the necessary condition $\sigma\in[-1,0)$
  for the existence of MPF in case of expanding normal velocities, \ie $\sigma<0$.
  Hence, the claim follows.
\end{proof}

\begin{theorem}[Norm of the second fundamental form II]\label{thm:|A| II} 
  Let $F=F^\sigma_2=\sgn(\sigma)\cdot |A|^\sigma$. 
  Then there are no $MPF$, if $0<\sigma\leq 1$.
\end{theorem}
\begin{proof}
  This is a direct consequence of Lemma \ref{lem:one} and Lemma \ref{lem:zeroone}.
\end{proof}

\begin{theorem}[Norm of the second fundamental form III]\label{thm:|A| III} 
  Let $F=F^\sigma_2=\sgn(\sigma)\cdot |A|^\sigma$. 
  Then there are no $MPF$, if $\sigma\geq \sigma_\delta$ with $\sigma_\delta=11$. 
  
  Using a computer algebra system, we even have $\sigma_\delta\approx 9.899$.
\end{theorem}
\begin{proof}
  Let $\sigma_\delta := \sigma_0 + \delta$ for some $\sigma_0>1$ and some $\delta>0$.
  Suppose there exists a MPF $w$ for any normal velocity $F=|A|^{\sigma_\delta}$. 
  
  Set $(a,b)=(\rho,1)$, where $0<\rho<1$.
  We recall the gradient terms $G^\alpha_\beta(\rho)$ from Lemma \ref{lem:necessary}.
  For $F=F^{\sigma_\delta}_2=\sgn(\sigma_\delta)\cdot |A|^{\sigma_\delta}$, we have $\beta=\rho$, and $\beta_a=1$.
  By Lemma \ref{lem:necessary}, $G^\alpha_\rho(\rho)$ is non-positive for all $0<\rho<1$.
  \begin{align*}  
    G^\alpha_\rho(\rho) =&\,-\alpha_a(\alpha\,\rho-1)(\beta\,\rho+1)(1-\rho) \\
                          &\,\quad+\beta_a(\alpha\,\rho-1)^2(1-\rho) \\
                          &\,\quad-(\alpha+\beta)(\alpha(1-\rho)(1-\sigma_0-\delta)
                            +\beta(1+\rho-(1-\rho)(\sigma_0+\delta))+2) \umbruch \\
                        =&\,-\alpha_a(\alpha\,\rho-1)\left(\rho^2+1\right)(1-\rho) \\
                         &\,\quad+1\cdot(\alpha\,\rho-1)^2(1-\rho) \\
                         &\,\quad-(\alpha+\rho)(\alpha(1-\rho)(1-\sigma_0-\delta)
                            +\rho(1+\rho-(1-\rho)(\sigma_0+\delta))+2).
  \end{align*}
  Since $w$ is MPF, using Theorem \ref{thm:necessary}, we get for $\rho\rightarrow 0$
  \begin{align*}
    \alpha_{w,\rho,\sigma_\delta}(\rho) =&\, \frac{c}{\rho} + o\left(\rho^{-1}\right) \\
             =&\, \frac{1}{\sigma_\delta\,\rho} + o\left(\rho^{-1}\right).
  \end{align*}
  Now, let $v$ be a vanishing function for $F=|A|^{\sigma_0}$.
  By Lemma \ref{lem:alpha vanishing V}, we get for $\rho\rightarrow 0$
  \begin{align*}
    \alpha_{v,\rho,\sigma_0}(\rho) = \frac{1}{\sigma_0\,\rho} + o\left(\rho^{-1}\right).
  \end{align*}
  By Lemma \ref{lem:alpha vanishing VIII}, $\alpha_{v,\rho,\sigma_0}(\rho)$ has one root $\rho_0\in(0,\rho_\star]$,
  if $\sigma_0\geq \sigma_\star$, where $\rho_\star \approx 0.596$, $\sigma_\star\approx 9.444$. 
  
  Combining the results on $\alpha_{w,\rho,\sigma_\delta}(\rho)$ and $\alpha_{v,\rho,\sigma_0}(\rho)$
  for $\rho\rightarrow 0$ yields the existence of some zero neighborhood of $\rho$,
  where
  \begin{align}\label{id:alpha v w A}
    \alpha_{v,\rho,\sigma_0}(\rho) \geq \alpha_{w,\rho,\sigma_\delta}(\rho)>0.
  \end{align}
  We get the last inequality by Lemma \ref{lem:alpha}. 
  
  Thus, there exists a first boundary/intersection point $\rho_1\in(0,\rho_0)$
  of $\alpha_{w,\rho,\sigma_\delta}(\rho)$ and $\alpha_{v,\rho,\sigma_0}(\rho)$, 
  such that 
  \begin{align}
    \begin{split}\label{id:replace alpha A}
      \alpha_{w,\rho,\sigma_\delta}(\rho_1) =&\, \alpha_{v,\rho,\sigma_0}(\rho_1), \\
      \alpha_{a\,(w,\rho,\sigma_\delta)}(\rho_1) =&\, \alpha_{a\,(v,\rho,\sigma_0)}(\rho_1) + \epsilon, 
    \end{split}
  \end{align}
  for some $\epsilon\geq 0$. 
  Otherwise, we have $\alpha_{w,\rho,\sigma_\delta}(\rho_2)=0$ for some $\rho_2\in(0,\rho_0)$,
  which results in a contradiction to inequality \eqref{id:alpha v w A}. 
  
  In the following, we calculate the gradient terms $G^\alpha_1(\rho)$ 
  at a first boundary/ intersection point $\rho_1\in(0,\rho_0)$.
  As a preliminary step, we replace $\alpha_a$ by $\alpha_a+\epsilon$.
  This yields
  \begin{align*}  
    G^\alpha_\rho(\rho_1) =&\,-(\alpha_a+\epsilon)(\alpha\,\rho-1)\left(\rho^2+1\right)(1-\rho) \\
                             &\,\quad+(\alpha\,\rho-1)^2(1-\rho) \\
                             &\,\quad-(\alpha+\rho)(\alpha(1-\rho)(1-\sigma_0-\delta)
                                +\rho(1+\rho-(1-\rho)(\sigma_0+\delta))+2).
  \end{align*}
  By Lemma \ref{lem:alpha vanishing IV}, we have
  \begin{align*}
    \alpha_{v,\rho,\sigma_0}(\rho) \leq \frac{1}{\rho}
  \end{align*}
  for all $0<\rho<1$. Applying this to $G^\alpha_\rho(\rho_1)\leq 0$ yields
  \begin{align*}
    0\leq&\, \epsilon \\
     \leq&\,\frac{-\alpha_a(1-\alpha\,\rho_1)(1-\rho_1)\left(1+\rho_1^2\right)}{(1-\alpha\,\rho_1)(1-\rho_1)\left(1+\rho_1^2\right)} \\
         &\quad+ \frac{-(\alpha+\rho_1)(\alpha(1-\rho_1)(\sigma_\delta-1)+\rho_1(1-\rho_1)(\sigma_\delta+1)-2-2\rho_1}{(1-\alpha\,\rho_1)(1-\rho_1)\left(1+\rho_1^2\right)}\\
         &\quad+ \frac{-(1-\rho_1)\left(\alpha\,\rho_1-1\right)^2}{(1-\alpha\,\rho_1)(1-\rho_1)\left(1+\rho_1^2\right)}\\
     =:&\, \Phi(\rho_1,\sigma_0,\delta).
  \end{align*}
  Therefore, the term $\Phi(\rho_1,\sigma_0,\delta)$ is non-negative at a first boundary/intersection point $\rho_1\in (0,\rho_0)$.
  So if $\Phi(\rho_1,\sigma_0,\delta)$ is strictly negative for all $\rho_1\in(0,\rho_0)$,
  no boundary/ intersection point $\rho_1$ can occur. 
  
  By Lemma \ref{lem:alpha vanishing VIII}, we have
  \begin{align*}
    \alpha_{v,\rho,\sigma_0}(\rho)=\frac{1+\rho^2-(1-\rho)\rho^2\,\sigma_0}{\rho\left(\rho+\rho^3+(1-\rho)\sigma_0\right)}.
  \end{align*}
  Differentiating $\alpha_{v,\rho,\sigma_0}(\rho)$ with respect to $\rho$ yields
  \begin{align*}
    \alpha_{a\,(v,\rho,\sigma_0)}(\rho) =&\,
    \frac{-\sigma_0+2\rho(\sigma_0-1)-\rho^2\sigma_0(\sigma_0-1)+2\rho^3\left(\sigma_0^2-2\right)}
    {\rho^2\left(\rho+\rho^3+(1-\rho)\sigma_0\right)^2} \\
    &\quad +\frac{-\rho^4\sigma_0(\sigma_0-1)+2\rho^5(\sigma_0-1)-\rho^6\sigma_0}
        {\rho^2\left(\rho+\rho^3+(1-\rho)\sigma_0\right)^2}.
  \end{align*}
  Here, we replace $\alpha$ by $\alpha_{v,\rho,\sigma_0}(\rho_1)$ and $\alpha_a$ by $\alpha_{a\,(v,\rho,\sigma_0)}(\rho_1)$ in $\Phi(\rho_1,\sigma_0,\delta)$.
  We can do so, since we assume the existence of a first boundary/intersection point $\rho_1\in(0,\rho_0)$, see \eqref{id:replace alpha A}.
  We get
  \begin{align*}
     \Phi(\rho_1,\sigma_0,\delta) = \frac{\delta\cdot\Phi_1(\rho_1,\sigma_0) + \Phi_2(\rho_1,\sigma_0)}{\Phi_3(\rho_1,\sigma_0)},
  \end{align*}
  where 
  \begin{align*}
    \Phi_1(\rho_1,\sigma_0) =&\, -(1-\rho_1)\left(1+\rho_1^3\right)^2\left(\rho_1+\rho_1^3+(1-\rho_1)\sigma_0\right), \\
    \Phi_2(\rho_1,\sigma_0) =&\, \rho_1(3\sigma_0-1)
                                -3\rho_1^2(\sigma_0-1)
                                -\rho_1^3(\sigma_0(5\sigma_0-9)+4) \\
                             &\quad +\rho_1^4(\sigma_0(13\sigma_0-19)\sigma_0+12) 
                                -2\rho_1^5(\sigma_0-1)(7\sigma_0-4) \\
                             &\quad +2\rho_1^6(\sigma_0-1)(5\sigma_0-6) 
                                -\rho_1^7(5\sigma_0(\sigma_0-3)-4) 
                                +\rho_1^8(\sigma_0-4)(\sigma_0-1) \\
                             &\quad -\rho_1^9(\sigma_0-1) 
                                +\rho_1^{10}(\sigma_0+1), \\
    \Phi_3(\rho_1,\sigma_0) =&\, \rho_1^2(1-\rho_1)^2(\sigma_0-1)\left(\rho_1+\rho_1^3+(1-\rho_1)\sigma_0\right)^2.
  \end{align*}
  Clearly, the term $\Phi_1(\rho_1,\sigma_0)$ is strictly negative and the term $\Phi_3(\rho_1,\sigma_0)$ is strictly positive for all $0<\rho_1<1$,
  and for all $\sigma_0>1$. 
  
  Recall that by Lemma \ref{lem:alpha vanishing VIII}, $\alpha_{v,\rho,\sigma_0}(\rho)$ has one root $\rho_0\in(0,\rho_\star]$,
  if $\sigma_0\geq \sigma_\star$, where $\rho_\star \approx 0.596$, $\sigma_\star\approx 9.444$. 
  
  As defined in the beginning of this proof we have, $\sigma_\delta := \sigma_0 + \delta$.
  In the last two paragraphs of this proof, we show that 
  $\delta\cdot\Phi_1(\rho_1,\sigma_0) + \Phi_2(\rho_1,\sigma_0)$ is strictly negative for some fixed $\sigma_0>1$,
  and some fixed $\delta>0$ for all $\rho_1\in(0,\rho_0)$.
  As noted before, $\Phi_1(\rho_1,\sigma_0)$ is strictly negative for all $0<\rho_1<1$. Therefore, we have 
  $\delta\cdot\Phi_1(\rho_1,\sigma_0) + \Phi_2(\rho_1,\sigma_0)\geq \delta_1\cdot\Phi_1(\rho_1,\sigma_0) + \Phi_2(\rho_1,\sigma_0)$,
  if $\delta\leq\delta_1$, for all $\rho_1\in(0,\rho_0)$.
  So no boundary/intersection point can occur for any $\sigma_0 + \delta_1 \geq \sigma_0 + \delta$.
  Hence, no MPF $w$ exists for any $\sigma\geq\sigma_\delta$. 
  
  Using a computer algebra system, we can compute $\delta\cdot\Phi_1(\rho_1,\sigma_0) + \Phi_2(\rho_1,\sigma_0)$
  for $\sigma_0 = \sigma_\star$, and $\delta=0.4558$, \ie $\sigma_\delta=\sigma_0+\delta\approx 9.899$.
  We get that $\delta\cdot\Phi_1(\rho_1,\sigma_0) + \Phi_2(\rho_1,\sigma_0)$ is strictly negative for all
  $0<\rho_1\leq \rho_\star$.
  By Lemma \ref{lem:alpha vanishing VIII}, $\alpha_{v,\rho,\sigma_0}(\rho)$ has one root $\rho_0\in(0,\rho_\star]$,
  if $\sigma_0\geq \sigma_\star$, where $\rho_\star \approx 0.596$, $\sigma_\star\approx 9.444$. 
  
  Hence, there are no MPF $w$ for any $F=|A|^\sigma$ with $\sigma\geq 9.899$. 
  
  By Lemma \ref{lem:tedious A}, we calculate $\delta\cdot\Phi_1(\rho_1,\sigma_0) + \Phi_2(\rho_1,\sigma_0)$
  for $\sigma_0=10$, and $\delta=1$.
  We get that $\delta\cdot\Phi_1(\rho_1,\sigma_0) + \Phi_2(\rho_1,\sigma_0)$ is strictly negative for all
  $0<\rho_1\leq 1/2$.
  By Lemma \ref{lem:alpha vanishing VIII}, $\alpha_{v,\rho,\sigma}(\rho)$ has one root $\rho_0\in(0,1/2]$ for all $\sigma\geq 10$.
  Hence, there are no MPF for any $F=|A|^\sigma$ with $\sigma\geq 11$. 
  
  This concludes the proof.
\end{proof}

\begin{lemma}[Calculations for Theorem \ref{thm:|A| III}]\label{lem:tedious A}
  Define
  \begin{align*}
    \tilde{\Phi}(\rho,\sigma,\delta) := \delta\cdot\Phi_1(\rho,\sigma) + \Phi_2(\rho,\sigma),
  \end{align*}
  where
  \begin{align*}
    \Phi_1(\rho,\sigma) =&\, -(1-\rho)\left(1+\rho^3\right)^2\left(\rho+\rho^3+(1-\rho)\sigma\right), \\
    \Phi_2(\rho,\sigma) =&\, \rho(3\sigma-1)
                                -3\rho^2(\sigma-1)
                                -\rho^3(\sigma(5\sigma-9)+4) \\
                             &\quad +\rho^4(\sigma(13\sigma-19)\sigma+12) 
                                -2\rho^5(\sigma-1)(7\sigma-4) \\
                             &\quad +2\rho^6(\sigma-1)(5\sigma-6) 
                                -\rho^7(5\sigma(\sigma-3)-4) 
                                +\rho^8(\sigma-4)(\sigma-1) \\
                             &\quad -\rho^9(\sigma-1) 
                                +\rho^{10}(\sigma+1),
  \end{align*}
  as in Theorem \ref{thm:|A| III}.
  Let $\sigma=10$, and $\delta=1$. 
  Then $\tilde{\Phi}(\rho,10,1)$ is strictly negative for all $0<\rho\leq 1/2$. 
  
  Furthermore, $\alpha_{v,\rho,\sigma}(\rho)$ from Lemma \ref{lem:alpha vanishing VIII} has
  one root $\rho_0\in(0,1/2]$ for all $\sigma\geq 10$.
\end{lemma}
\begin{proof}
  We begin this proof by calculating $\tilde{\Phi}(\rho,10,1)$. We obtain
  \begin{align*}
    \tilde{\Phi}(\rho,10,1) =&\, -10+48\rho-36\rho^2-435\rho^3+1161\rho^4-1206\rho^5 \\
                             &\quad +780\rho^6-333\rho^7+45\rho^8-10\rho^9+12\rho^{10}.
  \end{align*}
  Firstly, let $0<\rho\leq 0.3$. We rewrite $\tilde{\Phi}(\rho,10,1)$ as a sum of five polynomials
  \begin{align*}
    \tilde{\Phi}(\rho,10,1) = \tilde{\Phi}_1(\rho) + \tilde{\Phi}_2(\rho) + \tilde{\Phi}_3(\rho) + \tilde{\Phi}_4(\rho) + \tilde{\Phi}_5(\rho),
  \end{align*}
  where
  \begin{align*}
    \tilde{\Phi}_1(\rho) =&\, -10+48\rho-36\rho^2-85\rho^3,\\
    \tilde{\Phi}_2(\rho) =&\, -350\rho^3+1161\rho^4,\\
    \tilde{\Phi}_3(\rho) =&\, -1206\rho^5+780\rho^6,\\
    \tilde{\Phi}_4(\rho) =&\, -333\rho^7+45\rho^8,\\
    \tilde{\Phi}_5(\rho) =&\, -10\rho^9+12\rho{10}.
  \end{align*}
  Now, we show that $\tilde{\Phi}_1(\rho)$ is strictly negative, and 
  $\tilde{\Phi}_i(\rho),\,i=2,3,4,5$, is non-positive for all $0<\rho\leq 0.3$. 
  \begin{itemize}
    \item $\tilde{\Phi}_1(\rho)$ has no critical points, and $\tilde{\Phi}_1(0)=-10$, and $\tilde{\Phi}_1(0.3)=-1.135$,
    \item $\tilde{\Phi}_2(\rho)/\rho^3$ has no roots, and $\tilde{\Phi}_2(\rho)/\rho^3_{|\rho=1/4}=-239/4$,
    \item $\tilde{\Phi}_3(\rho)/\rho^5$ has no roots, and $\tilde{\Phi}_3(\rho)/\rho^5_{|\rho=1/4}=-1011$,
    \item $\tilde{\Phi}_4(\rho)/\rho^7$ has no roots, and $\tilde{\Phi}_4(\rho)/\rho^7_{|\rho=1/4}=-1287/4$.
    \item $\tilde{\Phi}_5(\rho)/\rho^9$ has no roots, and $\tilde{\Phi}_5(\rho)/\rho^9_{|\rho=1/4}=-7$.
  \end{itemize}  
  
  Secondly, let $0.3\leq\rho\leq 0.5$. We rewrite $\tilde{\Phi}(\rho,10,1)$ as a sum of five different polynomials
  \begin{align*}
    \tilde{\Phi}(\rho,10,1) = \tilde{\Phi}_1(\rho) + \tilde{\Phi}_2(\rho) + \tilde{\Phi}_3(\rho) + \tilde{\Phi}_4(\rho) + \tilde{\Phi}_5(\rho),
  \end{align*}
  where
  \begin{align*}
    \tilde{\Phi}_1(\rho) =&\, -10+48\rho-36\rho^2-57\rho^3,\\
    \tilde{\Phi}_2(\rho) =&\, -378\rho^3+1161\rho^4-855\rho^5,\\
    \tilde{\Phi}_3(\rho) =&\, -351\rho^5+780\rho^6-306\rho^7,\\
    \tilde{\Phi}_4(\rho) =&\, -27\rho^7+45\rho^8,\\
    \tilde{\Phi}_5(\rho) =&\, -10\rho^9+12\rho^{10}.
  \end{align*}
  Now, we show that 
  $\tilde{\Phi}_i(\rho),\,i=1,\ldots,5$, is strictly negative for all $0.3\leq\rho\leq 0.5$. 
  \begin{itemize}
    \item $\tilde{\Phi}_1(\rho)$ has one critical point at $\rho_{CP}=4/57(-3+\sqrt{66})\approx 0.36$, 
      and $\tilde{\Phi}_1(0.3)=-0.379$, $\tilde{\Phi}_1(\rho_{CP})=2/1083(-11463+1408\sqrt{66})\approx -0.045$,
          $\tilde{\Phi}_1(0.5)=-2.125$,
    \item $\tilde{\Phi}_2(\rho)$ has one critical point at $\rho_{CP}=3/475(86-\sqrt{746})\approx 0.371$, 
      and $\tilde{\Phi}_2(\rho)/\rho^3_{|\rho=0.3}=-106.65$, $\tilde{\Phi}_2(\rho)/\rho^3_{|\rho=\rho_{CP}}=-27/2375(2206+129\sqrt{746})\approx-3.317$,
          $\tilde{\Phi}_2(\rho)/\rho^3_{|\rho=0.5}=-11.25$,
    \item $\tilde{\Phi}_3(\rho)$ has one critical point at $\rho_{CP}=1/238(260-\sqrt{21190})\approx 0.481$, 
      and $\tilde{\Phi}_3(\rho)/\rho^5_{|\rho=0.3}=-144.54$, $\tilde{\Phi}_3(\rho)/\rho^5_{|\rho=\rho_{CP}}=-78/833(-229+5\sqrt{21190})\approx-46.71$,
          $\tilde{\Phi}_3(\rho)/\rho^5_{|\rho=0.5}=-37.5$,
    \item $\tilde{\Phi}_4(\rho)/\rho^7$ has no root, and $\tilde{\Phi}_4(\rho)/\rho^7_{\rho=0.4}=-9$,
    \item $\tilde{\Phi}_5(\rho)/\rho^9$ has no root, and $\tilde{\Phi}_5(\rho)/\rho^9_{\rho=0.4}=-5.2$.
  \end{itemize}   
  This concludes the proof.
\end{proof}

\begin{lemma}[Norm of the second fundamental form]\label{lem:|A| expanding}
  $\left.\right.$ \\
  Let $F=F^\sigma_2=\sgn(\sigma)\cdot |A|^\sigma$. Then 
  \begin{align*}
    w=\frac{(a-b)^2(a^3+b^3)(a\,b)^\sigma}{(a^2+b^2)^{1/2}(a\,b)^2}
  \end{align*}
  is a maximum-principle function, if $\sigma\in[-1,0)$.
\end{lemma}
\begin{proof}
  It is rather easy to see that $w$ is a MPF except for MPF condition \eqref{IV}. 
  Here, we need a computer algebra program.
  We compute $C(a,b)$, $E(a,b)$, $G(a,b)$ for the given $w$ and $F$. 
  Then we use a Monte-Carlo method to check for non-positivity of these terms.
  Since $C(a,b)$, $E(a,b)$, $G(a,b)$ are homogeneous, $C(a,b)=C(b,a)$, and $E(a,b)=G(b,a)$,
  it suffices to check for non-positivity testing only on $(a,b)=(\rho,1)$, where $0<\rho\leq 1$.
  
  In a separate paper, using this MPF we will show convergence to a sphere at $\infty$.
\end{proof}

\begin{corollary}[Norm of the second fundamental form]\label{cor:|A|}
  $\left.\right.$ \\
  Let $F=F^\sigma_2=\sgn(\sigma)\cdot |A|^\sigma$. 
  Then we have
  \begin{center}
    \begin{tabular}{| l | c || c | c |}
      \hline 
      & MPF & & \\
      \hline \hline             
      $\Bigg.\Bigg. \sigma\geq 9.
      \begin{small}
        \text{89}
      \end{small}
      $ & non-existent & & Theorem \ref{thm:|A| III} \\
      \hline
      $\Bigg.\Bigg. \sigma\in(8.
      \begin{small}
        \text{15}
      \end{small}
      ,9.
      \begin{small}
        \text{89}
      \end{small}
      )$ & open problem & & \\ 
      \hline
      $ \Bigg.\Bigg. \sigma\in(1,8.
      \begin{small}
        \text{15}
      \end{small}
      ]$ & $\frac{\left(a-b\right)^2\left(a^2+b^2\right)^\sigma}{\left(a\,b\right)^2}$ & B. Andrews, X. Chen & \cite{ac:surfaces} \\
      \hline
      $\Bigg.\Bigg. \sigma\in(0,1]$ & non-existent & & Theorem \ref{thm:|A| II} \\
      \hline \hline
      $\Bigg.\Bigg. \sigma\in[-1,0)$ & $\frac{\left(a-b\right)^2\left(a^3+b^3\right)\left(a\,b\right)^\sigma}{\left(a^2+b^2\right)^{1/2}\left(a\,b\right)^2}$ & & Lemma \ref{lem:|A| expanding} \\
      \hline
      $\Bigg.\Bigg. \sigma<-1$ & non-existent & & Theorem \ref{thm:|A| I} \\
      \hline  
    \end{tabular}
  \end{center}
\end{corollary}

\section{Trace of the second fundamental form, $F^\sigma_\sigma$}

In this chapter, we discuss $F^\sigma_\sigma=\sgn(\sigma)\cdot \tr A^\sigma$
for all $\sigma\in\R\setminus\{0\}$. 

For $\sigma>1$, the vanishing function 
\begin{align*}
  v = \frac{(a-b)^2\,F^2}{(a\,b)^2}
\end{align*}
from example \ref{exm:vanishing} is MPF.
This is a result by B. Andrews and X. Chen \cite{ac:surfaces}.
Interestingly, to date, this is the only time 
where convergence to a round point can be shown 
for a normal velocity homogeneous of 
an arbitrarily high degree using a MPF. 

For $0<\sigma\leq 1$, we prove the non-existence of MPF by
Lemma \ref{lem:one} and Lemma \ref{lem:zeroone}.
This is Theorem \ref{thm:trA}. 

For expanding $F^\sigma_\xi$, \ie $\sigma<0$, 
the necessary conditions of Theorem \ref{thm:necessary}
do not impose any constraints on the existence of MPF. 

We only know that for $\sigma=-1$ a MPF exists.
Since
\begin{align*}
  F^{-1}_{-1} =&\, -\tr A^{-1} \\
              =&\, -\left(a^{-1}+b^{-1}\right) \\
              =&\, -\left(\frac{a+b}{a\,b}\right) \\
              =&\, -\frac{H}{K},
\end{align*}
we have a MPF by O. Schn\"urer \cite{os:surfaces}.

\begin{theorem}[Trace of the second fundamental form]\label{thm:trA}
  $\left.\right.$ \\
  Let $F=F^\sigma_\sigma=\sgn(\sigma)\cdot \tr A^\sigma$. 
  Then there are no $MPF$ if $0<\sigma\leq 1$.
\end{theorem}
\begin{proof}
  This is a direct consequence of Lemma \ref{lem:one} and Lemma \ref{lem:zeroone}.
\end{proof}

\begin{corollary}[Trace of the second fundamental form]\label{cor:trA}
  $\left.\right.$ \\
  Let $F=F^\sigma_\sigma=\sgn(\sigma)\cdot \tr A^\sigma$. 
  Then we have
  \begin{center}
    \begin{tabular}{| l | c || c | c |}
      \hline 
      & MPF & & \\
      \hline \hline                  
      $\Bigg.\Bigg. \sigma>1$ & $\frac{(a-b)^2\left(a^\sigma+b^\sigma\right)^2}{(a\,b)^2}$ & B. Andrews, X. Chen & \cite{ac:surfaces} \\
      \hline
      $\Bigg.\Bigg. \sigma\in(0,1]$ & non-existent & & Theorem \ref{thm:trA} \\
      \hline \hline
      $\Bigg.\Bigg. \sigma\in(-1,0)$ & open problem & & \\
      \hline
      $\Bigg.\Bigg. \sigma=-1$ & $\frac{(a-b)^2}{(a\,b)^2}$ & O. Schn\"urer & \cite{os:surfaces} \\
      \hline
      $\Bigg.\Bigg. \sigma<-1$ & open problem & & \\
      \hline  
    \end{tabular}
  \end{center}
\end{corollary}

\section{Outlook}

`The search is over.' \\

At least, this is what had been our second choice for a title of this paper.
It sounded a bit too dramatic, so we decided to go for the first.
However, the MPF Ansatz followed by B. Andrews, O. Schn\"urer, F. Schulze, 
and many more, now seems to be exploited to the last drop.
Certainly that is for the Gauss curvature, $F^\sigma_0$,
the mean curvature, $F^\sigma_1$, 
and the norm of the second fundamental form, $F^\sigma_2$. 

But the search is not over.
In our opinion, MPF condition \eqref{IV} is sufficient, 
not necessary, and obviously far too restrictive for many normal velocities.
For example, in case of $F^\sigma_0$, we know that
no MPF exists for any power $\sigma<-2$, but we still have
convergence to a sphere at $\infty$ by C. Gerhardt \cite{cg:non}.
And there are no apparent reasons why convergence to a round point 
should stop for powers of the mean curvature greater than 6, 
or for powers of the norm of the second fundamental form greater than 11. 

We conjecture that we can exploit the MPF Ansatz further,
if we can find a way to weaken MPF condition \eqref{IV}.
Then we also might be able to recycle some of the vanishing functions
that are not MPF according to the present MPF Definition \ref{def:mpf}.

\bibliographystyle{amsplain} 

\begin{thebibliography}{10}


\bibitem{ba:contraction}
Ben Andrews, \emph{Contraction of convex hypersurfaces by their affine normal},
J.  Differential Geom. \textbf{43} (1996), no. 2, 207--230.

\bibitem{ba:gauss}
Ben Andrews, \emph{Gauss curvature flow: the fate of the rolling stones}, Invent. math. \textbf{138} (1999), 151–-161.

\bibitem{ba:motion}
Ben Andrews, \emph{Motion of hypersurface by Gauss curvature}, Pacific journal of Mathematics \textbf{195} (2000), no. 1, 1--34.

\bibitem{ac:surfaces}
Ben Andrews, Xuzhong Chen, \emph{Surfaces moving by powers of Gauss curvature},
{\tt arXiv:1111.4616 [math.DG]}.

\bibitem{ba:moving}
Ben Andrews, \emph{Moving surfaces by non-concave curvature functions},
Calc. Var. Partial Differential Equations \textbf{39} (2010),  no. 3-4, 649--657.

%


\bibitem{mf:on}
Martin Franzen, \emph{On maximum-principle functions for flows by powers of the Gauss curvature}, 
{\tt arXiv:1312.5107 [math.DG]}.

  
\bibitem{cg:curvature}
Claus Gerhardt, \emph{Curvature Problems}, Series in Geometry and Topology, vol. 39, International Press, Somerville, MA, 2006.

\bibitem{cg:non}
Claus Gerhardt, \emph{Non-scale-invariant inverse curvature flows}, Calculus of Variations and Partial Differential Equations
\textbf{49} (2014), no. 1-2, 471-489.

\bibitem{gh:flow}
Gerhard Huisken, \emph{Flow by mean curvature of convex surfaces into spheres}, J. Differential Geom. \textbf{20} (1984), no. 1, 237-266.

\bibitem{sh:analysis2}
Stefan Hildebrandt, \emph{Analysis 2}, Berlin/Heidelberg/New York (Springer Verlag) 2003.

\bibitem{ql:surfaces}
Qi-Rui Li, \emph{Surfaces expanding by the power of the Gauss curvature flow}, Proceedings of the American Mathematical Society, \textbf{138} (2010), no. 11, 4089-–4102.

\bibitem{os:surfacesA2}
Oliver~C. Schn\"urer, \emph{Surfaces contracting with speed $\vert A \vert^2$}, J. Differential Geom. \textbf{71} (2005), no. 3, 347-363.

\bibitem{os:surfaces}
Oliver~C. Schn\"urer, \emph{Surfaces expanding by the inverse Gau\ss\;curvature flow}, J. reine angew. Math. \textbf{600} (2006), 117-134.

\bibitem{os:alpbach}
Oliver~C. Schn\"urer, \emph{Geometric evolution equations}, Lecture Notes, 2007.

\bibitem{fs:convexity}
Felix Schulze, \emph{Convexity estimates for flows by powers of the mean curvature, appendix with O. C. Schn\"urer}, Ann. Scuol Norm. Sup Pisa Cl. Sci. \textbf{(5)}, Vol. V (2006), 261-277.

\end{thebibliography}
\def\weg#1{} \def\unterstrich{\underline{\rule{1ex}{0ex}}} \def\cprime{$'$}
  \def\cprime{$'$} \def\cprime{$'$} \def\cprime{$'$}
\providecommand{\bysame}{\leavevmode\hbox to3em{\hrulefill}\thinspace}
\providecommand{\MR}{\relax\ifhmode\unskip\space\fi MR }
\providecommand{\MRhref}[2]{%
  \href{http://www.ams.org/mathscinet-getitem?mr=#1}{#2}
}
\providecommand{\href}[2]{#2}

\end{document}